\documentclass[12pt,a4paper,reqno]{amsart}
\usepackage{amsmath,amssymb,graphics,epsfig,color,enumerate,mathrsfs}
\usepackage{dsfont}  \usepackage{verbatim}
\definecolor{refkey}{gray}{.75}
\definecolor{labelkey}{gray}{.5}

\newtheorem{Theorem}{Theorem}[section]

\newtheorem{TheoremA}{Theorem}
\newtheorem{Lemma}[Theorem]{Lemma}
\newtheorem{Proposition}[Theorem]{Proposition}
\newtheorem{Corollary}[Theorem]{Corollary}
\newtheorem{Remark}[Theorem]{Remark}

\newtheorem{Claim}[Theorem]{Claim}
\newtheorem{Definition}[Theorem]{Definition}
\newtheorem{Warning}{Warning}[section]

 \definecolor{darkgreen}{rgb}{0,0.4,0}

\definecolor{light}{gray}{0.9}

\newcommand{\cA}{\ensuremath{\mathcal A}}
\newcommand{\cB}{\ensuremath{\mathcal B}}
\newcommand{\cC}{\ensuremath{\mathcal C}}

\newcommand{\cG}{\ensuremath{\mathcal G}}
\newcommand{\cH}{\ensuremath{\mathcal H}}

\newcommand{\cK}{\ensuremath{\mathcal K}}
\newcommand{\cL}{\ensuremath{\mathcal L}}
\newcommand{\cM}{\ensuremath{\mathcal M}}
\newcommand{\cN}{\ensuremath{\mathcal N}}

\newcommand{\cP}{\ensuremath{\mathcal P}}

\newcommand{\cR}{\ensuremath{\mathcal R}}

\newcommand{\cV}{\ensuremath{\mathcal V}}
\newcommand{\cW}{\ensuremath{\mathcal W}}

\newcommand{\bbE}{{\ensuremath{\mathbb E}} }

\newcommand{\bbL}{{\ensuremath{\mathbb L}} }

\newcommand{\bbN}{{\ensuremath{\mathbb N}} }

\newcommand{\bbP}{{\ensuremath{\mathbb P}} }
\newcommand{\bbQ}{{\ensuremath{\mathbb Q}} }
\newcommand{\bbR}{{\ensuremath{\mathbb R}} }

\newcommand{\bbZ}{{\ensuremath{\mathbb Z}} }
\let\a=\alpha \let\b=\beta   \let\d=\delta  \let\e=\varepsilon
 \let\g=\gamma       \let\l=\lambda
      \let\o=\omega      
\let\r=\rho  \let\s=\sigma \let\t=\tau   
  \let\z=\zeta
   \let\G=\Gamma   
\let\O=\Omega      
\newcommand{\rosso}{\textcolor{black}}
\newcommand{\da}{\downarrow}

\newcommand{\toup}{\rightharpoonup}

\newcommand{\be}{\begin{equation}}
\newcommand{\en}{\end{equation}}

\author[A.~Faggionato]{Alessandra Faggionato}
\address{Alessandra Faggionato.
  Dipartimento di Matematica, Universit\`a di Roma `La Sapienza'
  P.le Aldo Moro 2, 00185 Roma, Italy}
\email{faggiona@mat.uniroma1.it}

\newcommand{\ra}{\rangle}
\newcommand{\la}{\langle}

\title[Stochastic homogenization   in amorphous media]{Stochastic homogenization   in amorphous media and  applications to exclusion processes}
\begin{document}

\begin{abstract}
We consider   random walks on marked simple  point processes   with symmetric jump rates and unbounded jump range. We prove   homogenization  properties   of   the associated  Markov generators.
 As an application, we derive the   hydrodynamic limit of the simple exclusion process given by  multiple random walks as above,    with  hard--core interaction, on a marked Poisson point process. The above results cover  Mott variable range hopping, which is a fundamental mechanism of phonon--induced  electron conduction in amorphous solids  as doped semiconductors. Our techniques, based on an  extension of  two-scale convergence,  can be adapted to other  models, as e.g. the random conductance model.

\smallskip

\noindent {\em Keywords}:
marked simple point process,  Palm distribution, random walks,   Mott variable range hopping, stochastic homogenization, two-scale convergence,  exclusion process, hydrodynamic limit.
   
\smallskip

\noindent{\em AMS 2010 Subject Classification}: 
60G55, % point processe
60H25, %	Random operators and equations 
60K37, %processes in random environemtn
	35B27.   %homogenization

\thanks{This work   has been  supported  by  PRIN
  20155PAWZB ``Large Scale Random Structures". }

\end{abstract}

\medskip

\maketitle

 \textcolor{blue}{\emph{This manuscript is and will remain unpublished. After the submission to a journal, in the non-short waiting time, I have further developed the methods presented here getting  new and stronger results.
Then it  became scientifically more reasonable to withdraw the submission and present directly the new results, in two separate papers for space reasons.  }}
 
 \textcolor{blue}{\emph{The homogenization results and the hydrodynamics results 
contained in this manuscript will be presented in two future works in a more general context.}}

%\today

%
%Da fare:
%\begin{itemize}
%\item Controllare la referenza \cite[Sect.~1.15]{MKM} in Section \ref{mspp}. guardare anche la versione inglese del libro di brezis e dettagliare referenza \cite{Br}  in Section \ref{mspp}.
%\item Definire la topologia di $\O$
%\item Controllare che la convergenza \eqref{scivolo} sia proprio vaguely
%\end{itemize}
%

\section{Introduction}
We consider stochastic jump dynamics in a random environment, not necessarly 
 with an underlying lattice structure, where jumps can be  arbitrarily long.   
A fundamental example comes from Mott variable range hopping (v.r.h.),  which is at the basis of 
electron transport   in disordered solids in the regime of strong Anderson localization, as doped semiconductors \cite{MD,POF,SE}.   Starting from a quantum modelization, in the regime of low density of impurities  one arrives to a classical model given by  a family  of non--interacting random walkers hopping on the sites of a marked simple point process (cf. \cite{AHL,ABS,FSS,MA}).  
The latter is given by a random subset $\{(x_i,E_i)\}\subset \bbR^d \times \bbR$ sampled as follows. The  locally finite subset $\{x_i\}\subset \bbR^d$ is given by a simple point process with  stationary and ergodic  (w.r.t. to spatial translations) law. For example, $\{x_i\}$ can be a Poisson point process on $\bbR^d$. Given $\{x_i\}$, to each point $x_i$ one associates independently   a random variable $E_i$ (called \emph{mark}) according to a fixed distribution $\nu$. 
%In doped semiconductors, the $x_i$'s correspond to the locations of the impurities, while $E_i$ is the fundamental energy of an electron with  quantum wavefunction   localized around $x_i$.
Given a realization $\{(x_i,E_i)\}$  of the environment, Mott v.r.h. can be described in terms of  a time--continuous random walk with state space $\{x_i\}$. The    probability  rate for a jump from $x_i$ to $x_j$ is then  given by $ c_{x_i,
x_j}(\o) = \exp \left\{
 -\g |x_i-x_j| -\beta (|E_i|+|E_j|+|E_i-E_j|)\right\} $,
 where $\b$ denotes the inverse temperature and $\g>0$ is a fixed constant.
The presence of long jumps and the special energy dependence in the jump rates is fundamental to explain the anomalous decay of conductivity in strongly disordered amorphous solids, which follows the so called Mott's law (cf. \cite{AHL,ABS,FM,FSS,MA} and references therein). In general, we will refer to Mott v.r.h. when the  jump rates have the form 
\be\label{germania}
  c_{x_i,
x_j}(\o) = \exp \left\{
 -\g |x_i-x_j| -u (E_i,E_j)\right\} \,,
\en
for some symmetric bounded function $u$.
%In inorganic doped semiconductors, the physically relevant distributions $\nu$ have the  form  $c\, |E|^{\alpha} dE$ on some interval $[-A,A]$, where $c$ is  the normalization constant and $\alpha$ is a nonnegative exponent. 
%According to Mott's law, in an isotropic medium the conductivity matrix $\sigma (\beta)$ can be approximated (at logarithmic scale) by  $ \exp 
% \bigl( - c\, \beta^{\frac{\alpha+1}{\alpha+1+d} }\bigr)\mathds{1}$, where $\mathds{1}$ denotes the identity matrix  and $c$ is a suitable  positive constant  with negligible temperature-dependence. Rigorous bounds in agreement with Mott's law have been derived in \cite{FM,FSS}. 

Stochastic jump dynamics, where hopping takes place  on marked simple point processes,  is relevant  e.g. also in population dynamics. If sites are given for example by a Poisson point process, then the medium is genuinely amorphous.  However, we include in our analysis also hybrid environments given by diluted lattices as e.g. site percolation. Indeed,  the random set $\{x_i\}$, where $x_i:=x+y_i$, $x$ is chosen with uniform probability among $[0,1]^d$ and $\{y_i\}$ is the realization of a site percolation on $\bbZ^d$, has stationary and ergodic law.

Of course, there are also other relevant models of stochastic jump dynamics with  jump rates having unbounded range, as for example the conductance model on $\bbZ^d$ where the random walk hops among sites of $\bbZ^d$ and the probability rate for a jump from $x$ to $y$ is given by a random number, called conductance \cite{FHS}. Another example is given by random walks on Delaunay triangulations \cite{Ro}.  We focus here on hopping  on marked simple point processes with jump rates not necessarily of the form \eqref{germania}, including the case of  random walks on Delaunay triangulations,  but our proofs and results can be adapted to other models as the random conductance model. 

The main part of our work is devoted to prove homogenization results for the random walk (cf. Theorems \ref{teo1} and \ref{teo2}).  We denote  by $\bbL_\o$ the generator of the random walk with environment $\o$ and by  $\bbL^\e_\o$  its version under an $\e$--parametrized  space rescaling. We consider the Poisson equation $\l u_\e - \bbL^\e_\o u_\e=f_\e$ with $\l>0$   and show the convergence of $u_\e $ to the solution $u$ of the effective equation $\l u - \nabla \cdot D \nabla u = f$ if $f_\e $ converges to $f$. Above $D$ is the effective diffusion matrix, having a variational characterization. We also prove the convergence of the associated  gradients,  energies and  semigroups.

The proof of Theorems \ref{teo1} and \ref{teo2} is based  on two--scale convergence, a notion introduced by G.~Nguetseng \cite{Nu}  and developed by G.~Allaire \cite{A}. In particular, our proof is inspired by the method developed in \cite{ZP} for random  differential operators on singular structures. Two--scale convergence has already been applied in random walks in random environment in \cite{F1,FHS,MP}.  Due to the jumps in  several directions, the standard  discrete gradients have  to  be replaced by \emph{amorphous gradients}. Roughly, given a random function  $v(\o)$,  its amorphous gradient keeps knowledge of all the differences $v(\t_x\o)-v(\t_y\o)$ as the sites $x,y$ vary among the sites of the marked simple point process, where $\t_z\o$ denotes 
 the  environment obtained from $\o$ by a translation along the vector $z$ (an analogous version holds for functions on $\bbR^d$). The  two--scale convergence is a genuine $L^2$--concept but  due to  the presence of infinite possible jumps 
a key technical obstruction to the analysis in \cite{ZP} appears. Even  when the gradient is square integrable, its contraction (obtained by a weighted averaging 
with weights given by the jump rates,
cf. Def.~\ref{artic}) is not necessarily square integrable.  We have been able to overcome this difficulty by means of  a cut-off procedure developed in Sections \ref{cut-off1} and \ref{cut-off2}.  We stress that the presence of long jumps  leads to new technical problems also in other spots when trying to adapt the strategy in \cite{ZP} to the present context.
  We point out that  in  \cite{FHS} (cf. Theorem 2.1 there)  the authors derive homogenization for the non--massive Poisson equation on a finite box with Dirichlet boundary conditions   for the random conductance model with possible unbounded jump range. Although dealing with arbitrary long jumps, the above obstruction does not emerge in \cite{FHS} since stronger conditions are assumed there (to derive  also  spectral homogenization).  Under these stronger assumptions, the availability of bounds in uniform norm allows to proceed without developing any cut-off procedure.
%  
%  \cite[Thm.~2.1]{FHS}, referred to the random conductance model,  
%  
%  
%  is somehow similar to Item (i) in Theorem \ref{teo1} (there the authors consider Dirichlet boundary conditions on a finite box). There are however two fundamental differences. 
%  Our homogenization results (the above  mentioned convergences concerning Poisson equation) hold for almost any environment, whatever  the choice of the known functions $f_\e, f$ with $f_\e$ converging to $f$. On the contrary,   in \cite[Thm.~2.1]{FHS} the class of environments for which one has homogenization depends on the functions $f_\e, f$.
% In addition, 
%    the method developed in \cite{FHS} is based on stronger additional assumptions,  concerning the existence of special paths with suitable  small resistance which implies between others  the Poincar\'e inequality.  Our method avoids this kind of technical assumptions   (cf. \cite{CF1} for isoperimetric and Poincar\'e inequalities  for Mott v.r.h.). 
We also point out that homogenization results for random walks on Delaunay triangulations have been obtained also in \cite{H} under the condition that  the diameters of the Voronoi cells  are uniformly bounded both  from below and from above (see Condition 1.2 in \cite{H}). Under this condition our analysis would be much simplified. 
    Finally, we mention also \cite{PZ} for other results on homogenization for non--local operators.

As an application of the homogenization results presented in Theorems \ref{teo1} and \ref{teo2}  we prove the hydrodynamic limit of the exclusion process obtained by taking multiple random walks as above with the addition of the hard-core interaction, when the sites $\{x_i\}$ are given by  a Poisson point process on $\bbR^d$. It is known (cf. e.g. \cite{F0,F1} and references therein) that the proof of the exclusion process with symmetric jumps rates in a random environment  can be obtained using  homogenization properties  of the Markov generator of a single random walk.
% (see also \cite{RSS} for recent progresses on  hydrodynamic limits in dynamical  random environments).
 We point out that, as a further application of Theorems \ref{teo1} and \ref{teo2} in interacting particle systems,  one could prove the hydrodynamic limit of  zero--range processes on marked simple  point processes similarly to what done in \cite{F2} using - between others - homogenization. This further application will be presented in a separate work. 

%Last but not least, homogenization plays a crucial role in proving the scaling limit of finite volume effective conductivity  in models of random conduction (see e.g. \cite{BP}). In a separate work we will consider homogenization on a finite  box  with  mixed Dirichlet--von Neumann 
%boundary conditions and  derive scaling properties for the effective conductivity of the Miller--Abrahams random resistor network, which is related to  

\bigskip

\emph{The arxiv version contains  supplementary proofs collected in   Appendix  \ref{aggiuntina}}

\section{Notation and setting}
In this section we fix our notation concerning point processes and state our main assumptions. 

\smallskip

We fix a  Polish space $S$ (e.g. $S\subset \bbR$)  and we 
 denote by  
   $\O$ the space of  locally finite subsets $\o \subset \bbR^d\times S$ such that for each $x\in \bbR^d$ there exists at most one element $s\in S$ with $(x,s)\in \o$.    
   %the map $\bbR^d\times S \ni  (x,s)\mapsto x \in \bbR^d$ is injective when restricted to $\o$ (i.e. for each $x
We write a generic $\o \in \O$ as 
 $\o= \{ ( x_i , s_i)\}$ ($s_i$ is called the  \emph{mark} at the point  $x_i$) and we set  $\hat\o :=\{x_i\}$. 
 % If $S$ has only one element, then the marked point process  on $\bbR^d$ reduces  to a standard point process on $\bbR^d$.
 We will identify the sets $\o= \{ ( x_i , s_i)\}$ and $\hat\o =\{x_i\}$ 
 with the counting measures $\sum _i  \d_{(x_i,s_i)}$ and $\sum_i \d_{x_i}$, respectively.  
 On $\O$ one defines a special metric $d$
 such that the following facts are equivalent: 
 (i)   a sequence  $(\o_n)$ converges to  $\o$ in $(\O, d)$, (ii) 
$\lim _{n\to \infty} \int _{\bbR^d\times S} f(x,s) d\o_n (x,s) = \int_{\bbR^d\times S}  f (x,s) d \o (x,s) \,,
$ for any bounded continuous function $f: \bbR^d  \times S\to \bbR$ vanishing outside a bounded set and (iii) 
$\lim_{n \to \infty} \o _n (A)=\o(A)$ for any bounded  Borel set  $A\subset \bbR^d \times S $ with $\o(\partial A)=0$
 (see \cite[App.~A2.6 and Sect.~7.1]{DV}). Moreover, the 
  $\s$--algebra $\cB(\O)$ of  Borel sets of $(\O,d)$ is generated by the sets $\{ \o (A)= k\}$ with $A$ and  $k$ varying respectively among the Borel sets of $\bbR^d\times S$  and the nonnegative integers. In addition, $(\O,d)$ is a separable metric space. Indeed, the above  distance $d$ is defined on the larger space $\cN$ of counting measures  $\mu=\sum _{i} k_i \d_{(x_i,s_i)}$,  where $k_i\in \bbN$  and $\{(x_i,s_i)\} $ is a locally finite subset of $\bbR^d\times S$, and one can prove that $(\cN,d)$ is a Polish space having $\O$ as Borel subset \cite[Cor.~7.1.IV, App.~A2.6.I]{DV}.
Finally,  
given $x\in \bbR^d$ we define the translation $\t_x:\O \to \O$ as 
\[ \t_x \o:= \{ ( x_i -x, s_i)\}
\; \text{ if  }\;\o= \{ ( x_i , s_i)\}\,.
\]

\medskip

We consider now a 
 \emph{marked simple point process}, which  is a measurable function from a probability space to  the measurable space $\left(\O,\cB(\O)\right)$. 
 We denote by $\cP$  its law  and by $\bbE[\cdot]$ the associated expectation. $\cP$ is therefore a probability measure on $\O$.  We assume that $\cP$ is stationary and ergodic w.r.t. translations.  Stationarity means  that $\cP (\t_x A)=A$ for any $A\in \cB( \O)$, while ergodicity means that  $\cP(A)\in\{0,1\}$
 for any  $A\in \cB( \O)$ such that 
 $\t_x A=A$. 
 Due  to our main assumptions stated below, $\cP$ 
 will have  finite positive  intensity $m$, i.e. 
  \begin{equation}\label{bimbi}
m:= \bbE\bigl[ |\hat \o \cap [0,1]^d|\bigr]\in (0, +\infty)\,.
\end{equation}
   As a consequence,  the Palm distribution $\cP_0$ associated to $\cP$ is well defined \cite[Chp.~12]{DV}.  Roughly, $\cP_0$ can be thought as $\cP$ conditioned to the event $\O_0$, where    \be 
    \O_0:=\{ \o \in \O\,:\, 0 \in \hat \o\}\,.
   \en
    $\cP_0$ is a probability measure with support inside  
    $\O_0$ and it can be  characterized by the identity 
\begin{equation}\label{marlena}
 \cP_0(A)= \frac{1}{m} \int _{\O}\cP(d\o)\int_{[0,1]^d}  d\hat{\o} (x) \mathds{1}_A(\t_x \o)\,, \qquad \forall A\subset \O_0 \text{ Borel}\,.
 \end{equation}
  The above identity  \eqref{marlena}  is a special case of the so--called 
 Campbell's formula (cf.~\cite[Eq.~(12.2.4)]{DV}): for any  nonnegative Borel function $f: \bbR^d\times \O \to[0,\infty) $ it holds
    % and for any function $f\in L^1( \bbR^d\times \O, dx \times \cP_0)$ it holds
 \begin{equation}\label{campanello}
 \int_{\bbR^d}dx  \int _{\O_0} \cP_0 ( d\o) f(x, \o) =\frac{1}{m} \int _{\O}\cP(d\o)\int_{\bbR^d}  d\hat{\o} (x) f(x, \t_x \o) \,.
 \end{equation}
  An alternative characterization of $\cP_0$ is  described in \cite[Section 1.2]{ZP}.

Below, we write $\bbE_0[\cdot]$ for the expectation w.r.t. $\cP_0$.
%Further properties of the Palm distribution will be discussed in Section \ref{mspp}.

\medskip
 
In addition to the marked simple point process with law $\cP$ we  fix  a nonnegative  Borel function
\[
\bbR^d\times \bbR^d \times \O \ni (x,y,\o) \mapsto c_{x,y} (\o)\in [0,+\infty)\,.
\]
The value of $ c_{x,y} (\o)$ will be relevant only when $x,y \in \hat \o$.

\smallskip

\medskip

%\begin
{\bf Assumptions.} 
We make the following assumptions:
\begin{itemize}
\item[(A1)] the law $\cP$ of the marked point process is stationary and ergodic w.r.t. spatial translations;
 \item[(A2)] $\cP$ has finite positive intensity, i.e.
 \begin{equation}\label{mom_palma0}
0< \bbE\bigl[ |\hat \o \cap [0,1]^d|  \bigr]<+\infty \,;
 \end{equation}
\item[(A3)] $\cP ( \o \in \O:  \t_x\o\not = \t_y \o   \; \forall x\not =y \text{ in }\hat \o )=1$;
\item[(A4)] the weights $c_{x,y}(\o)$ are symmetric and covariant, i.e. $c_{x,y}(\o)=c_{y,x}(\o)$  $\forall x,y\in \hat\o$   and 
$
 c_{x,y}(\o) = c_{x-a, y-a}( \t_a \o)$  $\forall x,y \in \hat \o$ and $\forall a \in \bbR^d$;
\item[(A5)]  it holds 
\begin{equation}\label{moma}
 \l_0  \in L^2(\cP_0)\,,\;\; % in L^1 \eqref{moma}
 \l_1 \in L^2 (\cP_0)\,,\;\;
 \l_2 \in L^1(\cP_0)\,,
 \end{equation} where    \begin{equation}\label{organic}
  \l_k(\o):=\int _{\bbR^d} d\hat \o (x) c_{0,x}(\o)|x|^k
 \end{equation}
 and  $|x|$ denotes the  norm of $x\in \bbR^d$;
 \item[(A6)] the function 
$
  F_*(\o) := \int d \hat \o (y) \int d \hat \o (z) c_{0,y} (\o) c_{y,z}(\o)
$
belongs to $L^1(\cP_0)$;
 \item[(A7)] the weights $c_{x,y}(\o)$ induce irreducibility for $\cP$--a.a. $\o\in \O$: for $\cP$--a.a. $\o\in \O$, given any $x,y \in \hat \o$ there exists a  path $x=x_0,x_1, \dots, x_{n-1}, x_n =y$ such that $x_i \in \hat \o$ and  $c_{x_i, x_{i+1}}(\o) >0$ for all $i=0,1, \dots, n-1$.
\end{itemize}

 \subsection{Comments on the  assumptions} \label{ranatina}
%Having that the marked simple point process is stationary,  the ergodicity of $\cP$ in Assumption (A1) is equivalent to the point--shift ergodicity of $\cP_0$, i.e. to the fact that $\cP_0(A)\in \{0,1\}$ whenever $A\subset \O_0$ is a Borel set such that $\t_x 
% \begin{equation}\label{autostrade}
% \cP_0( 
% \end{equation}

We recall that Mott v.r.h. corresponds to jump rates of the form \eqref{germania}.
\begin{Lemma}\label{re} The following holds:
\begin{itemize}
\item[(i)] Having that  $\cP$ is stationary, the ergodicity of $\cP$ is equivalent to the ergodicity of $\cP_0$ w.r.t. point--shifts, i.e. to the fact that $\cP_0(A)\in \{0,1\}$ for any Borel subset $A \subset \O_0$ such that $\o \in A$ if and only if $\t_x \o \in A$ for all $x \in \hat \o$.
\item[(ii)] The upper bound in Assumption (A2) is equivalent to the following fact: 
 $\bbE \bigl[ |\hat \o \cap U |  \bigr]<+\infty$ for any bounded Borel set $U \subset \bbR^d$. 
 \item[(iii)] Assumption (A3) is equivalent to the identity
  \begin{equation}\label{SS3}
\cP_0( \o \in \O_0\,:\, \t_x \o\not =\t_y \o \;\forall  x\not=y \text{ in } \hat\o) =1\,.
\end{equation}
\item[(iv)] Suppose that, for some function $h: \bbR_+\to\bbR_+$,  $c_{x,y}(\o) \leq h(|u-v|)$ for  any $u,v\in \bbZ^d$ and  any $x,y \in\hat \o$ with $x\in u+[0,1]^d$, $y\in v+[0,1]^d$.
   Then Assumption (A6) is implied by the bounds
 \be\label{spugna}
  \bbE\bigl[ |\hat \o \cap [0,1]^d | ^{3} \bigr]<+\infty\,, \qquad \sum_{u \in \bbZ^d}h (|u|)<+\infty\,.
 \en
\item[(v)] For Mott v.r.h. Assumptions (A4) and (A7) are always satisfied,  (A5) is equivalent to the bound $\bbE \bigl[ |\hat \o \cap [0,1]^d | ^{3} \bigr]<+\infty$, which implies (A6).

\item[(vi)] Assumption (A7) is equivalent to the fact that  the weights $c_{x,y}(\o)$ induce irreducibility for $\cP_0$--a.a. $\o\in \O_0$.

\end{itemize}
\end{Lemma}
We postpone the proof of Lemma \ref{re} to Appendix \ref{dimo}. We point out that the proof of the above Item (iii)  is based on Lemmas 1 and 2 in \cite{FSS}. The same arguments can be adapted to treat more general jumps rates.

% CONTROLLARE If $\cP$ is the $\nu$--randomization of $\hat \cP$, where  $\hat \cP$ is  the law of an ergodic  stationary simple point process on $\bbR^d$ with finite intensity  and $\nu$ is a  probability distribution on $\bbR$ not of the degenerate form $\d_a$,  then $\cP$ satisfies automatically Assumptions (A1), (A2), (A3). 

%For Mott v.r.h. condition \eqref{spugna} which implies (A6) reduces to  $\bbE\bigl[ |\hat \o \cap [0,1]^d | ^{3} \bigr]<+\infty$ since the jump rates have exponential decay. Mott v.r.h. automatically satisfies Assumption (A7).

  In the rest we will frequently used the symmetry and covariance of Assumption (A4) without explicit mention.
%%%%%%%%%%%%%%%%%%%%%%%%%%%%%%%%%%%
%
%
%   
%
%
%
%%%%%%%%%%%%%%%%%%%%%%%%%%%%%%%%%%%
\section{Main results}\label{MR}

In what follows,  given a measure $\cM$, we will denote by $\la \cdot, \cdot \ra _{\cM}$   the scalar product   in $L^2(\cM)$.

\subsection{The effective diffusive  matrix $D$}
\begin{Definition} We define the effective diffusive  matrix  as 
the $d\times d$ symmetric matrix $D$ such that
 \begin{equation}\label{def_D}
 a \cdot Da =\inf _{ f\in L^\infty(\cP_0) } \frac{1}{2}\int d\cP_0(\o)\int d\hat \o (x) c_{0,x}(\o) \left
 (a\cdot x - \nabla f (\o, x) 
\right)^2\,,
 \end{equation}
 where $\nabla f (\o, x) := f(\t_x \o) - f(\o)$. 
 \end{Definition}
 Above, and in what follows, we will denote by $a\cdot b$ the scalar product of the vectors $a$ and $b$.
\begin{Warning}\label{divido}
Since $D$ is symmetric, $D$ can be diagonalized. Let $ v_1,\dots, v_d$ be a basis of eigenvectors of $D$ and let  $\l_i$ be the eigenvalue of $v_i$.  
To simplify the notation, 
at cost of an orthogonal change of variables and without loss of generality, 
 we assume that 
\begin{align}
& \text{span}\{ e_1,e_2, \dots, e_{d_*}\}= \text{span}\{ v_i \,:\,\l_i>0\}\,,\label{lavita1}\\
& \text{span} \{ e_{d_*+1}, \dots, e_d\} = \text{span}\{ v_i \,:\,\l_i=0\}\label{lavita2}\,,
\end{align}
where $d-d_* $ is the dimension of the null eigenspace of $D$ and $e_1, \dots, e_d$ is the canonical basis of $\bbR^d$. In particular, $d_*\in \{0,1,\dots, d\}$.
Note that if $D$ is strictly positive, then $d_*=d$ and there is no need to change variables.
\end{Warning}

\subsection{The microscopic measure $\mu^\e_\o$}
Given  $\e>0$ and $\o \in \O$  we define  $\mu _\o^\e $ as the Radon measure on $\bbR^d$ 
 \begin{equation}\label{franci}
 \mu_\o^\e := \e^d \sum _{a\in \hat \o } \d_{ \e a }\,.
 \end{equation}
For 
 $\cP$--a.a. $\o $ the measure $ \mu _\o^\e $ converges vaguely to  $m dx $,  where $m$ is the intensity (cf. \eqref{bimbi}), i.e.  $\cP$--a.s. it holds
 \begin{equation}\label{scivolo}
  \lim_{\e\downarrow 0} \int d\mu_\o^\e  (x)   \varphi(x) = \int dx\,m \varphi(x) \; \qquad\forall \varphi \in C_c (\bbR^d)\,.
 \end{equation}
 The above convergence indeed follows from a stronger result which is at the basis of $2$--scale convergence:
\begin{Proposition}\label{prop_ergodico}
 Let  $g: \O_0\to \bbR$ be a Borel function with $\|g\|_{L^1(\cP_0)}<+\infty$. Then there exists a translation invariant   Borel subset $\cA[g]\subset \O$  such that $\cP(\cA[g])=1$ and such that,  for any $\o\in \cA[g]$ and any  $\varphi \in C_c (\bbR^d)$, it holds
\begin{equation}\label{limitone}
%\lim _{\e\da 0}  \e^{d} \int _{\bbR^d} d \hat{\o} (x)  \varphi (\e x ) g(\t_x \o ) =
\lim_{\e\da 0} \int  d  \mu_\o^\e  (x)  \varphi (x ) g(\t_{x/\e} \o )=
\int  dx\,m\varphi (x) \cdot \bbE_0[g]\,.
\end{equation}
\end{Proposition}
The above fact can be derived by Tempel'man's ergodic theorem for weighted means as discussed in Appendix \ref{dimo}.
%mammina
 %We refer to \cite{FV2} for  a self--contained proof.

 \smallskip
 
We recall the definition of weak and strong convergence for functions belonging to different functional spaces:
%%%%%%%%%%%%%%%%%%%%
\begin{Definition}\label{debole_forte} Fix $\o\in \O$ and a 
 family of $\e$--parametrized  functions $v_\e \in L^2( \mu^\e_\o)$. We say  that the family $\{v_\e\}$  \emph{converges weakly} to the function  $v\in L^2( m dx)$, and write  $v_\e \rightharpoonup v$, if
\begin{equation}\label{recinto}
\limsup  _{\e \da 0} \| v_\e\| _{L^2(\mu^\e_\o) } <+\infty
\end{equation}and 
\begin{equation}\label{deboluccio}
\lim _{\e \da 0} \int d  \mu^\e _\o (x)   v_\e (x) \varphi (x)= 
\int dx\,m  v(x) \varphi(x) \end{equation}
for all $\varphi \in C_c(\bbR^d)$. 
We  say that  the  family $\{v_\e\}$  \emph{converges strongly} to  $v\in L^2( m dx)$, and write $v_\e\to v$,  if  in addition to \eqref{recinto}  it holds
\begin{equation}\label{fortezza}
\lim _{\e \da 0} \int  d  \mu^\e _\o (x)   v_\e (x) g_\e (x)= 
\int dx\, m  v(x) g(x) \,,\end{equation}
for any family of functions $g_\e \in  L^2(\mu^\e_\o)$ weakly converging to $g\in L^2( m dx) $.
\end{Definition}

In general, when \eqref{recinto} is satisfied, one simply says that the family $\{v_\e\}$ is bounded.

\begin{Remark}\label{forte} One  can prove (cf.~\cite[Prop.~1.1]{Z}) that $v_\e \to v$ if and only if 
$v_\e \rightharpoonup v$ and $\lim _{\e\downarrow 0} \int_{\bbR^d}
 v_\e(x)^2 d \mu_\o^\e(x)=
\int _{\bbR^d} v(x)^2 m dx $.
\end{Remark}

\subsection{The microscopic measure $\nu^\e_\o$ and microscopic gradients}
We   define $\nu _\o ^\e $ as the Radon measure on $\bbR^d\times \bbR^d$ given by
\begin{equation}\label{nunu}
\begin{split}
 \nu_\o^\e
:&= \e^d  \int d  \hat\o (a)  \int d( \widehat {\t_a\o} )(z)  c_{a,a+z}( \o)  \d _{(\e a  ,z)} \\
&= \e ^d \int d \hat \o (a) \int d \hat \o (b) c_{a,b}(\o) \d_{(\e a, b-a)}\,.
%=  \e^d  d\int \hat\o (a)  \sum _{z \in \widehat {\t_a\o} } c_{0, z}(\t_a  \o)  \d _{(\e a  ,z)} 
\end{split}
 \end{equation}
%Note that  the total mass of $\nu_\o^\e$ is given by
%$
%\nu_\o^\e (\bbR^d)=\e^d \sum _{y \in \hat \o} 
%\sum _{z \in \widehat {\t_y\o} } c_{y,y+z}( \o)  =\e^d \sum _{y \in \hat \o} \l_0(\t_y \o)$, which is typically infinite. 

%%%%%%%%%%%%%%%%%%%%%%%%%%%
 %\begin{Definition}\label{ricola1}
 Given $\o \in \O$ and a real function $v$ whose domain contains  $\e \hat \o$,  we define the \emph{microscopic gradient} $\nabla_\e  v $  as the function 
 \begin{equation}\label{ricola}
 \nabla_\e v (x,z)= \frac{ v(x+\e z)- v(x)  }{\e}\,, \qquad x\in \e \hat \o, \; z \in \t_{x/\e} \hat{\o}\,.
 \end{equation}
 Note that if $v:\bbR^d \to \bbR$ is defined  $\mu^\e_\o$--a.e., then $\nabla_\e v$ is defined $\nu^\e_\o$--a.e.
 
 \smallskip

 We say that $v \in  H^1_{\o,\e}  $ if $v \in  L^2(\mu _\o ^\e) $ and $\nabla_\e  v  \in L^2( \nu_\o ^\e)$. Moreover, we endow the space  $H^1_{\o,\e} $ with the norm
 \[ 
 \|v\|_{H^1_{\o, \e} }:= \|v\| _{ L^2(\mu _\o ^\e) }+ \| \nabla_\e v \|_{  L^2(\nu_\o ^\e)}\,.
 \]
 Equivalently, we can identify $H^1_{\o,\e}$ with the subspace
 \[H_{\o,\e}:= \{ (v, \nabla_\e  v)\,:\, v \in L^2(\mu _\o ^\e) \,,\; \nabla_\e v \in L^2( \nu_\o ^\e)\}
 \]
 of the product Hilbert space $L^2(\mu _\o ^\e) \times L^2( \nu_\o ^\e)$. 
  \begin{Lemma}\label{dis}The space  $H_{\o,\e}$ is a closed subspace of $L^2( \mu _\o ^\e) \times L^2(\nu_\o ^\e)$. In particular,  $H_{\o,\e}$ and  $H^1_{\o,\e}$ are Hilbert spaces.
 \end{Lemma}
The proof is simple and given in Appendix \ref{dimo} for completeness.

\smallskip
 
 We introduce a notion of weak and strong convergence for microscopic gradients.

% \textcolor{green}{REMOVE:
% Consider the standard space $H^1(dx) $ given by functions $f$ in $L^2(dx)$ whose weak derivatives  belong to  $L^2(dx)$. Recall that $C^\infty _c(\bbR^d)$ forms a dense subset of $H^1(dx) $ (cf. \cite[Thm.~9.2]{Br}), in particular standard  gradients  $\nabla \varphi$ with $\varphi\in C^\infty _c(\bbR^d)$ can be used to  approximate  in $L^2(dx)$ the weak gradient $\nabla f$ when $f\in H^1(dx)$. }

 \begin{Definition}
 We introduce the space  $H^1_*(m dx)$ given by the functions $f \in L^2(m dx)$  such that the weak derivative $\partial_i f$ belongs to  $L^2(m dx)$  $\forall i\in \{1,\dots, d_*\}$, 
 endowed with the norm $\| f\| _{H^1_*( m dx)}:= \| f\|_{L^2(mdx)} + \sum_{i=1}^{d_*} \| \partial _i f\|_{L^2(m dx)} $.  
 Moreover, given $f\in H^1_*(m dx)$, we set
 \be
 \nabla_* f= (\partial_1 f, \dots, \partial _{d_*} f , 0,\dots, 0)\in L^2(m dx)^d\,.
 \en 
\end{Definition} 
We point out that $H^1_*(m dx)$ is an Hilbert space (it is enough to adapt the standard proof for $H^1(dx)$). Moreover, $C_c^\infty (\bbR^d)$ is dense in $H^1_*(m dx)$. Without  assumptions \eqref{lavita1} and \eqref{lavita2}, to define $H^1_*(mdx)$  one should consider the  weak derivatives along the vectors $\{v_i\,:\,\l_i>0\}$, which would make the notation heavier. 
 
 %%%%%%%%%%%%%%%%%%%%

\begin{Definition}\label{debole_forte_grad}   Fix $\o\in \O$ and a 
 family of $\e$--parametrized  functions $v_\e \in L^2( \mu^\e_\o)$. We say  that the family $\{\nabla_\e v_\e\}$  \emph{converges weakly} to the vector-valued function  $w$   belonging to the product space $ L^2( m dx)^d$  and with values in $\bbR^{d}$, and write  $\nabla_\e v_\e \rightharpoonup w$, if
\begin{equation}\label{recinto_grad}
\limsup  _{\e \da 0} \| \nabla_\e v_\e\| _{L^2(\nu^\e_\o) } <+\infty
\end{equation}and 
\begin{equation}\label{deboluccio_grad}
\lim _{\e \da 0} \frac{1}{2} \int   d \nu^\e _\o (x,z ) \nabla_\e v_\e (x,z)  \nabla_\e \varphi  (x,z )  = 
\int   dx\,m D w(x) \cdot \rosso{ \nabla_*} \varphi(x)   \end{equation}
for all $\varphi \in C^1_c(\bbR^d)$. 
We  say that   family $\{\nabla _\e v_\e\}$  \emph{converges strongly} to  $w$ \rosso{as above}, and write $\nabla_\e v_\e\to w$,  if  in addition to \eqref{recinto_grad}  it holds
\begin{equation}\label{fortezza_grad}
\lim _{\e \da 0}  \frac{1}{2} \int d  \nu^\e _\o (x,z)    \nabla_\e v_\e (x,z)  \nabla_\e g_\e (x,z )  = 
\int  dx\,m D w(x) \cdot \rosso{ \nabla_*} g(x) 
 \end{equation}
for any family of functions $g_\e \in  L^2(\mu^\e_\o)$ with $g_\e \toup g \in \rosso{L^2(m dx)}$ such that $g_\e \in H^1_{\o,\e}$ and \rosso{$g\in   H^1_* ( m dx)$}.
\end{Definition}
Denoting by $\varphi_\e$  the restriction of $\varphi$ to $\e \hat \o$,
due to Lemma \ref{ascoli}   in Section \ref{robot}, for $\cP$--a.a. $\o$  and in particular for all $\o \in \O_{\rm typ}$ defined below, any function  $\varphi\in C_c^1(\bbR^d)$ has the property that   $\varphi _\e \in H^1_{\o,\e}$. Trivially we have  $L^2(\mu^\e_\o)\ni \varphi_\e \to \varphi \in L^2( mdx )$. In particular, for such  environments $\o \in \O_{\rm typ}$,  if $\nabla _\e v_\e \to w$ then $\nabla_\e v_\e\toup w$.

\subsection{Difference operators}

We consider the operator $\bbL_{\o}^\e$ defined as 
\begin{equation}\
\bbL_{\o}^\e f(  \e a  ):= 
\e^{-2} \int  d\hat \o (y) c_{a,y}(\o) \bigl( f(\e y) - f (\e a) \bigr)\,, \qquad \e a \in \e \hat \o\,,
\end{equation}
for functions $f: \e \hat \o \to \bbR $ for which the series in the  r.h.s. is absolutely convergent for each $ a\in  \hat \o$. Note that  this property is fulfilled if  e.g. $f$ has compact support. Moreover, if $f,g$ have compact support, then the scalar product   $\la -\bbL _\o^\e f , g \ra_{\mu_\o^\e}$ in $L^2(\mu^\e_\o)$ is well defined (indeed, one deals only with finite sums) and it holds
 \begin{equation*}
 \begin{split}
 \la -\bbL_{\o}^\e f , g \ra_{\mu_\o^\e} & = 
 \frac{\e^{d-2}}{2} \int  \hat \o (da) \int  \hat \o (dy) c_{a,y}(\o)\bigl[ f(\e y) - f (\e a) \bigr]\bigl[ g(\e y) - g(\e a) \bigr]
 \\& =\frac{1}{2} \la \nabla _\e f , \nabla _\e g \ra _{\nu_\o ^\e}\,.
\end{split}
\end{equation*}
%where $\nu_\o^\e$ is the measure defined in \eqref{nunu}.
The above identity suggests a weak formulation of the equation
$-\bbL_\o^\e u+\l u =f$:
\begin{Definition}
Given $f\in L^2 (\mu_\o^\e)$  and $\l>0$,  a weak solution $u$ of the equation
\begin{equation}\label{strepitio}
-\bbL^\e_\o u+\l u =f
\end{equation}
is a function $u \in H^1_{\o,\e}$ such that 
\begin{equation}\label{salvoep}
\frac{1}{2} \la \nabla_\e v, \nabla_\e u \ra _{\nu_\o ^\e}+\l  \la v, u \ra _{\mu_\o^\e}= \la v, f \ra_{\mu_\o^\e}\qquad \forall v \in H^1_{\o,\e} \,.
\end{equation} 
\end{Definition}
By the  Lax--Milgram theorem \cite{Br}, given  $f\in L^2 ( \mu_\o^\e)$  the weak  solution $u $ of \eqref{strepitio} exists and is unique.

\smallskip

We now move to the effective equation, where $D$ denotes the  \emph{effective diffusion matrix} introduced in \eqref{def_D}.
\begin{Definition}
Given \rosso{$f\in L^2 (m dx)$} and $\l>0$, a weak solution $u$ of the equation
\begin{equation}\label{strepitioeff}
-\rosso{\nabla_* \cdot  D \nabla_* }u+\l u =f
\end{equation}
is a function $u \in \rosso{H^1_*( m dx) }$ such that 
\begin{equation}\label{delizia}
   \int   D \rosso{\nabla_*} v(x) \cdot  \rosso{\nabla_*} u (x) dx + \l \int   v(x) u(x) dx = \int  v(x) f(x)dx  \,, \qquad \forall v \in \rosso{H^1_*( m dx )}\,.
\end{equation} 
\end{Definition}
Again, by the   Lax--Milgram theorem, given  $f\in \rosso{L^2 (m dx)}$  the weak  solution $u $ of \eqref{strepitio} exists and is unique. \rosso{To check coercivity for the Lax--Milgram theorem, for $d_*\not =0$ one uses \eqref{lavita1} to get that  $\int   D \nabla_*  v(x) \cdot  \nabla_* u (x) dx\geq C \int \nabla_* v(x)\cdot
  \nabla_* u(x) dx $, where $C:=\min\{ \l_i\,:\, \l_i>0\} $.}

\smallskip

We can now state our first main results on homogenization:
\begin{TheoremA} \label{teo1} Let Assumptions (A1),...,\rosso{(A7)} be satisfied. Then
there exists a Borel  subset $\O_{\rm typ}\subset \O$, of so called \emph{typical environments}, fulfilling the following properties. $\O_{\rm typ}$ is translation invariant  and   $\cP(\O_{\rm typ})=1$. Moreover, given any   $\l>0$, $f_\e \in L^2(\mu^\e_\o)$ and  $f\in L^2( dx)$,  let  $u_\e$ and $u$ be defined  as the weak solutions,    respectively in $H^1_{\o, \e}$ and \rosso{$H^1_*(mdx)$}, of the equations 
\begin{align}
&-\bbL^\e_\o u_\e+\l u_\e =f_\e\,,\label{eq1}
\\
&
-\rosso{\nabla_*\cdot  D \nabla_*} u +\l u  =f\,.\label{eq2}
\end{align}
Then, for 
  any $\o\in \O_{\rm typ}$, we have:\begin{itemize}
\item[(i)] {\bf Convergence of solutions} (cf. Def. \ref{debole_forte}):
\begin{align}
f_\e \toup f \;\Longrightarrow \; u_\e \toup u \,,
\label{limite1}
\\
f_\e \to f  \;\Longrightarrow \; u_\e \to u\,.
\label{limite2}
\end{align}

\item[(ii)]  {\bf Convergence of flows} (cf. Def. \ref{debole_forte_grad}):
\begin{align}
f_\e \toup f \;\Longrightarrow \; \nabla_\e u_\e \toup \rosso{\nabla_* } u \,,
\label{limite3}
\\
f_\e \to f  \;\Longrightarrow \; \nabla_\e u_\e \to \rosso{ \nabla_*} u\,.
\label{limite4}
\end{align}

\item[(iii)] {\bf Convergence of energies}:
\begin{equation}\label{limite5}
f_\e \to f  \;\Longrightarrow \; \frac{1}{2}\la \nabla_\e u_\e , \nabla_\e u_\e \ra_{\nu_\o ^\e} \to \int dx\, m  \rosso{\nabla_*} u (x) \cdot D\rosso{ \nabla_*} u (x) \,. 
\end{equation}
\end{itemize}
\end{TheoremA}

\begin{Remark}\label{gelatino}
Let $\o \in \O_{\rm typ}$. Then it is trivial to check that, 
for any $f \in C_c(\bbR^d)$,   $L^2(\mu^\e _\o ) \ni f \to f \in L^2(mdx)$. By taking $f_\e:=f$ and using  \eqref{limite2}, we get that   $u_\e \to u$, where $u_\e$ and $u$ are defined as the weak solutions of \eqref{eq1} and \eqref{eq2}, respectively. \end{Remark}
\medskip

Given $\o \in \O_{\rm typ}$ it holds $\l_0(\t_z\o)<+\infty$ for any $z \in \hat \o$ (cf.  (S12) in Def.~\ref{budda} in Section \ref{topo}).  This allows to define,  up to a possible explosion time,   the random walk $(X_{\o, t}^{ \e})_{t\geq 0}$ on $\e \hat \o$ with probability rate for a jump from $\e y$ to $\e z$ given by $\e^{-2} c_{y,z} (\o )$. Since $\bbE[\l_0]<+\infty$ by (A5), as discussed in \cite[Sec.~3.3.1]{CFP}, for $\cP_0$--a.a. $\o\in \O_0$ explosion does not take place when the random walk starts in $0 \in \e \hat \o$.
Due to Lemma \ref{matteo} presented below, we conclude that for $\cP$--a.a. $\o$ explosion does not take place when the random walk $(X_{\o, t}^{ \e})_{t\geq 0}$  starts in any point $ \e y \in \e \hat \o$. From now on, we refine the definition of $\O_{\rm typ}$ including the property that for $\o \in \O_{\rm typ}$ the  above random walk   on $\e \hat \o$  is well defined and does not explode, whatever its initial point in $\e \hat \o$.

 We write $\bigl(P^\e _{\o ,t}\bigr)_{t \geq 0}$ for the $L^2(\mu^\e _\o)$--Markov semigroup associated to the random walk $(X_{\o, t}^{ \e})_{t\geq 0}$ on $\e \hat \o$. In particular, $P^\e _{\o ,t}=e^{t \bbL^\e _\o }$. Similarly we write $\bigl( P_t \bigr)_{t \geq 0} $ for the  Markov semigroup on $L^2( m dx)$ associated to the \rosso{(possibly degenerate)} Brownian motion on $\bbR^d$  with diffusion matrix \rosso{$2 D$} given by \eqref{def_D}. Note that, by Warning \ref{divido},  this Brownian motion is not degenerate in the first $d_*$ coordinates, while no motion is present in the remaining coordinates. In particular, writing $p_t(\cdot,\cdot )$ for the probability transition kernel of the Brownian motion on $\bbR^{d_*}$ with non--degenerate diffusion matrix $(2D_{i,j})_{1\leq i,j \leq d_*}$, it holds 
 \be\label{pesciolini}
 P_t f(x',x'') = \int _{\bbR^{d_*}} p_t ( x', y) f( y, x'' ) dy\,\qquad (x', x'')\in \bbR^{d_*} \times \bbR^{d-d_*}= \bbR^d\,. 
 \en 
%Then the following convergence of semigroups can be obtained as a byproduct of Theorem \ref{teo1} and the arguments  used in \cite{F1}:

%by the same arguments used in \cite{F1}:
\begin{TheoremA}\label{teo2}  Let Assumptions (A1),...,\rosso{(A7)} be satisfied. 
Take  $\o \in \O_{\rm typ}$ and $f \in C_c (\bbR^d)$. Then  it holds
\begin{equation}\label{marvel0}
 L^2(\mu^\e _\o ) \ni P^\e_{\o,t} f \to P_t f \in L^2(mdx)\,.
 \end{equation}
  For each $k \in \bbZ^d$ define the random variable  $N_k(\o)$ as $N_k(\o):= \hat \o ( k + [0,1)^d)$. Suppose that at least one of the following conditions is satisfied:
 \begin{itemize}
 \item[(i)] for $\cP$--a.a. $\o$ 
  $\exists C(\o)>0$ such that $\sup _{k \in \bbZ^d} N_k(\o) \leq C(\o)$;
  \item[(ii)]  $\bbE[N_0^2]<\infty$ and  there exists $C_0 \geq 0$, and there exists $\a >0$ if $d=1$, such that 
 \[
 \text{Cov}\,(N_k, N_k') \leq  
 \begin{cases} C_0 |k-k'| ^{-1} & \text{ for } d\geq 2 \\
 C_0|k-k'| ^{-1-\a} & \text{ for } d=1
 \end{cases}
 \]
  for any $k \not = k'$ in $\bbZ^d$.
 \end{itemize}
Then there exists a Borel set $\O_\sharp\subset \O $ with  $\cP( \O_\sharp)=1$  such that for any $ \o \in \O_\sharp \cap \O_{\rm typ}$ and any $f \in C_c(\bbR^d)$ it holds:
\begin{align}
%& L^2(\mu^\e _\o ) \ni P^\e_{\o,t} f \to P_t f \in L^2(mdx)\,, \label{marvel0}\\
& \lim_{\e \da 0} \int \bigl | P^\e_{\o,t} f(x) - P_t f (x) \bigr|^2 d \mu^\e_\o (x)=0\,,\label{marvel1}\\
& \lim_{\e \da 0} \int \bigl | P^\e_{\o,t} f(x) - P_t f (x) \bigr| d \mu^\e_\o (x)=0\,.\label{marvel2}
\end{align}
\end{TheoremA}
Note that any  marked Poisson point process satisfies the above  Condition (ii),  while any marked  diluted lattice satisfies Condition (i).

Theorem \ref{teo2} is obtained from Theorem \ref{teo1} by the strategy developed in  \cite{F1}. The proof is given in Section \ref{dim_teo2}.

\subsection{Hydrodynamic limit  of simple exclusion processes}

We discuss here how the above homogenization results together with the strategy in \cite{F0,F1} can be applied to obtain the hydrodynamic limit of exclusion processes \cite{L} on marked simple point processes.  The method used to get Theorem \ref{teo3} is rather general and can used to treat other exclusion processes.

Given $\o \in \O$ we consider the exclusion process on $\hat\o$ with formal generator 
\be\label{ringo}
\cL_\o f(\eta ) = \sum_{x\in \hat \o} \sum_{y\in \hat \o}c_{x,y}(\o) \eta_x ( 1- \eta_y) \left[ f( \eta ^{x,y})- f(\eta)\right]\,,\qquad \eta \in \{0,1\}^{\hat \o}\,,
\en
where 
\[
\eta^{x,y}_z=\begin{cases}
\eta_y & \text{ if } z=x\,,\\
\eta_x & \text{ if } z=y\,,\\
\eta_z & \text{ otherwise}\,.
\end{cases}
\]
Given a probability measure $\mathfrak{m}$ on $\{0,1\}^{\hat \o}$,
we write $\bbP_{\o,\mathfrak{m}}$ for the  above exclusion process with initial distribution $\mathfrak{m}$ and we write $\eta(t)$ for the particle configuration at time $t$.

\begin{TheoremA}\label{teo3} 
 Consider an ergodic stationary marked simple point process $\cP$ on $\bbR^d$ such that the law of its  spatial support $\hat\o$ 
 is   a Poisson point process with intensity $m>0$.  Take jump rates satisfying assumptions (A4)--\rosso{(A7)} and such that, for $\cP$--a.a. $\o$, 
 $c_{x,y}(\o) \leq g(|x-y|)$   for any $x,y \in \hat \o$, where $g(|x|)$ is a fixed bounded function in $L^1(dx)$ (for example take  Mott v.r.h.). 
 Then for $\cP$--a.a. $\o$ the above 
exclusion process is well defined for any initial distribution and  the following holds:

Let $\rho_0: \bbR^d \to [0,1]$ be a Borel function and let $\{\mathfrak{m}_\e\}$ be an $\e$--parametrized  family of probability measures on $\{0,1\}^{\hat{\o}}$ such that, for all $\d>0$ and all $\varphi\in C_c(\bbR^d)$, it holds
\be\label{marzolino}
\lim_{\e\da 0} \mathfrak{m}_\e\Big( \Big| \e^d \sum_{x \in \hat \o} \varphi (\e x) \eta_x -\int _{\bbR^d} \varphi(x) \rho_0(x) dx \Big|>\e\Big)=0\,.
\en
Then for all $t>0$,  $\varphi \in C_c(\bbR^d)$ and $ \d>0$ we have 
\be\label{pasqualino}
\lim_{\e\da 0} \bbP_{  \o,\mathfrak{m}_\e } \Big(\Big|  \e^d \sum_{x \in \hat \o} \varphi (\e x) \eta_x(\e^{-2} t) - \int _{\bbR^d} \varphi(x) \rho(x,t) dx\Big| >\d 
\Big)=0\,,
\en
where $\rho:\bbR^d\times[0,\infty)\to \bbR$  \rosso{is given by $\rho(x,t)= P_t \rho_0 (x)$  (cf.~\eqref{pesciolini}). $\rho$}
solves the heat equation \rosso{$\partial_t \rho = \nabla_*\cdot ( D \nabla_* \rho)$ for $t>0$} with boundary condition $\rho_0$ at $t=0$. \rosso{Above,}  $D$ is the effective diffusion matrix given by \eqref{def_D}. 
\end{TheoremA}
The proof of the above theorem is given in Section \ref{ida}.

\begin{Remark} 
The conditions that $\hat \o$ is a Poisson point process and that   $\cP$--a.s.
 $c_{x,y}(\o) \leq g(|x-y|)$, with $g$ as in Theorem \ref{teo3}, are used only  to derive Lemma \ref{igro1}. The rest of the derivation of Theorem \ref{teo3} relies on Lemma \ref{igro1} and not on these specific conditions, which could be replaced by something else, more suited to  other models.
\end{Remark}

 \section{Preliminary facts on the Palm distribution $\cP_0$}\label{corteggiare}

We recall a fact frequently used in the rest (see \cite[Lemma 1--(ii)]{FSS}): given a  translation invariant  subset $A\subset \O$,
it holds then $\cP(A)=1 $ if and only if $ \cP_0(A)=1$.

\begin{Lemma}\label{matteo}
Given a Borel subset $A\subset \O_0$, the following facts are equivalent:
\begin{itemize}
\item[(i)] $\cP_0(A)=1$;
\item[(ii)] $\cP\left(\o \in \O\,:\, \t_x \o \in A \; \forall x \in \hat\o\right)=1$;
\item[(iii)]  $\cP_0\left(\o \in \O_0\,:\, \t_x \o \in A \; \forall x \in \hat\o\right)=1$.
\end{itemize}
%Moreover, given a Borel set $B\subset \O$ with $\cP(B)=1$ it holds 
%\begin{equation}\label{ecologico}
%\cP(\o\in \O\,: \, \t_x \o \in B \; \forall x \in \bbR^d)=1
%\end{equation}
\end{Lemma}
\begin{proof}  By identity \eqref{marlena},   (ii) implies (i). If   (i) holds, by  Campbell's identity  \eqref{campanello} with
 $ f(x,\o):=(2\ell)^{-d} \mathds{1}_{[-\ell,\ell]^d}  (x) \mathds{1}_A (\o )$ and $\ell>0$, we get 
 \begin{equation}\label{fabio} 1=\cP_0(A)=  \frac{1}{m (2\ell)^d} \int _\O\cP(d\o) \int_{[-\ell, \ell]^d}d \hat{\o}(x) \mathds{1}_A (\t_x \o)\,.
 \end{equation}
Hence, we obtain 
 \begin{equation}\label{fabio57} 
 \int _\O\cP(d\o) \int_{[-\ell, \ell]^d} d\hat{\o}(x) \left(1-\mathds{1}_A (\t_x \o)\right)=0\qquad \forall \ell>0 \,,
 \end{equation}
 which implies  (ii). This proves that (i) implies (ii).
  
Finally, the equivalence between (ii) and (iii) follows from \cite[Lemma 1--(ii)]{FSS}.
% Since very simple, we give the proof for completeness.
% Since $0\in \hat \o$ for all $\o \in \O_0$, we have $\tilde A:=\{ \o \in \O_0\,:\, \t_x \o \in A \; \forall x \in \hat\o\} \subset A$ and therefore (iii) implies (i). Suppose now that (i) is satisfied, equivalently that (ii) is satisfied, i.e. $\cP(\hat A)=1$ where $\hat A:=\{\o \in \O\,:\, \t_x \o \in A \; \forall x \in \hat{\o}\}$.  We want to prove that  (iii) holds. Trivially,  if $\o \in \hat A$, then $\t_y \o \in \tilde A$ for any $y \in \hat \o$. 
% Due to this observation  and (ii), we conclude that 
%  $\cP(\o \in \O\,:\, \t_x \o \in \tilde A \; \forall x \in \hat\o)=1$. Due to the already proved equivalence between (i) and (ii) (applied now with $\tilde A$ instead of $A$), we conclude that $\cP_0(\tilde A)=1$, which corresponds to Item (iii). 
 \end{proof}
 
 \begin{Lemma}\label{eleonora}
 Let $f\in L^1(\cP_0)$. Let $B:= \{ \o\in \O\,:\, |f(\t_x \o  )|<+\infty \; \forall x \in \hat \o \}$. Then $B$ is translation invariant, $\cP(B)=1$ and $\cP_0(B)=1$.
  \end{Lemma}
 \begin{proof}
 We define $A:=\{ \o \in \O_0\,:\, |f(\o)| <+\infty \}$. Since $f\in L^1(\cP_0)$, we have $\cP_0(A)=1$.   By applying 
Lemma \ref{matteo} we get that $\cP(B)=1$ and $\cP_0(B)=1$. The translation invariance of $B$ follows immediately from the definition of $B$.
 \end{proof}
 
In what follows we will use the following properties of the Palm distribution $\cP_0$ obtained by extending \cite[Lemma 1--(i)]{FSS}:
\begin{Lemma}\label{lucertola1} %\cite[Lemma 1--(i)]{FSS}  
Let $k:\O_0\times \O_0 \to \bbR$ be a Borel function such that (i)  at least one of the functions   $\int d \hat{\o} (x) |k(\o, \t_x \o) | $ and $\int d\hat{\o}(x) |k(\t_x\o,\o)|$ is  in $L^1(\cP_0)$, or (ii) $k(\o,\o')\geq 0$. Then
\begin{equation}\label{prugna1}
\int d \cP_0(\o) \int d\hat{\o}(x) k(\o, \t_x \o) = \int d \cP_0(\o) \int d \hat{\o}(x) k (\t_x\o, \o)\,.
\end{equation}
\end{Lemma}
%%%%%%%%%%%%%%%%%%%%%%%%%%%%%
\begin{proof} Case  (i) with both functions in $L^1(\cP_0)$ corresponds to   \cite[Lemma 1--(i)]{FSS}. We now consider case (ii).
Given $n\in \bbN$ we define $k_n(\o,\o')$ as 
\[ k_n (\o,\o'):=
\begin{cases} 
k (\o,\o') \wedge n & \text{ if $ \o' = \t_x \o $ for some $x$ with $|x|\leq n$}\,,\\
0 & \text{ otherwise}\,.
\end{cases}
\]
Note that, given $\o\in \O_0$ and $x\in \hat \o$, it holds  $0 \leq  k_n(\o, \t_x \o) \nearrow   k(\o,\t_x\o)$ and $0\leq  k_n(\t_x \o,  \o) \nearrow  k(\t_x \o,\o)$. 
Hence, by monotone convergence, to get \eqref{prugna1} it is enough to prove the same identity with $k_n$ instead of $k$.
To this aim we observe that,  due to Assumption (A3) and Lemma \ref{re} (cf. \eqref{SS3}), for $\cP_0$--a.a. $\o$ we have $k_n(\o,\t_x  \o)=0=k(\t_x \o,\o)$ if $x \in \hat \o$ and $|x|>n$.  Hence, 
we can bound both $ \bbE_0\bigl[ \int d \hat{\o} (x) k_n(\o, \t_x \o) \bigr] $ and $\bbE_0\bigl[ \int d\hat{\o} (x) k_n(\t_x \o, \o) \bigr]$ by $ n  \bbE_0\bigl[ \int d\hat{\o} (x) \mathds{1}_{\{ |x| \leq n\}}  \bigr] <+\infty$. At this point,   \eqref{prugna1} with $k_n$ instead of $k$ follows from \eqref{prugna1} proved in  case (i) under the condition that both the functions considered in case (i) belong to $L^1(\cP_0)$.
This concludes the proof of \eqref{prugna1} in case (ii). We can now prove the thesis in case (i) in full generality. Indeed, by \eqref{prugna1} proved in case (ii), for any function $k(\o,\o')$ we have that $\int d \hat{\o} (x) |k(\o, \t_x \o) | =\int d \hat{\o}(x) |k(\t_x\o,\o)|$. Hence, if one of these two  integrals is finite, then both are finite since equal.
\end{proof} 
%
%
% 
%\begin{Remark}\label{salutare}
%Item (i) in Lemma \ref{lucertola1} can be weaken by requiring that at least one of the functions
% $\int d \hat{\o} (x) |k(\o, \t_x \o) | $ and $\int d \hat{\o}(x) |k(\t_x\o,\o)|$ is in $L^1(\cP_0)$. Indeed, by applying Item (ii) of the same lemma, we get automatically that also the other function is in $L^1(\cP_0)$.  
%\end{Remark}
%
%
%%%%%%%%%%%%%%%%%%%%%%%
\section{Space of square integrable forms}\label{sec_quadro}
We define $\nu$ as the Radon measure on $\O \times \bbR^d$ such that 
\begin{equation}\label{labirinto}
 \int d \nu (\o, z) g (\o, z) = \int  d \cP_0 (\o)\int d  \hat \o (z) c_{0,z}(\o) g( \o, z) 
 \end{equation} 
 for any nonnegative Borel function $g(\o,z)$. We point out that, by Assumption (A5), $\nu$ has finite total mass:
$ \nu(\O\times \bbR^d )=\bbE_0[\l_0]<+\infty$.
  Elements of  $L^2 ( \nu)$
 are called \emph{square integrable forms}.

 \smallskip
 
 Given a  function $u:\O_0\to \bbR$  we define the   function $\nabla u: \O \times \bbR^d \to \bbR$ as 
 \begin{equation}\label{cantone}
 \nabla u (\o, z):= u (\t_z \o)-u (\o)\,.
 \end{equation}
 Note that if $u,f :\O_0\to \bbR$  are such that  $u=f$  $\cP_0$--a.s., then $\nabla u= \nabla f $ $\nu$--a.s. Indeed,  by Lemma \ref{matteo}, setting $A:=\{ \o\in \O_0\,:\, u(\o)= f(\o)\}$ and $\tilde A :=\{ \o \in \O_0\,:\, u(\t_z \o) = f(\t_z \o) \; \forall z \in \hat \o\}$, we have that $\cP_0(A)=1$ and therefore   $\cP_0(\tilde A)=1$,  thus implying that $\nabla u= \nabla f $ $\nu$--a.s.. In particular, 
 if $u$ is defined  $\cP_0$--a.s., then $\nabla u$ is well defined $\nu$--a.s.
 
  If $u $ is bounded and measurable (i.e. Borel), then $\nabla u \in L^2(\nu)$.
The subspace of \emph{potential forms} $L^2_{\rm pot} (\nu)$ is defined as   the following closure in $L^2(\nu)$:
 \[ L^2_{\rm pot} (\nu) :=\overline{ \{ \nabla u\,:\, u \text{ is  bounded and measurable} \}}\,.
 \]
 The subspace of \emph{solenoidal forms} $L^2_{\rm sol} (\nu)$ is defined as the orthogonal complement of $L^2_{\rm pot} (\nu)$ in $L^2(\nu)$.
  
    \subsection{The subspace $H^1_{\rm env} $}\label{h1omega}
 We define 
 \be \label{anita} H^1 _{\rm env} :=\{   u \in L^2(\cP_0)\,:\,  \nabla u \in L^2(\nu) \}\,.
 \en We endow $ H^1 _{\rm env} $ with  the norm
 $\| u\|_{H^1 _{\rm env}} :=  \|u\|_{L^2(\cP_0)} +\|\nabla u \|_{L^2(\nu)}$.
 It is convenient to introduce also the space 
  \[
H_{\rm env} := \{   (u, \nabla u)\,:\, u  \in L^2(\cP_0) \text{ and } \nabla u \in L^2(\nu) \}\subset  L^2(\cP_0) \times L^2(\nu)\]
 endowed with  the norm $\|(u,\nabla u )\|_{H_{\rm env} }:=  \|u\|_{L^2(\cP_0)} +\|\nabla u \|_{L^2(\nu)}$. In particular,  $ H^1_{\rm env} $ and $H_{\rm env} $ are isomorphic   spaces.
% \begin{Lemma} Let $u \in L^2(\cP_0)$ be such that $\nabla u \in L^2(\nu)$. Then $\nabla u \in L^2_{\rm pot}(\nu)$. \end{Lemma}
 
  %%%%%%%%%%%%%%%%%%%%
 \begin{Lemma}
 The space $H_{\rm env} $ is a closed subspace of  $L^2(\cP_0) \times L^2(\nu)$, hence $H _{\rm env} $ and $H^1_{\rm env} $ are  Hilbert spaces.
 \end{Lemma}
 \begin{proof}
 Suppose that $(u_n, \nabla u_n)$ converges to $(u,g)$ in $L^2(\cP_0) \times L^2(\nu)$. At cost to extract a subsequence, there
  exists a Borel set $A \subset \O_0$ with $\cP_0(A)=1$ such that the following holds for any $\o \in A$:  $u_n(\o) \to u(\o)$ and 
 $u_n(\t_z \o)- u_n (\o)  \to g (\o, z) $ for all $z \in \hat \o$ with $c_{0,z}(\o)>0$.
  By  Lemma \ref{matteo}, since  $u_n(\o) \to u(\o)$ for all $\o \in A$,  we conclude that $u_n (\t_z \o) \to u(\t_z \o)$ for all $z\in \hat{\o}$ for  $\cP_0$--a.a. $\o$. Hence it must be $g(\o,z)= u (\t_z \o)- u (\o)  $  for $\cP_0$--a.a. $\o$ and for all $z \in \hat \o$ with $c_{0,z}(\o)>0$. This proves that $\nabla u \in L^2(\nu)$ and that $(u_n, \nabla u_n)\to (u, \nabla u)$. 
  \end{proof}
  
  We fix some simple notation which will be useful also later. Given $M>0$ and $a \in \bbR$,  we define $[a]_M$ as 
  \begin{equation}\label{taglio}
  [a]_M= M \mathds{1}_{\{a>M\}} +a \mathds{1}_{\{ |a|\leq M \}} -M \mathds{1} _{\{a<-M\}}\,.
  \end{equation}
%  \begin{cases} 
%  M & \text{ if } a \geq M\,,\\
%  a& \text{ if } |a| \leq M\,,\\
%  -M & \text{ if } a\leq -M\,.
%  \end{cases}
%  \end{equation}
Given $a\geq b$, it holds  $a-b \geq [a]_M - [b]_M\geq 0$. Hence, for any $a, b\in \bbR$, it holds
%  This implies that $|a-b-([a]_M- [b]_M)|\leq  a-b= |a-b|$. As a consequence we get 
   \begin{align}
&   \left |  [a]_M - [b]_M   \right |\leq\left |a-b\right |\,,  \label{r101}\\
&  \left | [ a-b]- [[a]_M - [b]_M ]  \right| \leq \left|a-b\right|
   \,.\label{r102}
  \end{align}

   %%%%%%%%%%%%%%%%%%%%%%%%%%%%%%%%%%%%%%%%
  \begin{Lemma}\label{denso}
  The subspace 
  $\{ (u, \nabla u) \,: u  \text{ is bounded and measurable} \}$  is a dense subspace of $H_{\rm env} $.
  \end{Lemma} 
  %%%%%%%%%%%%%%%%%%%%%%%%%%%%%%%%%%%%%%%%
  \begin{proof} If  $u$ is  bounded and measurable, then  $u\in H^1_{\rm env} $ (as $\nu$ has finite total mass).
  Let us take now a generic $u\in H^1(\cP_0) $ and show that $[u]_M\to u$ in $ H^1_{\rm env}  $.  Since $|u-[u]_M |\leq |u|$, by dominated convergence we have that $\| u - u_M \|_{L^2(\cP_0)}\to 0$. Due to \eqref{r102}, $|\nabla u - \nabla [u]_M| \leq |\nabla u|$.  Hence, by dominated convergence, we get  that $\| \nabla u - \nabla [u]_M\|_{L^2(\nu)} \to 0$. 
  \end{proof}
  
 %\club Gradient, potential forms $L^2_{\rm pot}(\nu)$, solenoindal forms $L^2_{\rm sol} (\nu)$.

\subsection{Divergence}\label{div_omega}

\begin{Definition}\label{def_div}
Given a square integrable form $v\in L^2(\nu)$  we define its divergence  ${\rm div} \, v \in L^1(\cP_0)$ as 
\begin{equation}\label{emma}
{\rm div}\, v(\o)= \int d \hat{\o} (z) c_{0,z}(\o) ( v(\o,z)-  v(\t_z \o, -z) )\,.
\end{equation}
\end{Definition}

By applying Lemma \ref{lucertola1} with $k(\o, \t_z\o) := c_{0,z}(\o) |v( \o, z)| $ (such a $k$ exists  by (A3) and \eqref{SS3}), Schwarz inequality and \eqref{moma}, one gets for any $v\in L^2(\nu)$    that 
\begin{equation}\label{fuoco1}
\begin{split}
&\int d\cP_0(\o)
\int  d\hat{\o} (z) c_{0,z}(\o) \bigl(| v(\o,z)|+|  v(\t_z \o, -z)| \bigr)\\
& =
2  \|v\|_{L^1(\nu )} \leq 2 \bbE_0[\l_0] ^{1/2} \|v\|_{L^2(\nu)} < +\infty\,.
\end{split}
\end{equation}
In particular, the definition of divergence  is well posed and  the map $L^2(\nu) \ni v\mapsto {\rm div} v \in L^1(\cP_0)$ is continuous.

\smallskip

%\club By applying Lemma \ref{lucertola1} with $k(\o, \t_x \o) := c_{0,x}(\o) v( \o, x)  u (\t_x\o) 
%$ and therefore  with $k(\t_x \o,\o)= c_{0,x} (\o) v(\t_x \o,-x) u (\o) $ (all is well defined due to (A3) and \eqref{SS3}), one easily  gets the following:
\begin{Lemma}\label{ponte} For any  $v \in L^2(\nu)$ and any  bounded and measurable function $u:\O_0 \to \bbR$, it holds
\begin{equation}\label{italia}
\int  d \cP_0(\o)   {\rm div} \,v(\o)  u (\o)= - \int d \nu(\o, z) v( \o, z) \nabla u (\o, z) \,.
\end{equation}
 \end{Lemma}
 \begin{proof} %Let us first take a  bounded measurable function $u$.
  By definition we have 
\begin{equation}\label{udine1}
\begin{split}
\int d\nu(\o, z)  v( \o, z) \nabla u (\o, z)  = \int d\nu(\o, z) v( \o, z) \bigl( u (\t_z\o)-u(\o) \bigr) \,.
\end{split}
\end{equation}
 We apply now  Lemma \ref{lucertola1}  to $k(\o, \o')$ such that  
$k(\o, \t_z \o) = c_{0,z}(\o) v( \o, z)  u (\t_z\o) 
$. Note that the definition is well posed due to (A3) and \eqref{SS3}.  Due to \eqref{fuoco1}  and since $u$ is bounded,  the conditions of  Lemma \ref{lucertola1}  are satisfied.   Since $k(\t_z \o,\o)= c_{0,z} (\o) v(\t_z \o,-z) u (\o) $,
 \eqref{prugna1} becomes 
 \begin{equation}\label{udine3}
 \begin{split}
&  \int _{\O }\cP_0(d\o)  \int _{\bbR^d}\hat{\o}(dz) c_{0,z}(\o) v( \o, z)  u (\t_z\o) \\
& = \int _{\O }\cP_0(d\o)  \int _{\bbR^d}\hat{\o}(dz)  c_{0,z} (\o) v(\t_z \o,-z) u (\o)\,.
\end{split}
\end{equation}
As  a consequence of \eqref{udine1} and \eqref{udine3} we get that \eqref{italia} is satisfied.
% for any $u$ bounded and measurable if $g$ is defined as the r.h.s. of \eqref{friuli}. To conclude it is enough to apply the density result given by Lemma \ref{denso} (since $g\in L^2(\nu)$ and $v$ is bounded, both the  l.h.s. and the r.h.s.  of \eqref{italia} are continuous functional  of $u \in H^1(\cP_0)$).
\end{proof}

Trivially, the above result implies the following:
\begin{Corollary}\label{grazioso} Given a square integrable form $v\in L^2(\nu)$, we have that  $v\in L^2_{\rm sol}(\nu)$ if and only if ${\rm div}\, v=0$ $\cP_0$--a.s.
\end{Corollary}

We recall that, since $\cP$ is ergodic, then $\cP_0$ is ergodic w.r.t. point--shifts (cf. Lemma \ref{re}--(i)). As a consequence, $u={\rm constant} $ $\cP_0$--a.s.   if 
 $u: \O_0 \to \bbR$ is a Borel function such that  for $\cP_0$--a.a. $\o$  it holds $u(\o)= u(\t_x \o)$ for all $x\in \hat{\o}$. %
%\rosso{\club e' un fatto generale...}
%
%%%%%%%%%%%%%%
%\begin{proof}
%Let $A=\{\o \in \O_0\,:\, u(\o)= u(\t_x \o) \;\forall x\in \hat\o\}$
%and let  $\tilde A:=\{\o \in \O\,:\, \t_x \o \in A \; \forall  x\in \hat \o\}$.
%  By Lemma \ref{matteo} we know  that $\cP(\tilde{A})=1$. Given $\o \in \O$ we define $v(\o):=\mathds{1}_{\tilde{A}}(\o) u( \t_x \o)$ where $x$ is a generic point in $\hat\o$. By the definition of $A$ and $\tilde A$, $v$ is a well defined functions and $v$ is left invariant by translation, i.e. $v(\t_x \o)=v(\o)$ for all $x\in \bbR^d$. At cost to take a cut-off (i.e. to replace $u$ with $[u]_M$ and afterwards take the limit $M\to \infty$) we can suppose $v$ to be bounded. Then by the Ergodic theorem \cite[Prop.~12.2.II]{DV2}, $v(\o)= \bbE[v]$ for $\cP$--a.a. $\o$. This implies that $\cP(\o\in \O\,:\, u(\t_x \o)= \bbE(v) \; \forall x\in \hat\o)=1$, and therefore (by Lemma \ref{matteo}) $\cP_0(\o \in \O_0\,:\, u(\o)= \bbE(v) )=1$.
%  \end{proof}

\smallskip

Using Assumption (A7) we get:
\begin{Lemma}\label{garibaldi_100}
If $\nabla u =0$  $\nu$--a.e.,  then   $u={\rm constant} $ $\cP_0$--a.s.
\end{Lemma}
\begin{proof}
We define 
\begin{align*}
 A:&=\{ \o \in \O_0\,:\, u(\t_z \o)=u(\o) \;\; \forall z \in \hat \o \text{ with } c_{0,z}(\o)>0\}\,,\\
 \tilde A:& = \{ \o \in \O_0\,:\, \t_z \o \in A \;\; \forall z \in \hat \o\}\\
 & = \{ \o \in \O_0\,:\, u(\t_y \o)=u(\t_z \o) \; \forall y,z \in \hat \o \text{ with } c_{z,y}(\o)>0\}\,.
\end{align*}
The property that $\nabla u =0$  $\nu$--a.e. is equivalent to $\cP_0(A)=1$. By Lemma \ref{matteo} we get that $\cP_0( \tilde A)=1$. Due to assumption (A7) and Lemma \ref{re}--(vi), we have that $u(\t_y \o)=u(\o)$ for all $y\in \hat \o$ if $\o \in \tilde A$.  Since $\cP_0$ is ergodic w.r.t. point-shifts, we conclude that  $u={\rm constant} $ $\cP_0$--a.s.
\end{proof}
%
%\begin{Remark}\label{garibaldi}
%Due to Assumption (A7), given $u$ with $\nabla u =0$  $\nu$--a.s.,      it holds $u(\o)= u(\t_x \o)$ for all $x\in \hat{\o}$, for $\cP_0$--a.a. $\o$. Due to the ergodicity of $\cP_0$  we conclude that $u={\rm constant} $ $\cP_0$--a.s.  if $\nabla u =0$  $\nu$--a.s.
%\end{Remark}
%

The proof of the following lemma is similar to  the proof of  \cite[Lemma 2.5]{ZP}.
For completeness  we have provided it in 
%mammina 
%\cite[App.~B]{FV2}. % versione sottomessa
 Appendix \ref{aggiuntina}. % versione arxiv
 Recall \eqref{anita}. 
\begin{Lemma}\label{minerale} Let $\z\in L^2(\cP_0)$ be orthogonal to  all functions 
$ g \in L^2(\cP_0) $ with $g = {\rm div} (\nabla u)$ for some  $ u \in H^1_{\rm env}$.
 Then $\z\in H^1_{\rm env}$ and 
$\nabla \z=0$ in $L^2(\nu)$.
\end{Lemma}

By combining Lemma \ref{garibaldi_100} and Lemma \ref{minerale} we get:
\begin{Lemma}\label{santanna} 
 The  functions $g\in L^2(\cP_0)$ of the form $g ={\rm div}\, v $ with $v\in L^2(\nu)$ are dense in $\{w \in L^2(\cP_0)\,:\, \bbE_0[w]=0\}$.
\end{Lemma}
%%%%%%%%%%%%%%%%%%%%%
\begin{proof}
 Lemma \ref{ponte} implies that  $\bbE_0[g]=0$ if  $g={\rm div}\,v $,  $v\in L^2(\nu)$.  
 Suppose the density fails. Then there exists $\z\in  L^2(\cP_0)$ different from zero  with $ \bbE_0[\z]=0$ and such that $\bbE_0 [ \z g]=0$ for any $g \in L^2(\cP_0)$ of the form $g= {\rm div} v$ with $v\in L^2(\nu)$. 
By  Lemma \ref{minerale},  we know that $\z\in H^1_{\rm env}$ and $\nabla \z=0$ $\nu$--a.s. 
By Lemma \ref{garibaldi_100} we get that $\z$ is constant $\cP_0$--a.s.  Since $\bbE_0[\z]=0$ it must be $\z=0$ $\cP_0$--a.s., which is absurd.
\end{proof}
%
%\begin{Definition}
%We write $\nabla^{(\ell)}$,  ${\rm div}^{(\ell)}$ and $H^1(\cP_0, \nu ^{(\ell)})$  for $\nabla$, ${\rm div}$ and $H^1(\nu)$, respectively,  where $c_{0,z}$ is replaced by $c_{0,z}\mathds{1}_{|z|\leq \ell}$.
%\end{Definition}
%\rosso{\loz  Note that all assumptions used until now remain valide with $c_{0,z}^{\ell}$}
%
%\begin{Lemma}\label{solare107} 
%Suppose that the  pair $(\cP_0, \nu)$ is 2--connected. Then the set 
%\begin{equation}\label{ungheria}
%\cup_{\ell=1}^\infty  \Big \{ g \in L^2(\cP_0)\,:\, g = {\rm div}^{(\ell)}  (\nabla u) \text{ for some } u \in H^1(\cP_0, \nu^{(\ell) } ) \Big \}
%\end{equation}
%is dense in $\{w \in L^2(\cP_0)\,:\, \bbE_0[w]=0\}$.
%\end{Lemma}
%\begin{proof}
%The proof is by contradiction. Suppose the opposite. Then there exists $\z\in  L^2(\cP_0)$ different form zero  with $ \bbE_0[\z]=0$ and such that $\bbE_0 [ \z w]=0$ for any $w $ in \eqref{ungheria}. By Lemma \ref{minerale} applied with rates  $c_{0,z}\mathds{1}_{|z|\leq \ell}$ we get that $\nabla^{(\ell)} \z=0$ $\nu^{\ell}$--a.s. Hence, $\cP_0(A_\ell)=1$ where
%\[ A_\ell:= \{\o\,:\, \z(\t_z\o)= \z(\o) \; \forall z \text{ with } c_{0,z}(\o) >0 \text{ and }|z|\leq \ell\}\,.\]
%Since $\cP_0(\cap _{\ell =1}^\infty A_\ell)=1$ we conclude that $\nabla \z =0$ $\nu$--a.s. By the $(\cP_0, \nu)$ 2--connectedness, we conclude that $\z={\rm constant}$ $\cP_0$ a.s. Since $\bbE[\z]=0$ it must be $\z=0$, which is absurd.
%\end{proof}
%%%%%%%%%%%%%%%%%%%%

%%%%%%%%%%%%%%%%%%%%%%%%%%%%%%%%%%%%%
%%%%%%%%%%%%%%%%%%%%%%%%%%%%%%%%%%%%%

\section{The   diffusion matrix $D$ and the quadratic form $q$}\label{sec_diff_matrix}

Since $\l_2 \in L^1(\cP_0)$ (see Assumption (A5)), given $a\in \bbR^d$   the form 
\begin{equation}\label{def_ua}
u_a(\o,z):= a\cdot z
\end{equation}
 is square integrable, i.e. it belongs to $L^2(\nu)$. 
 We note  that the symmetric diffusion matrix $D$  defined in \eqref{def_D} satisfies, for any $a\in \bbR^d$,
 \begin{equation}\label{giallo}
 \begin{split}
q(a):= a \cdot Da  %&= \inf _{ f\in L^\infty(\cP_0) }\int d\cP_0(\o)\int d\hat \o (x) c_{0,x}(\o) \left (a\cdot x - \nabla f (\o, x)  \right)^2\\
&= \inf_{ v\in L^2 _{\rm pot}(\nu) } \frac{1}{2} \int d\nu(\o, x) \left(u_a(x)+v(\o,x) \right)^2\\
& =  \inf_{ v\in L^2 _{\rm pot}(\nu) } \frac{1}{2}\| u_a+v \|^2_{L^2(\nu)}=\frac{1}{2} \| u_a+v ^a \|^2_{L^2(\nu)}\,,
\end{split}
 \end{equation}
 where  $v^a=-\Pi u_a$ and 
 $\Pi: L^2(\nu) \to L^2_{\rm pot}(\nu)$ denotes the orthogonal projection of $L^2(\nu)$ on $L^2_{\rm pot}(\nu)$.   Note that, as a consequence, the map $\bbR^d \ni a \mapsto v^a\in L^2_{\rm pot}(\nu)$ is linear. 
Moreover, $v^a$ is characterized by the property
\begin{equation}\label{jung}
v^a \in  L^2_{\rm pot}(\nu)\,, \qquad v^a+u_a\in L^2_{\rm sol}(\nu)\,.
\end{equation}
Hence we can write
 \begin{equation}\label{solitario}
 a\cdot Da=\frac{1}{2} \|  u_a+v^a \|_{L^2(\nu)}^2= \frac{1}{2}
 \la u_a, u_a + v^a \ra _{\nu}\,.
 \end{equation}
 The above identity can be rewritten as
 \begin{equation}\label{razzetto}
 a\cdot D a =\frac{1}{2} \int d\nu (\o, z) a\cdot z \bigl(a\cdot z + v^a( \o, z) \bigr)\,.
 \end{equation}
% Note that in the last identity we have used that $v^a $ is potential while $u_a + v^a$ is solenoidal.
Since the two symmetric bilinear forms $(a,b) \mapsto a\cdot Db$ and \[
(a,b)\mapsto  \frac{1}{2}\int d\nu (\o, z) a\cdot z \bigl(b\cdot z + v^{b}( \o, z) \bigr)=\frac{1}{2}  \int d\nu \bigl (u_a+ v^{a}\bigr)
 \bigl( u_b + v^{b} \bigr) \] coincide on  diagonal terms by \eqref{razzetto}, we conclude that 
 \begin{equation}\label{solare789}
 Da =\frac{1}{2}  \int d\nu (\o, z)  z \bigl(a\cdot z + v^a( \o, z) \bigr)\qquad \forall a \in \bbR^d\,.
\end{equation}

Let us come back to the 
 quadratic form $q$ on $\bbR^d$ defined in \eqref{giallo}. By \eqref{giallo}
 its  kernel  ${\rm Ker}(q)$  is given by 
\begin{equation}\label{rosone}
{\rm Ker}(q):=\{a\in \bbR^d\,:\, q(a)=0\}=\{ a\in \bbR^d\,:\, u_a \in L^2_{\rm pot}(\nu)\}\,.
\end{equation}

\begin{Lemma}\label{rock}
It holds
\begin{equation}\label{jazz}
{\rm Ker}(q)^\perp=\Big\{  \int  d \nu (\o,z) b(\o, z) z  \,:\, b\in L^2_{\rm sol} (\nu) \Big\}\,. 
\end{equation}
\end{Lemma}
Note that, since $\l_2 \in L^1(\cP_0)$ by  (A5),   the integral in the r.h.s. of \eqref{jazz} is well defined. The above lemma corresponds to \cite[Prop.~5.1]{ZP}.
\begin{proof}
Let $b\in  L^2_{\rm sol}(\nu) $ and $\eta_b:=\int  d \nu (\o,z)  b(\o, z) z $. Then, given $a\in \bbR^d$, 
$a\cdot \eta_b =\la u_a , b \ra_{\nu}$. By \eqref{rosone},  $a \in  {\rm Ker}(q)$ if and only if $u_a\in L^2_{\rm pot}(\nu)= L^2_{\rm sol }(\nu)^\perp $.
Therefore $a \in  {\rm Ker}(q)$ if and only if  $a\cdot \eta_b =0$ for any $b \in L^2_{\rm sol}(\nu)$. \end{proof}
%Up to now we have never used Assumption (A8). When $D>0$, we have ${\rm Ker}(q)=0$. Hence, as a byproduct of Lemma \ref{rock} and (A8), we get:
%\begin{Corollary}\label{samba}
%The family of vectors  $\bigl\{\o \mapsto \int  d \nu (\o,z) b(\o, z) z  \,:\, b\in L^2_{\rm sol} (\nu) \bigr\}$ is dense in $\bbR^d$.
%\end{Corollary} 

\rosso{
Due to Lemma \ref{rock} and Warning \ref{divido} we have:
\begin{Corollary}\label{jack}
$\text{Span}\{ e_1, \dots, e_{d_*}\}= \bigl\{  \int  d \nu (\o,z) b(\o, z) z  \,:\, b\in L^2_{\rm sol} (\nu) \bigr\}$.
\end{Corollary}
}
%%%%%%%%%%%%%%%%%%%%%%%%%%%%%%%%%%%%%%%%%%%%%%%%%%%%%%

%%%%%%%%%%%%%%%%%%%%%%%%%%%%%%%%%%%
\section{The contraction $b(\o,z)\mapsto \hat b(\o) $ and the set  $\cA_1[b]$}\label{sec_cinghia}

\begin{Definition}\label{artic}  Let $b(\o,z): \O _0 \times \bbR^d\to \bbR$ be a Borel function  with $\|b\|_{L^1(\nu)}<+\infty$. We define the Borel function  $c_b:\O_0 \to [0,+\infty]$ as
\begin{equation}\label{magno}
c_b (\o):=  \int d \hat{\o}(z) c_{0,z}(\o)  |b(\o,z)|\,, \end{equation}
 the Borel function $\hat b: \O_0 \to \bbR$ as
\begin{equation}\label{zuppa}
\hat b(\o):= 
\begin{cases}
\int d \hat{\o}(z) c_{0,z}(\o) b(\o,z) & \text{ if } c_b (\o) <+\infty\,,\\
0 & \text{ if } c_b (\o) = +\infty\,,
\end{cases}
\end{equation}
and the Borel set $\cA_1[b]:= \{ \o\in \O\,:\, c_{b } ( \t _z \o) <+\infty\; \forall z\in \hat \o\}$.
\end{Definition}

\begin{Lemma}\label{cavallo}   Let $b(\o,z): \O _0 \times \bbR^d\to \bbR$ be a Borel function  with $\|b\|_{L^1(\nu)}<+\infty$. 
 Then 
 \begin{itemize}
 \item[(i)] 
$\| \hat b\|_{L^1(\cP_0)} \leq  \| b\|_{L^1(\nu)}= \| c_b \|_{L^1(\cP_0)} $ and  $\bbE_0[\hat b]= \nu(b)$;
\item[(ii)]
given $\o \in  \cA_1[b]$ and 
   $\varphi \in C_c(\bbR^d)$, it holds\begin{equation}\label{rino}
\int     d\mu_\o^\e (x) \varphi ( x) \hat b (\t_{x/\e} \o) = \int  d \nu_\o^\e (x, z) \varphi(x) b( \t_{x/\e} \o, z) 
\end{equation}
(the  series in the l.h.s. and in the r.h.s.  are absolutely convergent);
\item[(iii)]    $\cP( \cA_1[b])=\cP_0( \cA_1[b])=1$ and 
$ \cA_1[b]$ is translation invariant. 
\end{itemize}
\end{Lemma}
%%%%%%%%%%%%%%%%%%%%%%%%%%%
\begin{proof}  It is trivial to check Item (i) and Item (ii).
%By Schwarz' inequality we can bound
%\begin{multline*}
%\int d \cP_0 (\o) | h (\o) | \leq \int d \cP_0 (\o) \int d \hat{\o}(dz) c_{0,z} (\o) | f(\o,z) |\\
%=
%\la 1, |f | \ra_{L^2(\nu)} \leq \| 1\|_{L^2(\nu)}    \| f\| _{L^2(\nu)} = \bbE_0[f]^{1/2} \|f\|_{L^2(\nu)}.
%\end{multline*}
%This proves \eqref{nicky}. The identities  $\bbE_0[h]= \nu(f)$   follows easily from the definitions.
%
%Let us prove Item (ii). Since the integral in the l.h.s. of \eqref{rino} is a finite sum, it is absolutely convergent. Let us  prove that the integral in the r.h.s. of \eqref{rino} corresponds to an absolutely convergent series. %Let $L>0$ such that 
%We have 
%\begin{equation*}
%%\begin{split}
%\int  d \nu_\o^\e (x, z) | \varphi(x)  b( \t_{x/\e} \o, z) |
%%= \e^d \sum_{a \in \hat \o} \sum _{z\in \widehat{\t_a \o}} c_{0,z} ( \t_a \o)  | \varphi(\e x)  b( \t_{x} \o, z) |\\ &
%%= \e^d \sum_{x\in \hat \o} | \varphi(\e x)  | \sum _{z\in \widehat{\t_x \o}} c_{0,z} ( \t_x \o) | b( \t_{x} \o, z) |
%=   \int d\mu^\e_\o(x)  | \varphi(x)  | c_b (\t_{x/\e} \o) \,.
%%\end{split}
%\end{equation*}
%Since $\o \in  \cA_1[b]$, 
% the last term is a finite sum of (finite) positive numbers, thus implying that
% the r.h.s. of \eqref{rino} is an absolutely convergent series.  As a consequence we can arbitrarily arrange the terms in the series, without changing the final value. Hence, as above,  we get \eqref{rino}.
  We move to Item (iii). We have 
%If $\|b\|_{L^2(\nu)}<+\infty$, using Schwarz inequality we get 
$\bbE_0[ c_b] = \| b\| _{L^1(\nu)} <\infty$. This implies that 
  $\cP_0(\{\o\,:\, c_{b} (\o) <+\infty\})=1$ and  therefore $\cP( \cA_1[b])=\cP_0( \cA_1[b])=1$  by Lemma \ref{matteo}. The last property of  $\cA_1[b]$ follows immediately from  the definition.
 \end{proof}

We point out that, since $\nu $ has finite mass, $L^2(\nu) \subset L^1(\nu)$ and therefore Lemma \ref{cavallo} can be applied to $b$ with $\|b \|_{L^2(\nu) }<+\infty$.

%
%
%
%
%%%%%%%%%%%%%%%%%%%%%%%%%%%%%%%%%%%%

\section{The transformation $b(\o,z)\mapsto \tilde b(\o,z) $}\label{hermione}
\begin{Definition}\label{ometto}
Given  a Borel function  $b :\O_0\times \bbR^d\to \bbR$ we set 
% $\tilde b :\O_0\times \bbR^d$  as 
\begin{equation}
\tilde b (\o, z):=
\begin{cases}  b (\t_{z} \o, -z)  & \text{ if } z\in \hat \o\,,\\
0 & \text{ otherwise}\,.
\end{cases}
\end{equation}
\end{Definition}
%%%%%%%%%%%%%%%%%%%%%%%%%%%%%
%
%
%
%
%%%%%%%%%%%%%%%%%%%%%%%%%%%%%%
\begin{Lemma}\label{gattonaZZ} Given  a Borel function  $b :\O_0\times \bbR^d\to \bbR$,
it  holds  $\tilde{\tilde b}(\o,z)=b(\o,z)$ if $z\in \hat \o$. If 
 $b \in L^2(\nu)$, then 
 $ \| b \| _{L^2(\nu)} =  \| \tilde b \| _{L^2(\nu)}$ and ${\rm div} \,\tilde b= - {\rm div}\, b$. \end{Lemma}
 \begin{proof}  Let  $b \in L^2(\nu)$. 
 We apply Lemma \ref{lucertola1} with  $k$ such that $k(\o, \t_z \o ) :=
c_{0,z} (\o) b (\o, z )^2$  if $\o \in \O_0$ and $z\in \hat \o$ (cf. Assumption (A3) and \eqref{SS3}).
If $z \in \hat  \o$ we have  $k(\t_z\o  , \o) =  c_{0,-z} (\t_z \o) b ( \t_z \o, -z)^2 = c_{0,z}(\o) \tilde b (\o, z)^2$.
 Then Lemma \ref{lucertola1} implies  that  $ \| b \| _{L^2(\nu)} =  \| \tilde b \| _{L^2(\nu)}$. The  other identities follow from the definitions.
\end{proof}

% kuka  \blu{ \spade Ho controllato che Lemma \ref{gattonaZ} vale anche per \eqref{wow1}}

%\begin{Definition}\label{roger} Given a function $h: \O_0\to \bbR\cup \{+\infty\}$ we define $\cA_*[h]=\{ \o \in \O_0\,:\, h(\t_z \o) <+\infty \; \forall z\in \hat\o\}$ (in particular, the set $\cA_1[b]$ in Def. \ref{artic} equals $ \cA_*[ c_b]$ for $b:\O_0 \times \bbR^d \to \bbR$).
%\end{Definition} 
%\begin{Remark}\label{federer} If $h: \O_0\to \bbR\cup \{+\infty\}$  satisfies $\bbE_0[h]<+\infty$, then $\cP_0( \{\o\in \O_0\,:\, h(\o) <+\infty )=1$ and therefore, by Lemma \ref{matteo}, 
% $\cP_0( \cA_*[h])=1$. 
%\end{Remark}
%%%%%%%%%%
Recall the set $\cA[g]\subset \O$ introduced in Prop. \ref{prop_ergodico} and recall Definition \ref{artic}.

 %%%%% modifiche      %%%%%%%%%

 \begin{Lemma}\label{gattonaZ} $ \,\,$\\
(i) 
 Let $b: \O_0 \times \bbR^d \to [0, +\infty]$ and  $\varphi, \psi:\bbR^d \to [0,+\infty] $  be
  Borel functions. Then, for each $\o \in \O$,  it holds
  \begin{equation}\label{micio2}
 \int d \nu^\e_\o (x, z) \varphi ( x) \psi (  x+\e z ) b(\t_{x/\e}\o, z)=
  \int d \nu^\e_\o (x, z) \psi (x) \varphi (  x+\e z  ) \tilde{b}(\t_{x/\e}\o, z)\,.
 \end{equation}
 (ii)  Let $b: \O_0 \times \bbR^d\to \bbR$ be a Borel function with $\| b \|_{L^1( \nu )}<+\infty$ and take $ \o \in \cA_1[b]\cap \cA_1[\tilde b]$.  
   Given functions  $\varphi, \psi:\bbR^d \to \bbR $   such that  at least one between $\varphi, \psi $ has  compact support  and the other is bounded,  identity \eqref{micio2} is still valid.
   Given now $\varphi $ with compact support and  $\psi$ bounded, it holds
 \begin{align}
&  \int d \nu^\e_\o (x, z) \nabla_\e  \varphi   ( x,z) \psi (   x+\e z ) b(\t_{x/\e}\o, z) \nonumber\\
 & \qquad \qquad \qquad  =
  - \int d \nu^\e_\o (x, z) 
  \nabla _\e \varphi  (x,z) 
  \psi (x) \tilde{b}(\t_{x/\e}\o, z)\,.\label{micio3}
   \end{align}
%Both the l.h.s. and the r.h.s. of \eqref{micio2}, \eqref{micio2} and \eqref{micio33} are finite for $\e \leq \e_{b}$.
Moreover,  the above integrals  in  \eqref{micio2}, \eqref{micio3} (under the hypothesis of this  Item (ii)) 
 correspond to absolutely convergent  series and are therefore well defined. 
 
\end{Lemma}
%%%%%%%%%%%%
%%%%%%%%%%%%
\begin{proof} 
 We check \eqref{micio2} for Item (ii) (the proof for Item (i) uses similar computations).  Since $c_{a,a'}(\o)= c_{a',a}(\o)$ and $b(\t_a \o, a'-a)= \tilde b(\t_{a'} \o, a-a')$ for all $a,a'\in \hat \o$, we can write  
 \begin{equation}\label{tom1}
 \begin{split}
 & \int d \nu^\e_\o (x, z) \varphi ( x) \psi (   x+\e z ) b(\t_{x/\e}\o, z)\\ 
% & =\e^d \sum_{x\in \hat \o}\sum _{z \in \widehat{\t_x \o}} c_{x,x+z}(\o) \varphi (\e x ) \psi (  \e  x+\e z ) b(\t_x\o, z)\\
 &= \e^d \sum_{a\in \hat \o} \sum_{a'\in  \hat \o } c_{a,a'}(\o) \varphi (\e a ) \psi (\e a') b( \t_a \o, a'-a)
 \\
 &= \e^d  \sum_{a'\in  \hat \o } \sum_{a\in \hat \o} c_{a',a}(\o)\psi (\e a')   \varphi (\e a ) \tilde b(\t_{a'} \o, a-a')
 \\
  &= \int d \nu^\e_\o (x, z) \psi ( x) \varphi ( x+ \e z  ) \tilde{b}(\t_{x/\e}\o, z)\,.
  \end{split}
 \end{equation}
 Since we deal with  infinite sums,  the above rearrangements have to be justified.  
 %If $\varphi, \psi, b\geq 0$, then all rearrengements are legal.  
  We recall  that $\varphi$ has compact support and $\psi $ is bounded, or viceversa.
 The same computations as above hold when taking  the modulus of all involved functions. To conclude that the series are absolutely convergent  we first suppose that  $\varphi$ has compact support and $\psi $ is bounded. Then  we can bound 
 \be\label{tom2}
  \int d \nu^\e_\o (x, z)| \varphi ( x)| | b(\t_{x/\e}\o, z)| \leq  \int d \mu^\e_\o (x) | \varphi  ( x)| c_b (\t_{x/\e}\o)\,.
  \en
  Since  $\o \in \cA_1[b]$   the integral in the r.h.s. corresponds to a finite sum of finite terms, hence the r.h.s of \eqref{tom2} is finite and all  the rearrangements in \eqref{tom1} are legal (recall that $\psi$ is bounded). 
  If $\psi$ has compact support and $\varphi$ is bounded, we do similar computations for $ \int d \nu^\e_\o (x, z) |\psi ( x)| |\tilde{b}(\t_{x/\e}\o, z)|$ and use that  $\o \in \cA_1[\tilde b]$  
  
%    \blu{\spade When dealing with  \eqref{wow1}  we need to take $a,a' \in \hat \o \cap (\e^{-1} S)$ in \eqref{tom1}. Then in \eqref{tom2} one has to take  $ \mu^\e_{\o, \bbR^d}$ and the conclusion is the same.}

 We now prove \eqref{micio3}.  We have 
\begin{equation}\label{tom3}
 \begin{split}
 & \int d \nu^\e_\o (x, z) \nabla_\e  \varphi  ( x,z) \psi (   x+\e z ) b(\t_{x/\e}\o, z)\\ 
 %& =\e^d \sum_{x\in \hat \o}\sum _{z \in \widehat{\t_x \o}} c_{x,x+z}(\o) \frac{v (\e x+\e z )- v(\e x)}{\e} \psi (  \e  x+\e z ) b(\t_x\o, z)\\
 &= \e^d \sum_{a\in \hat \o} \sum_{a' \in  \hat \o } c_{a,a'}(\o) \frac{ \varphi (\e a')-\varphi ( \e a)}{\e} \psi (\e a') b( \t_a \o, a'-a)
 \\&
 =-  \e^d  \sum_{a' \in  \hat \o }  \sum_{a\in \hat \o} c_{a',a }(\o) 
 \frac{ \varphi (\e a)-\varphi ( \e a ')}{\e}  \psi (\e a')   \tilde b(\t_{a'} \o, a-a')
 \\
   &=- \int d \nu^\e_\o (x, z) 
  \nabla _\e \varphi  (x,z) 
  \psi ( x) \tilde{b}(\t_{x/\e}\o, z)\,.
  \end{split}
 \end{equation}
 Since we deal with  infinite sums,  the above arrangements have to be justified. 
 %Indeed,  the same computations as above still hold when taking the modulus of the involved functions.
  To show that   all series are absolutely convergent,
 as $\psi$ is bounded it is enough to show that 
 \begin{align}
 & \int  d \nu^\e_\o (x, z)  |\varphi (  x )| \, |  b(\t_{x/\e}\o, z)|<+\infty\,,\label{coop101}\\
 & \int  d \nu^\e_\o (x, z)  |\varphi (  x+\e z  )| \, |  b(\t_{x/\e}\o, z)|<+\infty\,.\label{coop102}
 \end{align}
The term \eqref{coop101} can be treated as  in \eqref{tom2} and lines after. Due Item (i),  the term \eqref{coop102} equals
$ \int  d \nu^\e_\o (x, z)  |\varphi (  x  )| \, | \tilde b(\t_{x/\e}\o, z)|$ and we are back to the previous case with $\tilde b$ instead of $b$.
\end{proof}

\begin{Definition}\label{naviglio} Let $b: \O_0\times \bbR^d \to \bbR$ be a  Borel function.
 If $\o \in   \cA_1[b]\cap \cA_1[\tilde b]\cap \O_0$,  we set ${\rm div}_* b  (\o) := \hat b (\o) - \hat{\tilde{b}}(\o) \in \bbR$.
\end{Definition}
%Note that, if $\o \in   \cA_1[b]\cap \cA_1[\tilde b]$, then ${\rm div}_* b$ takes value in $\bbR$.
%%%%%
\begin{Lemma}\label{arranco}
Let  $b: \O_0\times \bbR^d \to \bbR$ be a Borel function with $\|b\|_{ L^2(\nu)}<+\infty$. Then $\cP_0 ( \cA_1[b]\cap \cA_1[\tilde b]) =1$ and 
${\rm div}_* b  = {\rm div }\,b $ in $L^1(\cP_0)$.
\end{Lemma}
\begin{proof}
By Lemma \ref{gattonaZZ} we have $  \| \tilde b \| _{L^2(\nu)}<\infty$. Hence, both $b$ and $\tilde b$ are $\nu$--integrable. By Lemma \ref{cavallo}--(iii) we get that $\cP_0 ( \cA_1[b]\cap \cA_1[\tilde b]) =1$ and therefore ${\rm div}_* b$ is defined $\cP_0$--a.s. The identity  ${\rm div}_* b  (\o) := \hat b (\o) - \hat{\tilde{b}}(\o)$ follows from the definitions of ${\rm div}\, b$ and $ {\rm div}_* b$. 
\end{proof}

%%%%%%%%%%%%%%
\begin{Lemma}\label{tav}
Let  $b: \O_0\times \bbR^d \to \bbR$ be a Borel function with  $\|b\|_{ L^2(\nu)}<+\infty$ and such that  its class of equivalence in $L^2(\nu)$ belongs to $L^2_{\rm sol}(\nu)$. Let 
\be\label{eq_tav}
\cA_d [b]:= \{\o \in \cA_1[b] \cap \cA_1[ \tilde b]\,:\, {\rm div}_* b (\t_z \o) =0 \; \forall z \in \hat \o\}\,.
\en
Then  $\cP( \cA_d [b])=1$ and  $\cA_d [b]$ is translation invariant.
\end{Lemma}
%%%%%%%%%%
\begin{proof} By Corollary \ref{grazioso} and  Lemma \ref{arranco}, the set $A:= \{\o\in \cA_1[b] \cap \cA_1[ \tilde b]\cap\O_0\,:\, {\rm div}_*b ( \o) =0 \}$ has $\cP_0$--probability equal to $1$. By Lemma \ref{matteo}, $\cP(\tilde A)=1$ where $\tilde A:=\{ \o \in \O\,:\, \t_z \o \in A \; \forall z \in \hat \o\}$. To get that  $\cP( \cA_d [b])=1$ it is enough to observe that $\tilde A=  \cA_d [b]$. The translation invariance  follows easily  from the definition.
\end{proof}
%%%%%%%%%%
\begin{Lemma}\label{lunetta}
Suppose that $b: \O_0\times \bbR^d \to \bbR$ is a Borel function with $\|b\|_{ L^2(\nu)}<+\infty$. Take $\o \in   \cA_1[b]\cap \cA_1[\tilde b]$.
% such that 
%\begin{equation}\label{castagna}
%c_{b}(\t_z \o) <+\infty\; \text{  and } \;c_{\tilde b}(\t_z \o) <+\infty \;\; \forall z \in \hat \o\,.
%\end{equation}
 Then for any $\e>0$ and any $u:\bbR^d\to \bbR$ with compact support it holds
\begin{equation}\label{sea}
\int d \mu^\e_\o (x) u(x) {\rm div}_* b (\t_{x/\e} \o) = - \e \int d \nu ^\e_\o (x,z) \nabla_\e u(x,z) b ( \t_{x/\e} \o, z) \,.
\end{equation}
\end{Lemma}
\begin{proof} Note that $\t_{x/\e} \o$ in the l.h.s. belongs to $ \cA_1[b]\cap \cA_1[\tilde b]\cap \O_0$. Hence, 
we can write the l.h.s. of \eqref{sea} as  
\begin{equation}\label{carte}
 \int d \nu ^\e _\o (x,z) u(x) b (\t_{x/\e} \o, z) - \int d \nu ^\e_\o (x,z) u(x) \tilde b ( \t_{x/\e} \o, z) \,.
 \end{equation}
Due to our assumptions  we are dealing with absolutely convergent series, hence the above rearrangements are free. By applying \eqref{micio2} to the r.h.s. of \eqref{carte} (see Item (ii) of Lemma \ref{gattonaZ}), 
we can rewrite \eqref{carte} as 
$\int d \nu ^\e_\o (x,z)   b ( \t_{x/\e} \o, z) [ u(x) - u(x+ \e z)]$ and this allows to conclude.
\end{proof}

\section{Typical environments}\label{topo}
Consider Proposition \ref{prop_ergodico}. We stress that the function   $g$ appearing there  is a given  function and not an element of $L^1(\cP_0)$ (which would be an equivalence class of functions equal $\cP_0$--a.s.). Indeed, the set $\cA[g]$ is defined in terms of $g$ and not of its equivalence class in $L^1(\cP_0)$.
%Below, when writing that   $g: \O_0\to \bbR$  satisfies  $\|g\|_{L^p(\cP_0)}<+\infty$, we will understand that $g$ is measurable. Moreover, given a function $g: \O_0\to \bbR$, we will write $[g]$ for the class of functions equal to $g$ $\cP_0$--a.s. when we need to distinguish between the concepts.\loz

Recall that the space $(\cN,d)$ is a Polish space, where $\cN$ is given by 
the counting measures $\mu $ on $\bbR^d\times S$ (i.e. $\mu$ is an integer--valued measure on the measurable space $\left(\bbR^d\times S, \cB(\bbR^d\times S) \right)$ such that $\mu$  is bounded on bounded sets, where $\cB(\bbR^d\times S) $ denotes the family of Borel subsets of $\bbR^d\times S$).  

\begin{Remark}\label{separati}
Since  $(\cN,d)$ is a separable metric space, the same holds for $(\O, d)$ and $(\O_0, d)$. By \cite[Theorem~4.13]{Br} we then 
 get that the spaces $L^p(\cP) $ and   $L^p(\cP_0) $ are separable for $1\leq p <+\infty$. 
 %We have that  the functions $f(\o)\varphi(x)$, with $f,\varphi$ Borel and bounded, span a dense subset of $L^p(\nu)$ (since they includes the characteristic functions $\mathds{1}_A(\o) \mathds{1}_B(x)$ with $A\subset \O_0$ and $B\subset \bbR^d$ Borel.  we have that $L^p(\nu)$ is separable for $1\leq p <+\infty$.
\end{Remark}

\begin{Lemma}\label{separati_sole} The space $L^2(\nu)$ is separable. 
\end{Lemma}
\begin{proof}
Due to Remark \ref{separati}, there exists a countable dense set $\{ f_j\}$  in $L^2(\cP_0)$. At cost  to  approximate, in $L^2(\cP_0)$, $f_j $ by $[f_j]_M$  as $M\to \infty$
(cf. \eqref{taglio}), we can suppose that  $f_j$ is bounded.  Let $\{B_k\}$ be the countable family of closed balls in $\bbR^d$ with rational radius and with center in  $\bbQ^d$. Since bounded, the map $(\o,z) \mapsto f_j(\o) \mathds{1}_{B_k}(z)$  belongs to $L^2(\nu)$. We claim that the subspace spanned by the above functions  $ f_j(\o) \mathds{1}_{B_k}(z)$ is dense in $L^2(\nu)$. To this aim, we only need to show that if $g\in L^2(\nu)$ is orthogonal to this subspace, then $g=0$ $\nu$--a.e. By orthogonality we have 
\[ \int d \cP_0(\o) f_j(\o) \int d \hat \o (z) c_{0,z} (\o)  \mathds{1}_{B_k}(z) g (\o, z)=0 \qquad \forall j,k\,.\]
As $\{f_j\}$ is dense in $L^2(\cP_0)$,  for $\cP_0$--a.a. $\o$ it holds 
$ \int d \hat \o (z) c_{0,z} (\o)  \mathds{1}_{B_k}(z) g (\o, z)=0$ for all $k$. As   $\hat \o$ is locally bounded, this implies that  $c_{0,z} (\o)  g (\o, z)=0$ whenever $z\in \hat \o$, for  $\cP_0$--a.a. $\o$.  Equivalently,  $g=0$  $\nu$--a.e.\end{proof}
\medskip

In the  construction of the functional sets  presented below,  we will use the separability of $L^2(\cP_0)$ and $L^2(\nu)$ without further mention.
The definition of these functional sets and  the typical environments (cf. Definition  \ref{budda}) consists of a long list of technical assumptions, which are necessary to justify several steps in the next sections (there, we will indicate explicitly which  technical assumption we are  using).

\smallskip
Recall  the set $\cA_1[b]$ introduced in Definition \ref{artic} and formula  \eqref{taglio} for $[a]_M$. 

\medskip

\noindent
$\bullet$ {\bf The functional sets $\cG_1,\cH_1$}. 
We fix a countable set $\cH_1$ of Borel functions $b: \O_0\times \bbR^d\to \bbR$ such that 
$\|b \|_{L^2(\nu)}<+\infty$ for any $b \in \cH_1$ and such that   $\{ {\rm div} \,b\,:\, b \in \cH_1\}$ is a dense subset of $\{ w \in L^2(\cP_0)\,:\, \bbE_0[ w]=0\}$ when thought of as set of $L^2$--functions (recall Lemma \ref{santanna}). 
For each $b \in \cH_1$ we define  the Borel function $g_b : \O_0 \to \bbR$ as 
\be \label{lupetto}
g_b (\o):= 
\begin{cases}
{\rm div}_* b (\o) & \text{ if } \o \in \cA_1[b]\cap \cA_1[\tilde b]\,,\\
0 & \text{ otherwise}\,.
\end{cases}
\en
Note that by Lemma \ref{arranco} $g_b={\rm div}\, b$ $\cP_0$--a.s.
Finally we set $\cG_1:=\{ g_b \,:\, b \in \cH_1\}$.
%Then \eqref{rabarbaro} must hold for any $b \in \cG_1$, while \eqref{yelena} must hold for any $b \in \cH_1$.

\medskip

\noindent
$\bullet$ {\bf The functional sets $\cG_2,\cH_2, \cH_3$}.  We fix a countable set $\cG_2$  of  bounded Borel functions 
$g: \O_0\to \bbR$ such that  the set $\{ \nabla g\,:\, g \in \cG_2\}$, thought in $L^2(\nu)$,  is  dense in $L^2_{\rm pot}(\nu)$
 (this is possible by the definition of $L^2_{\rm pot}(\nu)$).  We  define $\cH_2$ as the set of Borel functions $h:\O_0 \times \bbR^d\to \bbR$ such that $h=\nabla g$ for some $g\in \cG_2$. We define $\cH_3$ as  the set of Borel functions $h: \O_0 \times \bbR^d\to \bbR$ such that $h(\o,z)= g(\t_z \o) z_i$ for some $g\in \cG_2$ and some direction $i=1,\dots, d$. Note that, since 
  $\bbE_0[\l_2]<+\infty$ by (A5) and since $g$ is bounded,  $\|h\|_{L^2(\nu)} <+\infty$ for all $h \in \cH_3$.

\medskip

\noindent
$\bullet$ {\bf The functional set $\cW$}. We fix a  countable set $\cW$ of Borel functions $b:\O_0\times \bbR^d\to \bbR$ such that, thought of as subset of $L^2(\nu)$, $\cW$  is dense in $L^2_{\rm sol} (\nu)$.  By  Corollary \ref{grazioso} and Lemma~\ref{gattonaZZ},  $\tilde b \in L^2_{\rm sol}(\nu)$ for any $b \in L^2_{\rm sol}(\nu)$. Hence, at cost to enlarge $\cW$, we assume that $\tilde b \in \cW$ for any $b \in \cW$.
Since $L^2(\nu)$ is separable, such a set $\cW$ exists.

\medskip

\noindent
$\bullet$ {\bf The functional set $\cG$}. 
We fix a countable set $\cG$ of Borel functions $g:\O_0\to \bbR$ such that: 
\begin{itemize}
\item  $\|g\|_{L^2(\cP_0)}<+\infty$ for any $g\in \cG$.
\item  $ 1\in \cG$, $\cG_1\subset \cG$, $\cG_2\subset \cG$.
\item $\cG$, thought of as a subset of  $L^2(\cP_0)$, is dense in $L^2(\cP_0)$.
%At cost to enlarge $\cG_1$ (but keeping it countable) we can assume that $g g' \in \cG_1$ for any $g g'\in \cG_1$.
\item  For each $b \in \cW$, $M \in \bbN$ and coordinate $i=1,\dots,d$,  the function $f: \O_0 \to \bbR$ defined as 
\be\label{nonfamale}
f(\o):= \begin{cases}
\int d \hat \o (z) c_{0,z}(\o) z_i [b]_M(\o,z) & \text{if } \int d \hat \o (z) c_{0,z}(\o)|z_i|<+\infty \\
0 &\text{otherwise}
\end{cases}
\en
belongs to $\cG$.
Since $|f(\o)| \leq M \l_1(\o)$,  $\|f\|_{L^2(\cP_0)}<+\infty$ by (A5).
\item At cost to enlarge $\cG$ we assume that $[g]_M\in \cG$ for any $g\in \cG$ and $M\in \bbN$.
\end{itemize}

\medskip

\noindent
$\bullet$ {\bf The functional set $\cH$}. 
We fix a countable set of Borel functions $b : \O_0 \times\bbR^d\to \bbR$ such that  
\begin{itemize}
\item $\|b\|_{L^2(\nu)}<+\infty$  for any $b\in \cH$.
\item $\cH_1\cup  \cH_2 \cup \cH_3\cup  \cW \subset \cH$ and  $ 1\in \cH$.
\item for any $i=1,\dots, d$ the map $(\o,z) \mapsto z_i$ is in $\cH$ (recall (A5)).
\item $\cH$, thought as a subset of  $L^2(\nu)$, is dense in $L^2(\nu)$.
\item  At cost to enlarge $\cH$ we assume that $[b]_M\in \cH$ for any $b\in \cH $ and $M\in \bbN$.
\end{itemize}

Recall the set $\cA[g]$  in Proposition \ref{prop_ergodico} for a  Borel function  $g: \O_0\to \bbR$ with $\|g\|_{L^1(\cP_0)}<+\infty$. Recall  Definition \ref{artic} and in particular the set $\cA_1[b]$ for a  Borel form $b: \O_0 \times \bbR^d\to \bbR$ with $\|b\|_{L^1(\nu)}<+\infty$. Recall Definition \ref{ometto}. Given a function $f: \O\to [0,+\infty]$ such that  $\cP( f =+\infty)=0$,  we define $\cA[f]$ as $\cA[f_*]$, where $f_*: \O \to \bbR$ is defined as $f$ on $\{ f <+\infty\}$ and as $0$ on $\{ f =+\infty\}$. 
%Finally, given a set $A\subset \O_0$, we denote  $A^\sharp$ the set $\{ \o \in \O\,:\, \t_z \o \in A\; \forall z\in \hat \o\}$.

\begin{Definition}\label{budda}
We define $\O_{\rm typ}$ as the intersection of the following sets:
\begin{itemize}
\item[(S1)]  $\cA[g g']$ as  $g,g'$ vary  among $ \cG$. $\| g g'\|_{L^1(\cP_0) } <+\infty$  by Schwarz inequality, hence $\cA[g g']$ is well defined. Note also that $\cA[g]= \cA[g \cdot 1]$. 

\item[(S2)] $\cA_1[bb'] $ as $b,b'$ vary among $\cH$.   $\|b  b'\|_{L^1(\nu)}<+\infty$  
by Schwarz inequality.

% for all $b,b'\in \cH$. By Lemma \ref{cavallo}  $\cP_0(\cA^{(1)}_{bb'})=1$.

\item[(S3)] $\cA[\widehat {bb'}]$ as $b,b'$ vary among $\cH$.  Note that $\|bb'\|_{L^1(\nu)}<+\infty$ by Schwarz inequality, hence $\|\widehat {bb'}\|_{L^1(\cP_0)} <+\infty$ by Lemma \ref{cavallo}.

\item[(S4)] $\cA_1[|b-[b]_M| ]$,  as $b$ varies in $\cW$ and $M$ varies in $\bbN$.  Since $\|b\|_{L^2(\nu)}<+\infty$, we have $\|b-[b]_M\|_{L^1(\nu)}<+\infty$.

\item[(S5)] $\cA[\widehat{ d}] \cap \cA[\widehat{ d^2}]$, where $d:=|b-[b]_M|$, as $b$ varies in $\cW$ and $M$ in $\bbN$. Since $\|b\|_{L^2(\nu)}<+\infty$, it holds  $\| \widehat{ d}\|_{ L^1(\cP_0)},\| \widehat{ d^2}\|_{ L^1(\cP_0)}<+\infty$ by Lemma \ref{cavallo}.

\item[ (S6)] $ \cA_1[b]\cap \cA_1[\tilde b]$ as $b$ varies in $\cH_1$. Since $\|b\|_{L^2(\nu)}<+\infty$, it holds $\|b\|_{L^1(\nu)}<+\infty$ and  $\|\tilde b\|_{L^1(\nu)}<+\infty$ (see  Lemma \ref{gattonaZZ}).

\item[(S7)] $\cA_1[ |z|^2]=\{ \o \in \O\,:\, \l _2(\t_z\o)<+\infty \; \forall z \in \hat \o\}$.  Since  $\|\l_2\|_{L^1(\cP_0)}<+\infty$ by (A5), the form $(\o,z) \mapsto |z|^2$ is $\nu$--integrable.

\item[(S8)] $\cA_1 [(\tilde b)^2] \cap \cA_1[ b^2]$ as $b$ varies in $\cH$. Note that $\| (\tilde b)^2\|_{L^1(\nu)}=\| b^2\|_{L^1(\nu)}<+\infty$ by Lemma \ref{gattonaZZ}.

\item[(S9)] $\cA[ \l_2]$. By Assumption (A5) $\|\l_2\|_{L^1(\cP_0)}<+\infty$.

\item[(S10)] $\cA [\widehat{(  \tilde b^2+ b^2 )}]$  as $b$ varies in $ \cH_1$. Note   that $\bbE_0[ \widehat{(  \tilde b^2+ b^2 )}]= \nu( \tilde b^2)+\nu (b^2) <\infty$ by Lemma \ref{gattonaZZ} and since $\|b\|_{L^2(\nu)}<+\infty$.

\item[(S11)] $\cA[\l_0^2]$. By (A5), $\| \l_0^2 \|_{L^1( \cP_0)} <+\infty$.

\item[(S12)] $\cA_1 [ 1]=\{ \o \in \O\,:\, \l _0(\t_z\o)<+\infty \; \forall z \in \hat \o\}$. By Assumption (A5)  $\|1\|_{L^1(\nu)}=\|\l_0\|_{L^1(\cP_0)}<+\infty$.

\item[(S13)] $\{\o\in \O\,:\, F_*(\t_z\o)<+\infty \; \forall z\in \hat \o\}\cap \cA[F_*]$ with $F_*$ as in (A6). By  (A6), $\|F_*\|_{L^1(\cP_0)}<+\infty$.

\item[(S14)] $\cA_d[b]$ as $b$ varies in $\cW$ (recall \eqref{eq_tav} in Lemma \ref{tav}). 
%Note that by Lemma \ref{tav} $\cP_0( \cA_d[b])=1$ and  by definition $ \cA_d[b]\subset \cA_1[b] \cap \cA_1[\tilde b]$.

\item[(S15)] 
 $\cA_1[(b-[b]_M)^2 ]\cap \cA_1[(\tilde b-[\tilde b]_M)^2 ]$,  as $b$ varies in $\cW$ and $M$ varies in $\bbN$.  Note that both $(b-[b]_M)^2 $ and $(\tilde b-[\tilde b]_M)^2$ have bounded $L^1(\nu)$--norm since $\|b\|_{L^2(\nu)
}<+\infty$ and by Lemma \ref{gattonaZZ}. 
 
 \item[(S16)]  $\cA[h_\ell ]$ as $\ell$ varies in  $\bbN$, where $h_\ell (\o):= \int d \hat \o(z) c_{0,z}(\o) |z|^2  \mathds{1}_{\{ |z| \geq  \ell\}} $.
 We have $\|h_\ell \|_{L^1(\cP_0)}\leq \|\l_2\|_{L_1(\cP_0)}<+\infty$ by Assumption (A5).
 \item[(S17)] $\cA[\l_0]$. We have $\|\l_0\|_{L^1(\cP_0)} <+\infty$ by Assumption (A5).
% Since $\|b\|_{L^2(\nu)}<+\infty$ for any $b \in\cW$, by Lemma \ref{cavallo} $\cP_0\bigl( \cA_1[  b-[b]_M]\bigr)=1$.
%, hence the definition is well posed (cf. Prop. \ref{prop_ergodico}).
\end{itemize}
\end{Definition}

%By immediate consequence of Prop.~\ref{prop_ergodico}, Lemma \ref{cavallo} and Lemma \ref{tav}, we get the following property:
\begin{Proposition}\label{stemino}
The above set $\O_{\rm typ}$ is a translation invariant Borel subset of $\O$ such that  $\cP(\O_{\rm typ})=1$.
\end{Proposition}
\begin{proof}
The claim  follows from Proposition \ref{prop_ergodico} for all sets (Sx) of the form $\cA[g]$, from   Lemma \ref{cavallo} 
for all sets (Sx) of the form $\cA_1[b]$,  from 
 Lemma   \ref{tav}  for all sets (Sx) of the form $\cA_d[b]$ and from Lemma \ref{eleonora} for $\{\o\in \O\,:\, F_*(\t_z\o)<+\infty \; \forall z\in \hat \o\}$ in (S13).  \end{proof}

%\begin{Definition}\label{gigetto} We define the functional space $\cG$ as 
%\[ \cG= \cG_1 \cup\]
%\end{Definition}

%%%%%%%%%%%%%%%%%%%%%%%%%%%%%%%%%%%%%%%%%%%%%%%%%%%%%%%%%%%%%%%%%%%%%%%%%%%%%%%%%%%%%%%%%%%%%%%%%%%%%%%%%%%%%%%%%%%%%
\section{2-scale convergence of $v_\e \in L^2 (\mu^\e_{\tilde\o})$ and  of $w_\e \in L^2 (\nu^\e_{\tilde \o})$}\label{sec_2scale}
%In general, given  an $\e$--parametrized family of Banach spaces $\bbB_\e$ endowed with the norm $\| \cdot \|_{\bbB_\e}$, 
%an $\e$--parametrized family of functions $v_\e\in \bbB_\e$ is called \emph{bounded} (or bounded in $\bbB_\e$) if 
%\[ \limsup_{\e\downarrow 0}  \|v_\e\|_{\bbB_\e} 
% <\infty\,.
%\]
%The weak 2-scale convergence and the strong 2-scale convergence refine the concept of weak and strong convergence introduced in Definition \ref{debole_forte}.

\subsection{2-scale convergence of $v_\e \in L^2 (\mu^\e_{\tilde\o})$}
\begin{Definition}\label{priscilla}
Fix   $\tilde \o\in \O_{\rm typ}$, an $\e$--parametrized  family $v_\e \in L^2( \mu^\e_{\tilde \o})$ and  a function $v \in L^2 \bigl(m dx \times \cP_0 \bigr)$.
\\
$\bullet$ We say that \emph{$v_\e$ is weakly 2-scale convergent to $v$}, and write 
$v_\e \stackrel{2}{\rightharpoonup} v$, 
if the family $\{v_\e\}$ is bounded, i.e.
$
 \limsup_{\e\downarrow 0}  \|v_\e\|_{L^2(\mu^\e_{\tilde \o})}<+\infty$, 
 and 
\begin{equation}\label{rabarbaro}
\lim _{\e\downarrow 0} \int d \mu _{\tilde \o}^\e (x)  v_\e (x) \varphi (x) g ( \t _{x/\e} \tilde{\o} ) =\int d\cP_0(\o)\int dx\, m v(x, \o) \varphi (x) g (\o)  \,,
\end{equation}
for any $\varphi \in C_c (\bbR^d)$ and any $g \in \cG$.\\
$\bullet$
 We say that \emph{$v_\e$ is strongly 2-scale convergent to $v$}, and write 
$v_\e \stackrel{2}{\to} v$, 
if the family $\{v_\e\}$ is bounded and 
\begin{equation}\label{gingsen}
\lim_{\e\downarrow 0} \int d\mu^\e_{\tilde \o}(x)   v_\e(x) u_\e(x) = \int d\cP_0(\o)\int  dx\,m   v(x, \o) u(x,\o)  
\end{equation}
whenever $u_\e \stackrel{2}{\rightharpoonup} u$. 
\end{Definition}

\begin{Lemma}\label{medaglia} Let $\tilde \o \in \O_{\rm typ}$. Then,
for any $\varphi\in C_c (\bbR^d) $ and $g \in \cG$, setting $v_\e(x):= \varphi(x) g(\t_{x/\e} \tilde \omega)$ it holds 
$v_\e \stackrel{2}{\rightharpoonup} \varphi(x) g(\o)$.
\end{Lemma}
\begin{proof} Since $\tilde \o \in \O_{\rm typ}$ (cf.  (S1) in Def.~\ref{budda} and Prop. \ref{prop_ergodico}), we get $\lim_{\e\da 0}  \|v_\e\|_{L^2(\mu^\e_{\tilde \o})}^2= \int dx\, m \varphi(x)^2 \bbE[g^2]$, hence the family $\{v_\e\}$ is bounded in $L^2(\mu^\e_{\tilde \o})$.
 Since $g\in \cG\subset  L^2(\cP_0)$, we have $\varphi(x) g(\o) \in L^2(m dx \times \cP_0)$. Take $\varphi_1 \in C_c(\bbR^d)$ and $g_1\in \cG$. Since $\tilde \o \in \O_{\rm typ}$ (cf.  (S1) and Prop. \ref{prop_ergodico}), it holds $\int d \mu _{\tilde \o}^\e (x)  v_\e (x) \varphi _1(x) g_1 ( \t _{x/\e} \tilde{\o} ) \to\int dx\, m \varphi(x) \varphi_1(x) \bbE_0[g g_1]$.\end{proof}
% kuka \blu{\space Mantengo la stessa definizione \ref{priscilla} anche se $\mu^\e_{\o}$ is defined as in \eqref{wow1} ma ora $\varphi \in C_c (S)$}
%%%%%%%%%%%%%%%%%%%%%
\begin{Lemma}\label{alba}
Given $\tilde \o \in \O_{\rm typ}$, if $v_\e \stackrel{2}{\toup} v$ then 
\begin{equation}\label{acquario}
\varlimsup _{\e \da 0} \int d\mu^\e_{\tilde \o}(x)v_\e^2(x)   \geq \int d\cP_0(\o)\int dx\,m v(x, \o)^2\,.
\end{equation}
\end{Lemma} 
%%%%%%%%%%%%%%%%%%%%%%%%%%%%%%%
The proof is similar to the proof of \cite[Item~(iii),~p.~984]{Z}. We give it for completeness since our definition of 2-scale convergence involves the set $\cG$. 
%%%%%%%%%%%%%%%%%%%%%%%%%%%%%%%
\begin{proof} Since $\cG$ is dense in $L^2(\cP_0)$ and $C_c(\bbR^d)$ is dense in $L^2( dx)$, given $\d>0$  we can find functions $g_1,\dots, g_k \in \cG$, $\varphi_1, \dots, \varphi_k \in C_c(\bbR^d)$ and coefficients $a_1, \dots, a_k \in \bbR$ such that 
the norm of $v-\Phi$ in  $L^2 \bigl( m dx \times \cP_0 \bigr)$ is bounded by $\d$, where $\Phi(x,\o) :=  \sum_{i=1}^k a_i \varphi_i(x) g_i(\o)$. 
We have 
\begin{align}
& \int  d \mu _{\tilde \o}^\e (x)  v_\e (x)^2  \geq 
2 \int  d \mu _{\tilde \o}^\e (x) v_\e (x)  \Phi  (x ,  \t _{x/\e} \tilde{\o} )  -
\int  d \mu _{\tilde \o}^\e (x)
 \Phi  (x ,  \t _{x/\e} \tilde{\o} ) ^2  \label{raggio0}\\
& \int d \mu _{\tilde \o}^\e (x)  v_\e (x)  \Phi  (x ,  \t _{x/\e} \tilde{\o} )  =\sum_{i=1}^k a_i \int d \mu _{\tilde \o}^\e (x) v_\e (x)  \varphi_i (x)  g_i(  \t _{x/\e} \tilde{\o} )  \label{raggio1}\\
& \int d \mu _{\tilde \o}^\e (x) \Phi  (x ,  \t _{x/\e} \tilde{\o} ) ^2 = \sum_{i=1}^k \sum_{j=1}^k a_i a_j  \int  d \mu _{\tilde \o}^\e (x) (\varphi_i \varphi_j) (x)  (g_i g_j)(  \t _{x/\e} \tilde{\o} )  \,.\label{raggio2}
\end{align} 
We take the limit $\e\da 0$ in \eqref{raggio0}. We use that $v_\e \stackrel{2}{\toup} v$ and  $\tilde \o \in \O_{\rm typ}$ to deal with  \eqref{raggio1} (cf. Def.~\ref{priscilla})  and we use that $\o \in \cA[g_i g_j]$  to deal with \eqref{raggio2} (cf. (S1) in Def. \ref{budda} and Prop. \ref{prop_ergodico}). We then get
\begin{equation}
\begin{split}
\varlimsup _{\e \da 0}  & \int  d\mu^\e_{\tilde \o}(x) v_\e^2(x)  \\
& \geq 2 \int  d\cP_0(\o)\int   dx \, m\,v(x, \o) \Phi(x, \o) - \int d\cP_0(\o)\int   dx \, m\, \Phi(x, \o)^2\\
%& =  - \int  d\cP_0(\o)\int   dx \, m\, \bigl[ v-\Phi\bigr]^2 (x,\o)+  \int  d\cP_0(\o)\int   dx \, m\,v(x,\o)^2\\
& \geq - \d +  \int  d\cP_0(\o)\int  dx \, m\,v(x,\o)^2\,.
\end{split}
\end{equation}
By the arbitrariness of $\d$, we get \eqref{acquario}.
\end{proof}

%%%%%%%%%%%%%%%%%%%%%%%%%%%%%%%%%%%%%%%%%%%%%%
%\begin{Lemma}\label{tajut}
% Take $\phi \in C_c(\bbR^d)$ and fix $g \in \cG_* $.  Fix $\tilde \o \in \O_{\rm typ}$ and set  
%$v_\e(x):= \phi (x) g ( \t_{x/\e} \tilde \o)$. Then $v_\e  \stackrel{2}{\toup} v $ where $v(x,\o)= \phi(x) g (\o)$.
%\end{Lemma}
%\begin{proof}
%By construction $\cG_*$ is given by bounded functions, thus implying that $v$ is bounded and therefore $v\in L^2( mdx \times \cP_0)$. Take now $\varphi\in C_c(\bbR^d)$ and $b \in \cG_*\cup \cC$. We have 
%\begin{equation}\label{polenta}
% \int _{\bbR^d} v_\e (x) \varphi (x) b ( \t _{x/\e} \tilde{\o} ) d \mu _{\tilde \o}^\e (x)=\int (\phi \varphi )(x) (g b)  ( \t _{x/\e} \tilde{\o} ) d \mu _{\tilde \o}^\e (x)\,.
% \end{equation}
%Note that  $g b \in \cG_*$ if $b\in \cG_*$ and  $g b \in \cC$ if $b\in \cC$.  
% Hence, if  $\tilde \o \in \O_{\rm typ}$, we have $\tilde \o \in \cA_{g b}$ by properties (P2) and (P6). By Proposition \ref{prop_ergodico} we conclude that  \eqref{polenta} converges as $\e \da 0$  to $\int _\O d\cP_0(\o) (g b) (\o) 
% \int _{\bbR^d} (\phi \varphi )(x) m dx$,  thus corresponding to \eqref{rabarbaro}.\end{proof}
%%%%%%%%%%%%%%%%%%%%%%%%%%%%%%%%%%%%%%%%%%%%%%
Using Lemmas \ref{medaglia} and \ref{alba} one gets the following characterization:
\begin{Lemma}\label{fantasia}
Given  $\tilde \o \in \O_{\rm typ}$, $v_\e \stackrel{2}{\to} v$ if  and only if $v_\e \stackrel{2}{\rightharpoonup} v$ and 
\begin{equation}\label{orchino}
\lim _{\e\da 0} \int_{\bbR^d} d\mu^\e_{\tilde \o}(x) v_\e(x)^2 = \int_\O d \cP_0(\o) \int _{\bbR^d} dx \,m v(x,\o)^2 \,.
\end{equation}
\end{Lemma}
%%%%%%%%%%%%%%%%%%%%%%
%%%%%%%%%%%%%%%%%%%%%%%%%
\begin{proof}
 If $v_\e \to v$, then
 $v_\e \stackrel{2}{\rightharpoonup} v$ by 
    Lemma \ref{medaglia}.  By then applying \eqref{gingsen} with $u_\e:=v_\e$, we get \eqref{orchino}. The opposite  implication corresponds to \cite[Item~(iv),~p.~984]{Z} and the proof there can be easily adapted to our setting due to Lemma \ref{alba}.
For completeness  we have provided it in 
%mammina 
%\cite[App.~B]{FV2}. % versione sottomessa
 Appendix \ref{aggiuntina}. % versione arxiv
\end{proof}

% For completeness, we give the proof in Appendix \ref{dimo}.

\begin{Lemma}\label{compatto1}
 Let $\tilde \o\in \O_{\rm typ}$.    Then, given a bounded family of functions $v_\e\in L^2 ( \mu^\e_{\tilde{\o}})$,  there exists a subsequence $\{v_{\e_k}\}$ such that
 $ v_{\e_k}
  \stackrel{2}{\rightharpoonup}   v $ for some 
 $  v \in L^2(m dx \times \cP_0 )$ with $\|  v\|_{   L^2(m dx \times \cP_0 )}\leq \limsup_{\e\da 0} \|v_\e\|_{L^2( \mu^\e_{\tilde{\o}})}$.
 \end{Lemma}
 The proof is similar to the proof of \cite[Prop.~2.2]{Z}. We give it for completeness since our definition of 2-scale convergence involves the set $\cG$. 
 
 \begin{proof}  Since $\{v_\e\}$ is bounded in $L^2(\mu_{\tilde \o}^\e)$, there exist $C, \e_0$ such that $\| v_\e\|  _{L^2(\mu_{\tilde \o}^\e)}\leq C $ for $\e \leq \e_0$. We fix a countable set $\cV \subset C_c(\bbR^d)$ such that $\cV$ is dense in $L^2(m dx)$.
We call $\cL$ the family of functions $\Phi$ of the form  $\Phi(x,\o) =  \sum_{i=1}^r a_i \varphi_i(x) g_i(\o)$, where   $r\in \bbN_+$, 
 $g_i \in \cG$, $\varphi_i \in \cV$ and  $a_i \in \bbQ$. Note that $\cL$ is a  dense subset  of $L^2( m dx \times \cP_0)$.  By Schwarz inequality we have  
 \begin{equation}\label{cocchi}
\Big| \int   d \mu _{\tilde \o}^\e (x)  v_\e (x)  \Phi  (x ,  \t _{x/\e} \tilde{\o} ) \Big | \leq  C 
\Big[ \int d \mu _{\tilde \o}^\e (x)  \Phi  (x ,  \t _{x/\e} \tilde{\o} )  ^2  \Big]^{1/2}\,.
%\| \Phi \|_{ L^2( m dx \times \cP_0) }\,.
 \end{equation}
% Since $\o \in \O_{\rm typ}$ and $g_i \in \cG_*$, by linearity and property (P2) 
% \begin{equation}\label{CL1}
%  \lim_{\e \da 0}
%  \int _{\bbR^d} v_\e (x)  \Phi  (x ,  \t _{x/\e} \tilde{\o} )  d \mu _{\tilde \o}^\e (x)=
%  m \int _\O d \cP_0 (\o)   \int _{\bbR^d} dx\,m \Phi(x, \o) \,.
%  \end{equation}
 By expanding the square in the r.h.s.,  
 since  $\tilde \o \in \O_{\rm typ}$ (cf. (S1) in Def. \ref{budda} and Prop. \ref{prop_ergodico}),    we have 
 \begin{multline}\label{CL}
\lim_{\e \da 0}  \int   d \mu _{\tilde \o}^\e (x) \Phi  (x ,  \t _{x/\e} \tilde{\o} )  ^2  \\
= 
\sum_ i \sum_j a_i a_j \int dx\,m\, \varphi_i(x) \varphi_j(x) \bbE_0[ g_i g_j]= \| 
\Phi\|^2_{ L^2( m dx \times \cP_0)}\,.
\end{multline}
As a first application of \eqref{CL} we get that, for $\e$ small, the l.h.s. of \eqref{cocchi} is bounded uniformly in $\e$, hence it admits a convergent subsequence. 
  Since $\cL$ is  a  countable family, by a diagonal procedure we can extract a subsequence $\e_k \da 0 $ such that the limit 
 \[
 F( \Phi):=  \lim _{k \to \infty}   \int d \mu _{\tilde \o}^{\e_k} (x) 
 v_{\e_k} (x)  \Phi  (x ,  \t _{x/\e_k} \tilde{\o} )   \] exists for any $\Phi \in \cL$  and it satisfies $| F(\Phi)| \leq C \| \Phi \|_{ L^2( m dx \times \cP_0) }$ by \eqref{cocchi} and \eqref{CL}. 
 Since $\cL$ is  a dense subset of  $L^2( m dx \times \cP_0)$, by Riesz's representation theorem  there exists a unique $v\in L^2( m dx \times \cP_0)$ such that $F(\Phi)= \int d\cP_0(\o) \int dx \,m \, \Phi(x, \o) v(x,\o)$ for any $\Phi \in \cL$. We have $\| v\| _{ L^2( m dx \times \cP_0) }\leq C$.
 As $\Phi(x,\o):= \varphi(x) g(\o) $ - with $\varphi \in \cV $ and $b \in \cG$ - belongs to $\cL$, we get
 that \eqref{rabarbaro} is satisfied along the subsequence $\{\e _k\}$ for any  $\varphi \in \cV $, $b \in \cG$. 
 It remains to show that we can indeed take   $\varphi \in \cC_c(\bbR^d) $. To this aim we observe that we can take $\cV$ fulfilling the following properties: (i) for each $N\in \bbN_+$ $\cV$ contains a function  $\phi_N\in C_c(\bbR^d)$ with values in $[0,1]$ and equal to $1$ on $[-N,N]^d$; (ii)
 each $\varphi\in C_c(\bbR^d) $ can be approximated in uniform norm  by functions $\psi_n\in \cV$ such  that   $\psi_n$ has support inside $[-N,N]^d$, 
 where  $N=N(\varphi)$ is the minimal integer for which $\varphi$ has support inside 
 $[-N,N]^d$.
  By Schwarz inequality and the boundedness of $\{v_\e\}$  we can estimate
 \[
\Big |  \int d \mu _{\tilde \o}^\e (x)  v_\e (x) [\varphi (x)-\psi_n] g ( \t _{x/\e} \tilde{\o} )
 \Big|^2 \leq C^2 \|\varphi - \psi_n\|  _\infty ^2\int d \mu _{\tilde \o}^\e (x) \phi_N (x)
 g ( \t _{x/\e} \tilde{\o} )^2\,.
 \]
 Since $\tilde \o \in \O_{\rm typ}\subset  \cA[g^2]$ for all $g\in \cG$ (cf.  (S1) in Def. \ref{budda}), by Prop. \ref{prop_ergodico} the last integral converges as $\e \to \infty$ to $C':= \int dx \, m  \phi_N \bbE_0[g^2]$. 
In particular,  using also  that $\psi_n \in \cV$, along the subsequence $\{ \e_k\}$ we have  
 \be
 \begin{split}  & \varlimsup_{\e \da 0} \int d \mu _{\tilde \o}^\e (x)  v_\e (x) \varphi (x) g ( \t _{x/\e} \tilde{\o} )\\ & \leq C'  \|\varphi - \psi_n\|  _\infty 
 + \varlimsup_{\e \da 0}\int d \mu _{\tilde \o}^\e (x)  v_\e (x) \psi _n(x)g ( \t _{x/\e} \tilde{\o} ) \\
 &=  C'    \|\varphi - \psi_n\|  _\infty  + \int d\cP_0(\o) \int dx \,m \,  v(x,\o)  \psi_n(x) g(\o)\,.
 \end{split}
 \en
 We now take the limit $n\to \infty$.
 % Since $\psi_n(x) \to \varphi(x)$ for any $x$ and 
 %\[ | v(x,\o)\psi_n(x)  b(\o)| \leq L \mathds{1}_{\{|x|\leq N\}} |b(\o)| \,  | v(x,\o)|\in L^1( m dx \times \cP_0)\,,\]
 By dominated convergence we conclude that, along the subsequence $\{ \e_k\}$,
 \be 
 \varlimsup_{\e \da 0} \int d \mu _{\tilde \o}^\e (x)  v_\e (x) \varphi (x) g ( \t _{x/\e} \tilde{\o} ) \leq \int d\cP_0(\o) \int dx \,m \,   v(x,\o) \varphi(x) g(\o)\,.
 \en
 A similar result holds with the liminf, thus implying that \eqref{rabarbaro}  holds along the subsequence $\{ \e_k\}$ for any $\varphi \in  C_c(\bbR^d)$ and $g \in \cG$.
 \end{proof}

\subsection{2-scale convergence of  $w_\e \in L^2 (\nu^\e_{\tilde \o})$} 
 \begin{Definition}
Given  $\tilde \o\in \O_{\rm typ}$, a family $w_\e \in L^2( \nu ^\e_{\tilde \o})$ and  a function  $w \in L^2 \bigl( m dx \times d\nu\bigr)$, we say that \emph{$w_\e$ is weakly 2-scale convergent to $w$}, and write $  w_\e \stackrel{2}{\rightharpoonup}w   $,  if   $\{w_\e\}$ is bounded
 in $L^2( \nu ^\e_{\tilde \o})$, i.e. $
 \varlimsup_{\e \da 0} \|  w_\e\|_{L^2( \nu ^\e_{\tilde \o})}<+\infty$,
   and
\begin{multline}\label{yelena}
 \lim _{\e\downarrow 0} \int   d \nu_{\tilde \o}^\e (x,z) w_\e (x,z ) \varphi (x) b ( \t _{x/\e} \tilde{\o},z )\\
=\int  dx\,m \int d \nu (\o, z) w(x, \o,z ) \varphi (x) b (\o,z )  \,,
\end{multline}
for any $\varphi \in C_c (\bbR^d)$ and  any $b \in \cH $. 
\end{Definition}

% kuka \blu  {\spade In the above definition I would take $\varphi \in C_c(S)$ when $\nu^\e_\o= \nu^\e_{\o,S}$.}
%\rosso{
%\loz By reasoning as in \eqref{befana} and \eqref{merlino11} we have that the integral in the l.h.s. of \eqref{rabarbaro} is meaningful. oppure vedi stime in proof lemma \ref{compatto2}. va precisato.  Problem $\cH$ deve essere dato da funzioni limitate.
%}

%\club  We point out that the integral in the l.h.s. of \eqref{yelena} does not depend on the representative  of $b $.

\begin{Lemma}\label{compatto2} Let $\tilde \o\in \O_{\rm typ}$.  Then, given a bounded family of functions $w_\e\in L^2 ( \nu^\e_{\tilde{\o}})$,  there exists a subsequence $\{w_{\e_k}\}$   such that  $w_{\e_k} \stackrel{2}{\rightharpoonup} w$ for some 
 $ w \in L^2(  m\,dx \times \nu )$ with  $\|  w\|_{   L^2(m dx \times \nu )}\leq \limsup_{\e\da 0} \|w_\e\|_{L^2( \nu^\e_{\tilde{\o}})}$.
% Moreover, if $\limsup_{\e \da 0} \| w_\e \|_{L^\infty ( \nu^\e_{\tilde{\o}})}<+\infty$,  \eqref{yelena} holds also for any $\varphi \in C_c(\bbR^d)$ and any $b \in \cW$.
\end{Lemma}
\begin{proof} 
The proof of Lemma \ref{compatto2} is similar  to the proof of Lemma \ref{compatto1}. We only give some comments on some new steps. One has to replace $L^2( m\, dx \times \cP_0)$ with $L^2(  m\,dx \times \nu )$.  Now  $\cL$ is  the family of functions $\Phi$ of the form  $\Phi(x,\o,z) =  \sum_{i=1}^r a_i \varphi_i(x) b_i(\o,z)$, where   $r\in \bbN_+$, 
 $b_i \in \cH $, $\varphi \in \cV$ and  $a_i \in \bbQ$. Above $\cV$ is a countable dense subset of $L^2(mdx)$ given by functions on $C_c(\bbR^d)$.  Then $\cL$ is a countable dense subset of $L^2(  m\,dx \times \nu )$. 
Due to  Def. \ref{artic}, Lemma \ref{cavallo} and since $\tilde \o \in \O_{\rm typ}$ (cf. (S2) in Def. \ref{budda}), we can write
 \begin{multline}
 \int d \nu _{\tilde \o}^\e (x,z)   \Phi  (x ,  \t _{x/\e} \tilde{\o},z  )  ^2  \\ = \sum _i \sum_j a_i a_j 
 \int d \nu _{\tilde \o}^\e (x,z)  \varphi_i(x) \varphi_j (x) b_i (  \t _{x/\e} \tilde{\o},z) b_j (  \t _{x/\e} \tilde{\o},z)\\
 = \sum _i \sum_j a_i a_j  \int d \mu _{\tilde \o}^\e (x) \varphi_i(x) \varphi_j (x) \widehat{b_i b_j}(  \t _{x/\e} \tilde{\o})
  \,.
  \end{multline}
 % Since $b_i,b_j\in L^2(\nu)$, we have $b_i b_j \in L^1(\nu)$ and, by Lemma \ref{cavallo}
 Since $\tilde \o \in \O_{\rm typ}$ (cf. (S3) in Def. \ref{budda})  by Prop. \ref{prop_ergodico} we have 
   \begin{multline}%\label{CL100}
\lim_{\e \da 0}  \int  d\nu _{\tilde \o}^\e (x,z)   \Phi  (x ,  \t _{x/\e} \tilde{\o},z  )  ^2   \\
= 
\sum_ i \sum_j a_i a_j \int _{\bbR^d} dx\,m\, \varphi_i(x) \varphi_j(x) \bbE_0[ \widehat{b_i b_j}]= \| 
\Phi\|^2_{ L^2( m dx \times \nu)}\,.
\end{multline}
Above, to get the last identity,  we have used that $\bbE_0[ \widehat{b_i b_j}]= \int d \nu  b_i b_j$.
\end{proof}

%%%%%%%%%%%%%%%%%%%%%%%%%%%%%%%%%%%%%%%%%%%%%%%%%%%%%%%%%%%%%%%%%%%%%%%%%%%%%%%%%%%%%%%%%%%%%%%%%%%%%%%%%%%%%%%%%%%%%%%%%%%%%%%%%%%%%%%%%%%%%%%%%%%%%%%%%%%%%%%%%%%%%%%%%%%%%%%%%%%%%%%%%%%%%%%%%%%%%%%%%%%%%%%%%%%%%%%%%%%%%%%%%%%%%%%%%%%%%%%%%%%%%%%%%%%%%%%%%%%%%%%%%%%%%%%%%%%%%%%%%%%%%%

%%%%%%%%%%%

%\section{$2$--scale convergence of amorphous gradients}
%%%%%%%%%%%%%%%%%%%%%%%%%%%%%%%%%%%%

\section{Cut-off for functions  $v_\e\in L^2(\mu^\e_{\tilde \o})$}\label{cut-off1}

We recall that  $\bbN_+$ denotes the  set of positive integers. Recall \eqref{taglio}.
\begin{Lemma}\label{lemma1}
Let $\tilde \o \in \O_{\rm typ}$ and let $\{v_\e\}$ be a family of functions such that   $v_\e\in L^2(\mu^\e_{\tilde \o})$ and 
$\varlimsup_{\e\da 0} \| v_\e\|_{L^2(\mu^\e_{\tilde \o})}<+\infty$. Then  there exist functions $v, v_M\in L^2 (m dx \times \cP_0) $ with $M$ varying in $\bbN_+$ such that 
\begin{itemize}
\item[(i)]  $v_\e \stackrel{2}{ \toup}  v$ and $[v_\e ]_M\stackrel{2}{ \toup}  v_M$ for all $M\in \bbN_+$, along a subsequence $\{\e_k\}$;
\item[(ii)] for any $\varphi \in C_c(\bbR^d) $ and $u\in \cG$ it holds
\begin{multline}\label{queen1}
\lim_{M \to \infty}\int dx \,m \int d\cP_0 (\o) v_M(x,\o) \varphi (x) u(\o)\\
=
\int dx\,  m \int d \cP_0(\o) v(x,\o) \varphi(x) u(\o)\,.
\end{multline}
\end{itemize}
\end{Lemma}
%%%%%%%%%%%%%%%%%%%%%%%%%%%%%%%%%%%%

%%%%%%%%%%%%%%%%%%%%%%%%%%%%%%%%%%%%%%%%%%%%%%%%%%%%%%%%%%%%%%%%%%%%%%%%%%%%%%%%%%%%%%%%%%%%%%%%%%%%%%%%%%%%%%%%%%%%%%%%
\begin{proof}
Without loss, we assume that 
 $\| v_\e\|_{L^2(\mu^\e_{\tilde \o})}\leq C_0 <+\infty$ for all $\e$. 
We set $v^\e_M:= [v_\e]_M$. Since  $\| v^\e_M\|_{L^2(\mu^\e_{\tilde \o})}\leq \|v_\e\|_{L^2(\mu^\e_{\tilde \o})}\leq C_0$, Item (i) follows from Lemma \ref{compatto1} and a diagonal procedure.

Just  to simplify the notation, we assume that the $2$-scale convergence in Item (i) takes place for $\e\da 0$ (avoiding in this way to specify  continuously the subsequence $\{\e_k\}$).
Let us define $F( \bar v, \bar \varphi, \bar u ):= \int dx\,m \int \cP_0(d\o) \bar v (x,\o) \bar \varphi (x) \bar u (\o)$. Then Item  (ii) corresponds to  the limit
\begin{equation}
\label{pavimento}
\lim_{M\to \infty}  F(v_M , \varphi, u )= F(v, \varphi, u)\qquad \forall \varphi \in C_c (\bbR^d)\,,\; \forall u \in \cG\,.
\end{equation}
We fix  functions $\varphi, u$ as in \eqref{pavimento} and set $u_k:=[u]_k$ for all $k \in \bbN_+$. By definition of $
\cG$, we have $u_k \in \cG$ for all $k$ (see Section \ref{topo}).

\begin{Claim}\label{cipolla1}
For each $k,M\in \bbN_+$ it holds
\begin{align}
& | F(v,\varphi, u) - F(v, \varphi, u_k)| \leq C_0 \|\varphi \|_{L^2(mdx)} \|u-u_k\|_{L^2(\cP_0)}\,,\label{ricola1}\\
& | F(v_M,\varphi, u) - F(v_M, \varphi, u_k)| \leq C_0 \|\varphi \|_{L^2(mdx)} \|u-u_k\|_{L^2(\cP_0)}\,.\label{ricola2}
\end{align}
\end{Claim}
\begin{proof}
By Schwarz inequality 
\begin{align*}
& | F(v,\varphi, u) - F(v, \varphi, u_k)|= \Big| \int dx\, m \int d\cP_0 (\o) v(x,\o) \varphi (x) ( u-u_k) (
\o) \Big| \\
& \leq\|   v\| _{L^2 (m dx \times \cP_0) }   \|   \varphi (u-u_k) \| _{L^2 (m dx \times \cP_0) }  \,.
\end{align*}
To get \eqref{ricola1} it is then enough to apply Lemma \ref{alba} (or Lemma \ref{compatto1}) to bound $\|   v\| _{L^2 (m dx \times \cP_0) } $ by $ C_0$. The proof of \eqref{ricola2} is identical.
\end{proof}

\begin{Claim}\label{cipolla2}
For each $k,M\in \bbN_+$ it holds 
\begin{equation}\label{ricola3}
| F(v, \varphi, u_k)- F(v_M, \varphi, u_k)| \leq (k/M) \|\varphi \|_\infty C_0^2\,.
\end{equation}
\end{Claim}
\begin{proof} We note that $(v_\e -v^\e_M)(x)=0$ if $|v_\e(x)|\leq M$. Hence we can bound
\begin{equation}\label{viola} |v_\e - v^\e_M|(x) = |v_\e - v^\e_M|(x) \mathds{1}_{\{ |v_\e(x)|> M\}}\leq |v_\e - v^\e_M|(x)  \frac{ |v_\e(x)|}{M} \leq \frac{ v_\e(x)^2}{M} \,.
\end{equation}
We observe that $F(v, \varphi, u_k)=\lim _{\e\da 0} \int d \mu ^\e_{\tilde \o}(x) v_\e(x) \varphi (x) u_k (\t_{x/\e}\tilde \o) $,
   since $u_k\in \cG$ and $v_\e \stackrel{2}{\toup} v$.  A similar representation holds for $F(v_M, \varphi, u_k)$. As a consequence, and using \eqref{viola},  we get
\begin{multline*}
| F(v, \varphi, u_k)- F(v_M, \varphi, u_k)| \leq \varlimsup_{\e\da0} \int d \mu ^\e_{\tilde \o}(x)
\bigl|
(v_\e - v^\e_M)(x) \varphi(x) u_k (\t_{x/\e}\tilde \o) \bigr|\\
 \leq (k/M) \|\varphi\|_\infty \varlimsup_{\e\da 0}  \int d \mu ^\e_{\tilde \o}(x)  v_\e(x)^2\leq (k/M) \|\varphi \|_\infty C_0^2\,.
\end{multline*}
\end{proof}
We can finally conclude the proof of Lemma \ref{lemma1}. Given $\varphi\in C_c(\bbR^d)$ and $u\in \cG$, by applying Claims \ref{cipolla1} and \ref{cipolla2},  we can bound
\begin{multline}
 |F(v_M , \varphi, u )- F(v, \varphi, u)| \leq |F(v_M , \varphi, u )- F(v_M, \varphi, u_k)|\\
+|F(v_M , \varphi, u_k )- F(v, \varphi, u_k)|+|F(v , \varphi, u_k )- F(v, \varphi, u)| \\\leq
2 C_0 \|\varphi \|_{L^2(mdx)} \|u-u_k\|_{L^2(\cP_0)} +(k/M) \|\varphi \|_\infty C_0^2\,.
\end{multline}
The thesis then follows by taking first the limit $M\to \infty$ and afterwards the limit $k\to \infty$, and using   that $\lim_{k\to \infty} \|u-u_k\|_{L^2(\cP_0)}=0$.
\end{proof}
%%%%%%%%%%%%%%%%%%%%%%%%%%%%%%%%%%%%%%%%%%%%%%%%%%%%%%%%%%%%%%%%%%%%%%%%%%%%%%%%%%%%%%%%%%%%%%%%%%%%%%%%%%%%%%%%%%%%%%%%%%%%%%%%%%%%%%%%%%%%%%%%%%%%%%%%%%%%%%%%%%%%%%%%%%%%%%%%%%%%%%%%%%%%%%%%%%%%%%%%%%%%

%%%%%%%%%%%%%%%%%%%%%%%%%%%%%%%%%%%%%%%%%%%%%%%%%%%%%%%%%%%%%%%%%%%%%%%%%%%%%%%%%%%%%%%%%%%%%%%%%%%%%%%%%%%%%%%%%%%%%%%%%%%%%%%%%%%%%%%%%%%%%%%%%%%%%%%%%%%%%%%%%%%%%%%%%%%%%%%%%%%%%%%%%%%%%%%%%%%%%%%%%%%%

\section{Structure of the 2-scale weak limit  of a  bounded family in $H^1_{\o,\e}$: part I} \label{sec_risso} 
It is simple to check the following Leibniz rule for discrete gradient:
 \begin{equation}\label{leibniz} \nabla _\e  (fg)(x,z)
   =\nabla _\e  f (x, z ) g (x )+ f (x+\e z ) \nabla _\e g  ( x, z)
\end{equation}
where $f,g: \e \hat \o \to \bbR$.

The following  Proposition  \ref{risso}  is    the analogous of  \cite[Lemma 5.3]{ZP}. 

%%%%%%%%%%%%%%%%%%%%%%%%%%%%%%%%%%%%%
\begin{Proposition}\label{risso} Let $\tilde \o \in \O_{\rm typ}$. Let $\{v_\e\} $ be a  family of functions $v_\e \in H^1_{\tilde \o, \e} $ satisfying 
\begin{equation}\label{istria}
\limsup _{\e \da 0} \| v_\e\|_{L^2(\mu_{\tilde \o}^\e) } <+\infty\,,\qquad 
\limsup_{\e \da 0}  \|\nabla_\e v_\e\| _{L^2(\nu^\e_{\tilde \o})}  <+\infty \,.
\end{equation} Then, along a subsequence, we have that $v_\e \stackrel{2}{\toup} v$, where $v\in  L^2( m dx\times \cP_0)$ does not depend on $\o$: for $dx$--a.e. $x\in \bbR^d$  the function $\o \mapsto v(x,\o)$ is constant.
\end{Proposition}
\begin{proof} 
Due to Lemma \ref{compatto1} we have that $v_\e \stackrel{2}{\toup} v \in  L^2( m dx\times \cP_0)$ along a subsequence $\{\e_k\}$.
Recall the definition of the functional sets $\cG_1, \cH_1$ given in Section \ref{topo}.
We  claim that $\forall \varphi \in C^1_c (\bbR^d)$ and $\forall \psi \in \cG_1$ it holds 
\begin{equation}\label{chiavetta}
\int dx\, m \int \cP_0(\o) v (x,\o) \varphi (x) \psi(\o)=0\,.
\end{equation}
Before proving our claim, let us explain how it leads to the thesis. Since $\varphi $ varies among $C^1_c(\bbR^d)$ while $\psi$ varies in a countable set,  \eqref{chiavetta} implies that,  $dx$--a.e.,  $ \int \cP_0(\o) v (x,\o)\psi(\o)
=0$ for any $\psi \in \cG_1$. We conclude that,  $dx$--a.e., $v(x,\cdot)$ is orthogonal in $L^2(\cP_0)$ to $\{ w \in L^2(\cP_0)\,:\, \bbE_0[ w]=0\}$ (due to the density of $\cG_1$), which is equivalent to the fact that $v(x,\o)= \bbE_0[ v(x, \cdot)]$ for $\cP_0$--a.a. $\o$.

It now remains to prove \eqref{chiavetta}. Since $\tilde\o \in \O_{\rm typ}$ and due to \eqref{istria}, along  a subsequence Items (i) and (ii) of Lemma \ref{lemma1} hold (we keep the same notation of Lemma \ref{lemma1}). Hence,  in oder to prove \eqref{chiavetta}, it is enough to prove for any $M$ that, given $\varphi \in C^1_c (\bbR^d)$ and $ \psi \in \cG_1$, 
\begin{equation}\label{chiavettaM}
\int dx  \,m \int \cP_0(\o) v_M (x,\o) \varphi (x) \psi(\o)=0\,.
\end{equation}
We write $v^\e_M:= [v_\e]_M$. Since $|\nabla _\e v^\e_M|\leq |\nabla_\e v_\e| $ (cf. \eqref{r101}),  by Lemma \ref{compatto2} (using \eqref{istria}) and a diagonal procedure,  at cost to refine the subsequence $\{\e_k\}$ we have for any $M$  that 
$\nabla_\e v^\e_M \stackrel{2}{\toup} w_M \in L^2( mdx \times \nu)$, along the subsequence $\{\e_k\}$. In what follows, we understand that the parameter  $\e$ varies in $\{\e_k\}$. Note in particular that, by \eqref{rabarbaro} and since  $\tilde \o \in \O_{\rm typ}$ and $\psi \in \cG_1\subset \cG$,
\begin{equation}\label{nord1}
\text{l.h.s. of }\eqref{chiavettaM}= \lim_{\e\da 0} \int d \mu^\e_{\tilde \o} (x) v_M^\e(x) \varphi (x) \psi( \t_{x/\e}\tilde \o) \,.
\end{equation}
Let us write $\psi= g_b$ with $b \in \cH_1$ (recall \eqref{lupetto}).  By Lemma \ref{lunetta},  since $\tilde \o \in \O_{\rm typ}$ (recall (S6)),   the r.h.s. of \eqref{nord1} equals the limit as $\e\da 0$ of 
\be \label{nord2}
 -\e \int d\nu ^\e_{\tilde \o} (x,z) \nabla_\e( v_M^\e \varphi ) (x,z) b (\t_{x/\e} \tilde{\o}, z)= -\e C_1(\e)+ \e C_2(\e)\,,
\en
where (due to \eqref{leibniz})
\begin{align*}
& C_1(\e):= \int d\nu ^\e_{\tilde \o} (x,z) \nabla_\e v_M^\e (x,z) \varphi(\e x) b (\t_{x/\e} \tilde{\o}, z)\,,\\
&  C_2(\e):=  \int d\nu ^\e_{\tilde \o} (x,z) v_M^\e ( x+\e z  ) \nabla_\e \varphi( x,z) b (\t_{x/\e} \tilde{\o}, z)   \,.
\end{align*}
%To get the last identity in \eqref{nord2} we have used \eqref{leibniz}, i.e.  
%$\nabla_\e( v_M^\e \varphi ) (x,z) =\nabla_\e v_M^\e(x,z) \varphi (\e x) + v^\e_M(x+\e z) \nabla_\e\varphi (x,z)$,
% and  \eqref{micio3} in Lemma \ref{gattonaZ}. 

Due to \eqref{nord1} and \eqref{nord2}, to get \eqref{chiavettaM} we only need to show that $\lim_{\e \da 0} \e C_1(\e)=0$ and $\lim_{\e \da 0} \e C_2(\e)=0$.
Since $\nabla_\e v^\e_M \stackrel{2}{\toup} w_M$ and $b \in \cH_1$, by \eqref{yelena} we have that 
\begin{equation}\label{nord3}
\lim_{\e \da0} C_1(\e) = \int dx \, m \int d\nu (\o, z) w_M (x, \o,z) \varphi ( x) b (\o, z)\,,
\en
which is finite, thus implying that $\lim_{\e\da 0} \e C_1(\e)=0$.

 We move to $C_2(\e)$.
Let $\ell$ be such that $\varphi (x)=0$ if $|x| \geq \ell$. Fix   $\phi \in C_c(\bbR^d)$ with values in $[0,1]$, such that $ \phi(x)=1$ for $|x| \leq \ell$ and $\phi(x)=0$ for $|x| \geq \ell+1$. Since $\nabla _\e \varphi(x,z)=0$ if $|x| \geq \ell$ and $|x+\e z|\geq \ell$, by   the mean value theorem we conclude that 
\be\label{paradiso}
\bigl | \nabla _\e \varphi(x,z) \bigr | \leq \| \nabla \varphi \|_\infty |z| \bigl( \phi(x)+ \phi(x+\e z) \bigr) \,.
\en
We apply the above bound and Schwarz inequality  to $C_2(\e)$ getting
\be\label{xar}
\begin{split}
|C_2(\e)| & \leq M \| \nabla \varphi \|_\infty \int d \nu ^\e _{\tilde \o} (x,z) |z|\,|  b (\t_{x/\e}\tilde \o, z) | \bigl( \phi(x)+ \phi(x+\e z) \bigr) \\
& \leq M \|\nabla \varphi \|_\infty  A_1(\e) ^{1/2} A_2 (\e) ^{1/2}\,,
\end{split}
\en
where  (see below for explanations)  
\begin{align*}
 A_1(\e):& =\int  \nu ^\e _{\tilde \o} (x,z) |z| ^2  \bigl( \phi(x)+ \phi(x+\e z)   \bigr)= 2\int  \nu ^\e _{\tilde \o} (x,z) |z| ^2   \phi(x)^2\,,\\
 A_2(\e):&=\int \nu ^\e _{\tilde \o} (x,z)   b (\t_{x/\e}\tilde \o, z)^2  \bigl( \phi(x) + \phi(x+\e z)  \bigr)\\& = 2  \int \nu ^\e _{\tilde \o} (x,z)( b^2+ \tilde b^2 ) (\t_{x/\e}\tilde \o, z) \phi(x)^2\,.
\end{align*}
To get the second identities in the above formulas for $A_1(\e)$ and $A_2(\e)$ we have applied 
 Lemma \ref{gattonaZ}--(i) to the forms $(\o,z)\mapsto |z|^2$ and  $(\o, z)\mapsto  b ^2(\o,z)$.
 % and used that $\tilde \o \in \O_{\rm typ}$ (recall $(S7)$ and $(S8)$ in Def.~\ref{budda} and Assumption (A5)).

We now write  
\[ A_1(\e)= 2 \int d \mu ^\e _{\tilde \o}(x) \l_2( \t_{x/\e}\tilde \o) \phi(x)^2\,, \;A_2(\e)=  2 \int d \mu ^\e _{\tilde \o}(x)
\widehat{(  \tilde b^2+ b^2 )} (\t_{x/\e}\tilde \o) \phi(x)^2\,.\]
Note that the second identity follows from the fact that $\tilde \o \in \O_{\rm typ}$ (recall (S8)).
 At this point we use again that $\tilde \o \in \O_{\rm typ}$ (cf. (S7), (S9) and (S10)). Due to Prop. \ref{prop_ergodico} we conclude that $A_1 (\e), A_2(\e)$ have finite limits as $\e\da 0$, thus implying (cf. \eqref{xar}) that $\lim_{\e\da 0} \e C_2(\e)=0$. This concludes the proof of \eqref{chiavettaM}.
\end{proof}

%%%%%%%%%%%%%%%%%%%%%%%%%%%%%%%%%%%%%%%%%%%%%%%%%%%%%%%%%%%%%%%%%%%%%%%%%%%%%%%%%%%%%%%%%%%%%%%%%%%%%%%%%%%%%%%%%%%%%%%%%%%%%%%%%%%%%%%%%%%%%%%%%%%%%%%%%%%%%%%%%%%%%%%%%%%%%%%%%%%%%%%%%%%%%%%%%%%%%%%%%%%%%%%%%%%%%%%%%%%%%%%%%%%%%%%%%%%%%%%%%%%%%%%%%%%%%%%%%%%%%%%%%%%%%%%%%%%%%%%%%%%%%%%%%%%%%%%%%%%%%%%%%%%%%%%%%%%%%%%%%%%%%%%%%%%%%%%%%%%%%%%%%%%%%%%%%%%%%%%%%%%%%%%%%%

\section{Cut-off for gradients $\nabla_\e v_\e$} \label{cut-off2}
%\rosso{Qui $\cG$ e $\cH$ sono come nelle note scritte a mano}

%\rosso{\loz $\cH$ is a countable subset of $L^2(\nu)$ for which \eqref{yelena} holds. We need that, if $b\in \cH$, then $[b]_k\in \cH$ for all $k\in \bbN_+$.}

\begin{Lemma}\label{lemma2}
Let $\tilde \o \in \O_{\rm typ}$ and let $\{v_\e\}$ be a family of functions  with    $v_\e\in H^1_{\tilde \o, \e} $, satisfying \eqref{istria}.
%\begin{equation}\label{helena}
%\limsup _{\e \da 0} \| v_\e\|_{L^2(\mu_{\tilde \o}^\e) } <+\infty\,,\qquad 
%\limsup_{\e \da 0}  \|\nabla v_\e\| _{L^2(\nu^\e_{\tilde \o})}  <+\infty \,.
%\end{equation}
 Then  there exist functions $w, w_M\in L^2 (m dx \times \nu) $, with $M$ varying in $\bbN_+$, such that 
\begin{itemize}
\item[(i)] $\nabla_\e  v_\e \stackrel{2}{ \toup}  w$ and $\nabla_\e[ v_\e ]_M\stackrel{2}{ \toup}  w_M$ for all $M\in \bbN_+$;
\item[(ii)] for any $\varphi \in C^1_c(\bbR^d) $ and $b\in \cH$ it holds
\begin{multline}\label{queen2}
\lim_{M \to \infty}\int dx \,m \int d\nu  (\o,z) w_M(x,\o,z) \varphi (x) b(\o,z)\\
=
\int dx\,  m \int d \nu(\o,z) w(x,\o,z) \varphi(x) b(\o,z)\,.
\end{multline}
\end{itemize}
\end{Lemma}
\begin{proof}
At cost to restrict to $\e$ small enough, we can assume that $ \| v_\e\|_{L^2(\mu_{\tilde \o}^\e) } \leq C_0$ and 
$ \|\nabla _\e v_\e\| _{L^2(\nu^\e_{\tilde \o})} \leq C_0$ for some $C_0<+\infty$ and all $\e>0$. Due to \eqref{r101}, the same holds respectively for $v^\e_M$ and $\nabla_\e v^\e_M$, for all $M\in \bbN_+$, where we have set $v^\e_M:= [ v_\e]_M$. In particular, by a diagonal procedure, due to Lemmas \ref{compatto1} and \ref{compatto2} along a subsequence  we have  that $v^\e_M \stackrel{2}{\toup} v_M$,  $v_\e  \stackrel{2}{\toup} v$, $\nabla_\e v^\e_M \stackrel{2}{\toup} w_M$ and $\nabla_\e v_\e \stackrel{2}{\toup} w$,  where $v_M,v\in L^2( m dx \times \cP_0)$, $w_M,w\in L^2( m dx \times \nu)$, simultaneously for all $M\in \bbN_+$. This proves in particular Item (i). We point out that we are not claiming that $v_M=[v]_M$, $w_M=[w]_M$. Moreover, from now on we restrict to $\e$ belonging to the above special subsequence without further mention.

\smallskip

We set $H(\bar w, \bar \varphi, \bar b):= \int dx m \int d\nu(\o,z) \bar w (x, \o, z) \bar \varphi (x) \bar b (\o,z)$. Then \eqref{queen2} corresponds to the limit $\lim _{M\to \infty} H(w_M, \varphi, b)=  H(w, \varphi, b)$. Here and below $b \in \cH$ and $\varphi \in C^1_c(\bbR^d)$. Recall that $b_k:=[b]_k \in \cH$ for any $k \in \bbN_+$ (see Section \ref{topo}).

\medskip

Reasoning exactly as in the proof of Claim \ref{cipolla1} we get the following bounds: 
\begin{Claim}\label{utti1}
For each $k,M\in \bbN_+$ it holds
\begin{align}
& | H(w,\varphi, b) - H(w, \varphi, b_k)| \leq C_0 \|\varphi \|_{L^2(mdx)} \|b-b_k\|_{L^2(\nu )}\,,\label{leoncino1}\\
& | H(w_M,\varphi, b) - H(w_M, \varphi, b_k)| \leq C_0 \|\varphi \|_{L^2(mdx)} \|b-b_k\|_{L^2(\nu)}\,.\label{leoncino2}
\end{align}
\end{Claim}

\begin{Claim}\label{utti2} For any $k \in \bbN_+$, it holds
\begin{equation}
| H(w, \varphi, b_k)- H( w_M, \varphi, b_k) | \leq \frac{k}{\sqrt{M}} C_0^{3/2} C(\varphi)\,,
\end{equation}
where $C(\varphi)$ is a positive constant depending only on $\varphi$.
\end{Claim}
\begin{proof} In what follows $C(\varphi)$ is a positive constant, depending at most on $\varphi$, which can change from line to line. We note that $\nabla_\e v_\e (x,z) = \nabla_\e v_M^\e(x,z)$ if $|v_\e(x)|\leq M$ and $|v_\e(x+\e z)|\leq M$. Moreover, by \eqref{r102}, we have  $ | \nabla _\e v_\e - \nabla_\e v_M^\e|\leq | \nabla _\e v_\e|$. Hence we can bound
\begin{equation}
\bigl| \nabla_\e v_\e - \nabla_\e v^\e_M\bigr| (x,z) \leq | \nabla_\e v_\e|(x,z)\bigl( \mathds{1} _{\{ |v_\e(x)|\geq M\}} + \mathds{1} _{\{ |v_\e(x+\e z) |\geq M\} }\bigr)\,.
\end{equation}
Due to the above bound  we can estimate (see comments below)
 \begin{equation}\label{mat}
\begin{split}
& | H(w, \varphi, b_k)- H(w_M, \varphi, b_k) | \\
&= \bigl| 
\lim_{\e \da 0} \int d \nu^\e_{\tilde \o} (x,z) \bigl( \nabla_\e v_\e - \nabla_\e v^\e_M \bigr) (x,z) \varphi(x) b_k( \t_{x/\e} \tilde \o, z) \bigr|\\
&\leq k\, \varlimsup_{\e \da 0} \int d \nu^\e_{\tilde \o} (x,z) | \nabla_\e v_\e|(x,z)\bigl( \mathds{1} _{\{ |v_\e(x)|\geq M\}}+ \mathds{1} _{\{ |v_\e(x+\e z) |\geq M\} }\bigr) |\varphi(x) |\,.
%\\ &\leq  2 k\, \varlimsup_{\e \da 0} \int d \nu^\e_{\tilde \o} (x,z) | \nabla_\e v_\e|(x,z)  \mathds{1} _{\{ |v_\e(x)|\geq M\}}  \,.
\end{split}
\end{equation}
Note that the identity in \eqref{mat}  follows  from \eqref{yelena} since $b_k \in \cH$ (recall that  $\tilde \o \in \O_{\rm typ}$,  $\nabla_\e v^\e_M \stackrel{2}{\toup} w_M$,  $\nabla_\e v_\e \stackrel{2}{\toup} w$).
By Schwarz inequality  we have 
\be\label{mat1}
 \int d \nu^\e_{\tilde \o} (x,z) | \nabla_\e v_\e|(x,z) \mathds{1} _{\{ |v_\e(x)|\geq M\}}  |\varphi(x) |\leq C_0 A(\e)^{1/2}\,,
\en
where, by applying a Chebyshev-like estimate and Schwarz inequality,
\begin{equation*}
\begin{split}
A(\e):& = \int  d \nu^\e_{\tilde \o} (x,z)  \mathds{1} _{\{ |v_\e(x)|\geq M\}} \varphi(x) ^2 
\leq  M^{-1} \int d \mu ^\e _{\tilde \o} (x) |v_\e (x)| \varphi (x) ^2 \l_0 (\t_{x/\e} \tilde \o)\\
& \leq M^{-1} \| v_\e\|_{ L^2(  \mu ^\e _{\tilde \o} )} \Big[   \int d \mu ^\e _{\tilde \o} (x)  \varphi(x)^4 \l_0 ^2 ( \t_{x/\e} \tilde \o) 
\Big]^{1/2}\,.
\end{split}
\end{equation*}
Since $\tilde \o \in \O_{\rm typ}$ (cf. (S11), (S12)  and Prop. \ref{prop_ergodico}),  $  \int d \mu ^\e _{\tilde \o} (x)  \varphi(x)^4 \l_0 ^2 ( \t_{x/\e} \tilde \o) $ has  finite limit as $\e \da 0$.
As a consequence, we get 
\be
\label{galley1}
\varlimsup_{\e \da 0} A(\e) \leq ( C_0/M) C(\varphi)\,.
\en
Reasoning as above we have 
\be\label{mat2}
 \int d \nu^\e_{\tilde \o} (x,z) | \nabla_\e v_\e|(x,z) \mathds{1} _{\{ |v_\e(x+ \e z )|\geq M\}}  |\varphi(x) |\leq C_0 B(\e)^{1/2}\,,
\en
where (applying also  \eqref{micio2} for the map $(\o,z) \mapsto 1$) 
\begin{equation*}%\label{navetta}
\begin{split}
B(\e):& = \int  d \nu^\e_{\tilde \o} (x,z)  \mathds{1} _{\{ |v_\e(x+ \e z )|\geq M\}} \varphi(x) ^2  \leq  \frac{1}{M} \int d \nu ^\e _{\tilde \o} (x,z ) |v_\e (x+ \e z )| \varphi (x) ^2\\
& =
 \frac{1}{M} \int d \nu ^\e _{\tilde \o} (x,z ) |v_\e (x )| \varphi (x+\e z) ^2=\frac{\e ^d}{M} \sum_{y \in  \widehat {\tilde \o}} | v_\e(\e y) | \sum_{a \in \widehat {\tilde \o} } c_{y,a}(\tilde \o)  \varphi (\e a ) ^2
  \,.\end{split}
\end{equation*}
Due to Schwarz inequality, we have therefore that $B(\e) \leq ( C_0/M) C(\e)^{1/2}$ where 
\begin{multline*}
C(\e):= \e^d \sum_{y \in \widehat{\tilde  \o}}\Big [ \sum_{a \in \widehat{\tilde \o}} c_{y,a}(\tilde \o)  \varphi (\e a ) ^2\Big]^2\\
\leq \| \varphi\|_\infty^2 \e^d \sum_{y \in \widehat{\tilde  \o}}\sum_{a \in \widehat{\tilde  \o}}\sum_{e \in \widehat{\tilde  \o}}
c_{y,a} (\tilde \o) \varphi (\e a)^2 c_{y,e}(\tilde \o) = \| \varphi \|_\infty^2 \e^d \sum_{a \in \widehat{\tilde  \o}} F_*( \t _a \tilde \o)\varphi ( \e a )^2 \,,\end{multline*}
where $ F_*(\o) := \int d \hat \o (y) \int d \hat \o (z) c_{0,y} (\o) c_{y,z}(\o)$ as in (A6).
The r.h.s. converges to a finite constant as $ \e \da 0$ since  $\tilde \o \in \O_{\rm typ}$ (recall (S13) and Prop. \ref{prop_ergodico}). We therefore conclude that $\varlimsup_{\e \da 0} C_\e \leq C(\varphi)$. Since $B(\e) \leq ( C_0/M) C(\e)^{1/2}$, we get that $\varlimsup_{\e \da 0} B(\e) \leq ( C_0/M) C(\varphi)$. Since the same holds for $A(\e)$ (cf. \eqref{galley1}), due to \eqref{mat}, \eqref{mat1} and \eqref{mat2} we get the claim.
\end{proof}

We can finally derive \eqref{queen2}, i.e. that $\lim _{M\to \infty} H(w_M, \varphi, b)=  H(w, \varphi, b)$.
By using Claims \ref{utti1} and \ref{utti2} we have 
\begin{equation*}
\begin{split}
| H(w_M, \varphi , b)- H(w, \varphi, b)| \leq |H( w_M, \varphi, b)- H( w_M , \varphi, b_k)| +\\
| H(w_M, \varphi, b_k)- H( w, \varphi, b_k)| + | H(w, \varphi, b_k)- H( w, \varphi, b)|\\
\leq C_0 C(\varphi) \| b-b_k\|_{L^2(\nu) }+ C_0 ^{3/2} C(\varphi) ( k/\sqrt{M} )\,.
\end{split}
\end{equation*}
At this point it is enough to take first the limit $M\to \infty$ and afterwards the limit $k\to \infty$ and to use that  $\lim_{k \to \infty}\| b-b_k\|_{L^2(\nu) }=0$.
\end{proof}
%%%%%%%%%%%%%%%%%%%%%%%%%%%%%%%%%%%%%%%%%%%%%%%%%%%%%%%%%%%%%%%%%%%%%%%%%%%%%%%%%%%%%%%%%%%%%%%%%%%%%%%%%%%%%%%%%%%%%%%%%%%%%%%%%%%%%%%%%%%%%%%%%%%%%%%%%%%%%%%%%%%%%%%%%%%%%%%%%%%%%%%%%%%%%%%%%%%%%%%%%%%%%%%%%%%%%%%%%%%%%%%%%%%%%%%%%%%%%%%%%%%%%%%%%%%%%%%%%%%

%%%%%%%%%%%%%%%%%%%%%%%%%%%%%%%%%%%%%

\section{Structure of the 2-scale weak limit  of a  bounded family in $H^1_{\o,\e}$: part II} \label{sec_oro}
The next result is the analogous of \cite[Lemma 5.4]{ZP}.
\begin{Proposition}\label{oro}
%Suppose $q$ to be a non--degenerate form. 
Let $\tilde \o \in \O_{\rm typ}$
 and let $\{v_\e\} $ be a  family of functions $v_\e \in H^1_{\tilde \o, \e} $ uniformly bounded in  $H^1_{\tilde \o, \e} $, i.e. satisfying \eqref{istria}.
Then, along a subsequence $\{ \e_k\}$, we have:
\begin{itemize}
\item[(i)]
$v_\e \stackrel{2}{\toup} v$, where $v\in  L^2(  m dx\times \cP_0)$ does not depend on $\o$.  Writing $v$ simply as $v(x)$ we have that  \rosso{$v\in H^1 _*( m dx )$};
\item[(ii)] $\nabla v_\e  (x,z) \stackrel{2}{\toup}\rosso{  \nabla_*  }v (x) \cdot z + v_1 (x,\o,z)$,
where $v_1\in L^2\bigl( \bbR^d, L^2_{\rm pot} (\nu)\bigr )$.
\end{itemize}
\end{Proposition}

% kuka \loz \rosso{The statement is for non--degenerate $q$. Va generalizzato??}\\

The property $v_1\in L^2\bigl( \bbR^d, L^2_{\rm pot} (\nu)\bigr )$ means that 
for $dx$--almost every $x$ in $\bbR^d$ the map $(\o,z)\mapsto v_1(x, \o, z) $ is a potential form, hence in  $L^2_{\rm pot} (\nu)$, moreover the map  $\bbR^d \ni x \to v_1(x, \cdot, \cdot) \in L^2_{\rm pot} (\nu)$ is measurable and 
\begin{equation}
\int  dx \| v_1(x, \cdot, \cdot)\|_{L_2(\nu) }^2=
\int dx  \int  \nu (\o,z) v_1 (x, \o, z)^2 <+\infty\,.
\end{equation}

\begin{proof}[Proof of Prop.~\ref{oro}] 
At cost to restrict to $\e$ small enough, we can assume that $ \| v_\e\|_{L^2(\mu_{\tilde \o}^\e) } \leq C_0$ and 
$ \|\nabla _\e v_\e\| _{L^2(\nu^\e_{\tilde \o})} \leq C_0$ for some $C_0<+\infty$ and all $\e>0$.  We can assume the same bounds for $v^\e_M:=[v_\e]_M$.  
Along a subsequence the $2$-scale convergences in Item (i) of Lemma \ref{lemma1} and in Item (i) of Lemma \ref{lemma2} take place.   By Lemmas \ref{compatto1} and \ref{compatto2}  the norms $\| v_M\|_{L^2(mdx \times \cP_0)}$, $\| v\|_{L^2(m dx \times \cP_0)}$, $\| w_M\|_{L^2(m dx \times \nu)}$ and $\| w\|_{L^2(m dx\times \nu)}$ are upper bounded by $C_0$. 

Due to Prop.~\ref{risso} $v=v(x)$ and $v_M=v_M(x)$. We claim that 
 for   each  solenoidal form $b  \in L^2_{\rm sol}(\nu) $ and each  function $\varphi \in C^2_c(\bbR^d)$, it holds 
 \begin{equation}\label{kokeshi}
 \int  dx   \varphi(x) \int  d \nu (\o ,z) w(x,\o,z) b(\o, z) =  
  -\int   dx  v(x) \nabla \varphi(x) \cdot \eta_b\,, 
\end{equation}
where  $\eta_b := \int d\nu (\o,z) z b(\o,z)$.
Note that $\eta_b$ is well defined since both $b$ and the map $(\o,z) \mapsto z$ are in $ L^2(\nu)$ (cf.  (A5)). Moreover, by  applying Lemma \ref{lucertola1} with $k(\o, \t_z \o):= c_{0,z} (\o ) z b(\o,z) $ (cf. (A3) and \eqref{SS3}), we get that $\eta_b = -\eta_{\tilde b}$ (cf. Def.~\ref{ometto}).

Before proving our Claim \eqref{kokeshi} we show how to conclude the proof of Prop.~\ref{oro} starting with Item (i).
\rosso{Due to Corollary \ref{jack} for each $i=1,\dots, d_*$ there exists  $b_i \in L^2_{\rm sol}(\nu)$ such that $\eta_{b_i}= e_i$},
$e_i$ being the $i$--th vector of the canonical basis of $\bbR^d$. Consider the measurable function 
  \begin{equation}
  g_i (x):=  \int   d \nu (\o ,z) w(x,\o,z) b_i(\o, z)\,, \qquad \rosso{ 1\leq i \leq d_*}\,.
  \end{equation}
We have that $g_i \in L^2(dx) $ since, by Schwarz inequality, 
\begin{multline}
\int g_i(x) ^2 dx =
   \int dx \left[ \int  d \nu (\o ,z) w(x,\o,z) b_i(\o, z)\right]^2\\
\leq  \|b _i \|^2_{L^2(\nu)}\int dx   \ \int  d \nu (\o ,z) w(x,\o,z) ^2 <\infty\,.
\end{multline}
   Moreover, by \eqref{kokeshi} we have that $\int dx \varphi(x) g_i(x) =  -\int   dx\, v(x) \partial_i \varphi(x) $ \rosso{for $1\leq i \leq d_*$}.  This proves that $v(x)\in \rosso{H^1_*(mdx)}$ and   $\partial_i v(x)= 
g_i(x)$ \rosso{ for $1\leq i \leq d_*$}.
  This concludes the proof of Item (i).
  
  We move to Item (ii) (always assuming 
  \eqref{kokeshi}). By Item (i) \rosso{and Corollary \ref{jack} implying that $\eta_b \in \text{span}\{e_1, \dots, e_{d_*}\} $ for all $b\in L^2_{\rm sol}(\nu)$},  we can  replace  the r.h.s. of \eqref{kokeshi}  by $ \int  dx  (\rosso{\nabla_*} v(x) \cdot \eta_b ) \varphi(x)$. Hence   \eqref{kokeshi} can be rewritten as 
 \begin{equation}\label{cocco}
 \int  dx  \varphi(x) \int   d \nu (\o ,z)
 \left[ w(x,\o,z) -\rosso{\nabla_*} v(x)  \cdot z
 \right] b(\o, z) =  
  0\,.
  \end{equation}
By the arbitrariness of $\varphi$ we conclude that $dx$--a.s. 
\begin{equation}\label{criceto13}
\int  d \nu (\o ,z)
 \left[ w(x,\o,z) -\rosso{\nabla_*} v(x)  \cdot z
 \right] b(\o, z) =  
  0\,, \qquad \forall b \in L^2 _{\rm sol}(\nu)\,.
\end{equation}
%Let us note that, $dx$--a.s., $ w(x,\o,z) -\nabla v(x)  \cdot z\in L^2(\nu)$.
Let us now show that the map 
$w(x,\o,z) -\rosso{\nabla_*} v(x) \cdot z $ belongs to $L^2( dx, L^2(\nu) )$. Indeed,  we have 
$
\int dx \|w(x, \cdot, \cdot)\|_{L^2(\nu) }^2 =
 \| w\|^2_{ L^2( mdx \times d\nu)}<+\infty$ and also
 \begin{equation}
\int dx \|\rosso{\nabla_*} v(x) \cdot z \|_{L^2(\nu) }^2 \leq \int dx | \rosso{\nabla_*} v(x) |^2  \int d\nu (\o, z) |z|^2 
% = \int dx | \nabla v(x) |^2 \bbE_0[\l_2]
<\infty\,,
\end{equation}
  by Schwarz inequality
 and since $\rosso{\nabla_*} v\in L^2(dx)$ and $\bbE_0[\l_2]<\infty$ by (A5).

As the map 
$w(x,\o,z) -\rosso{\nabla_*} v(x) \cdot z $ belongs to $L^2( dx, L^2(\nu) )$, for $dx$--a.e. $x$ we have that the map 
  $(\o,z) \mapsto w(x,\o,z) -\rosso{\nabla_*} v(x) \cdot z$ belongs to $ L^2(\nu)$ and therefore, by \eqref{criceto13}, to   $ L^2_{\rm pot} (\nu)$.
This concludes the proof of Item (ii).

\smallskip
It remains to prove \eqref{kokeshi}. Here is a roadmap: (i) we reduce  \eqref{kokeshi}  to \eqref{kokeshiM}; (ii) we prove  \eqref{provence}; (iii) 
by \eqref{provence} we reduce 
\eqref{kokeshiM} to \eqref{tramonto}; (iv) we prove \eqref{cuoricino1}; (v) by \eqref{cuoricino1} we reduce \eqref{tramonto} to \eqref{tramontoA}; (vi) we prove \eqref{tramontoA}. 

Since both sides of \eqref{kokeshi} are continuous as functions of $b \in L^2_{\rm sol}(\nu)$, it is enough to prove it for $b\in \cW$ (see Section~\ref{topo}). We apply Lemma \ref{lemma2}--(ii) (recall that $b\in \cW\subset \cH$) to approximate the l.h.s. of \eqref{kokeshi} and Lemma \ref{lemma1}--(ii) with $u:=1\in \cG$ to approximate the r.h.s. of \eqref{kokeshi}. Then to prove \eqref{kokeshi}
it is enough to show that 
  \begin{equation}\label{kokeshiM}
 \int  dx \,m  \varphi(x) \int  d \nu (\o ,z) w_M(x,\o,z) b(\o, z) =  
  -\int   dx \,m v_M(x) \nabla \varphi(x) \cdot \eta_b\,, 
\end{equation}
 for any $\varphi \in C_c^2(\bbR^d)$, $b \in \cW$ and $M\in \bbN_+$. From now on $M$ is fixed.

 \smallskip
  Since $\tilde \o \in \O_{\rm typ}$, $\nabla_\e v^\e_M\stackrel{2}{\toup} w_M$ and $b \in \cW \subset \cH$ (cf.~\eqref{yelena}) we can write
  \be\label{bruna1}
  \text{l.h.s. of }\eqref{kokeshiM}= \lim _{\e \da 0} \int d \nu ^\e _{\tilde \o}(x,z) \nabla_\e v^\e _M(x,z) \varphi (x) b ( \t_{x/\e} \tilde \o, z)\,.
  \en
  Since $b \in L^2_{\rm sol}(\nu)$ and $\tilde \o \in \O_{\rm typ}$ (cf. (S14), Lemmata \ref{tav} and    \ref{lunetta}), we get
  \[
  \int d \nu ^\e _{\tilde \o}(x,z) \nabla_\e ( v^\e  _M \varphi ) (x,z)  b ( \t_{x/\e} \tilde \o, z)=0\,.
  \]
Using the above identity, \eqref{leibniz}  and finally \eqref{micio3} in Lemma \ref{gattonaZ} (as $\tilde \o \in \O_{\rm typ}$ and due to (S6)), we conclude that 
\be\label{bruna2}
\begin{split}
\int d \nu ^\e _{\tilde \o}(x,z) &\nabla_\e v^\e _M(x,z) \varphi (x) b ( \t_{x/\e} \tilde \o, z)\\
&=- \int d \nu ^\e _{\tilde \o}(x,z)  v^\e _M(x+ \e z) \nabla_\e \varphi (x,z) b ( \t_{x/\e} \tilde \o, z)
\\
&  =\int d \nu ^\e _{\tilde \o}(x,z)  v^\e _M(x) \nabla_\e \varphi (x,z) \tilde b ( \t_{x/\e} \tilde \o, z)\,.
\end{split}
\en
Up to now we have obtained that 
\be\label{alba1}
  \text{l.h.s. of }\eqref{kokeshiM}= \lim _{\e \da 0} \int d \nu ^\e _{\tilde \o}(x,z)  v^\e _M(x) \nabla_\e \varphi (x,z) \tilde b ( \t_{x/\e} \tilde \o, z)\,.
    \en
We now set $\tilde b_k:= [\tilde b ]_k= \widetilde{b_k}$. 
We  want to prove that 
\be\label{provence}
\varlimsup_{k \uparrow \infty} \varlimsup  _{\e\da 0}
 \bigl| \int d \nu ^\e _{\tilde \o}(x,z)  v^\e _M(x) \nabla_\e \varphi (x,z)( \tilde b- \tilde b_k) ( \t_{x/\e} \tilde \o, z)\bigr|
 =0\,.
%\int d \nu ^\e _{\tilde \o}(x,z)  v^\e _M(x) \nabla_\e \varphi (x,z) \tilde b ( \t_{x/\e} \tilde \o, z)\\=\lim_{k \to \infty}
%\int d \nu ^\e _{\tilde \o}(x,z)  v^\e _M(x) \nabla_\e \varphi (x,z) \tilde b _k( \t_{x/\e} \tilde \o, z)\,.
\en
%To this aim
To this aim let $\ell$ be such that $\varphi (x)=0$ if $|x| \geq \ell$. Fix   $\phi \in C_c(\bbR^d)$ with values in $[0,1]$, such that $ \phi(x)=1$ for $|x| \leq \ell$ and $\phi(x)=0$ for $|x| \geq \ell+1$. 
  Using \eqref{paradiso} and Schwarz inequality  we can bound
\be \label{acqua}
\begin{split}
& \bigl| \int d \nu ^\e _{\tilde \o}(x,z)  v^\e _M(x) \nabla_\e \varphi (x,z)( \tilde b- \tilde b_k) ( \t_{x/\e} \tilde \o, z)\bigr|
\\ 
&\leq  M \| \nabla \varphi \|_\infty \int d \nu ^\e _{\tilde \o}(x,z) |z| \bigl( \phi (x) + \phi (x+ \e z) \bigr) | \tilde b- \tilde b_k| ( \t_{x/\e} \tilde \o, z)\\
& \leq M \| \nabla \varphi \|_\infty [ 2 A(\e) ] ^{1/2} [ B(\e,k)+ C(\e,k)]^{1/2}\,
\end{split}
\en
where (using  \eqref{micio2} in Lemma \ref{gattonaZ}  for $A(\e)$ and $C(\e)$)
% ,  $\tilde \o \in \O_{\rm typ}$, (S7) and (S15))
\begin{align*}
 A(\e): &= \int d \nu ^\e _{\tilde \o}(x,z) |z|^2 \phi(x)= \int d \nu ^\e _{\tilde \o}(x,z) |z|^2 \phi(x+ \e z)\,,\\
 B(\e,k): & =  \int d \nu ^\e _{\tilde \o}(x,z)  ( \tilde b- \tilde b_k)^2 ( \t _{x/\e} \tilde \o, z) \phi (x) \,,\\
 C(\e,k):& =  \int d \nu ^\e _{\tilde \o}(x,z)  ( \tilde b- \tilde b_k)^2 ( \t _{x/\e} \tilde \o, z) \phi (x+\e z)\\
& =  \int d \nu ^\e _{\tilde \o}(x,z)  ( b-   b_k)^2 ( \t _{x/\e} \tilde \o, z) \phi (x)\,.
\end{align*}
% \end{proof}
 Due to (S7), (S9) and Prop.~\ref{prop_ergodico}   $A(\e)= \int d \mu^\e_{\tilde \o} (x) \phi(x) \l_2 ( \t_{x/\e} \tilde \o) $ has finite limit as $\e\da 0$. 
Hence to get \eqref{provence} we only need to show that $\lim_{k \uparrow  \infty, \e \da 0 } B(\e,k)= \lim _{k \uparrow  \infty, \e \da 0 }C(\e, k)=0$. We can write
$
B(\e,k)= 
\int d \mu^\e _{\tilde \o} (x) \phi(x)
\widehat{d}^2  (\t_{x/\e} \tilde \o )$ where $ d:=   | \tilde b- \tilde b_k|
$. Since $\tilde b \in \cW$ for any $b\in \cW$ (see Section \ref{topo}),
due to (S5), (S8),  the property $\tilde \o \in \O_{\rm typ}$ and  Prop.~\ref{prop_ergodico}, we conclude that
$
\lim_{\e \da 0} B(\e,k)= \int dx \, m \phi (x)  \|  \tilde b- \tilde b_k\|_{L^2(\nu)}^2 $ (cf. Lemma \ref{cavallo}).
Similarly we get that $
\lim_{\e \da 0} C(\e,k)= \int dx \, m \phi (x)  \|   b-  b_k\|_{L^2(\nu)}^2 $.  As the above limits go to zero as $k\to \infty$, we get  \eqref{provence}.

Due to \eqref{alba1}, \eqref{provence} and since, by Schwarz inequality, $\lim _{k \to \infty} \eta_{\tilde b_k}= \eta _{\tilde b}= - \eta_b$, to prove \eqref{kokeshiM} we only need to show, for fixed $M,k$,  that 
 \begin{equation}\label{tramonto}
 \lim _{\e \da 0} \int d \nu ^\e _{\tilde \o}(x,z)  v^\e _M(x) \nabla_\e \varphi (x,z) \tilde b _k( \t_{x/\e} \tilde \o, z)=
 \int   dx \, m v_M(x) \nabla \varphi(x) \cdot \eta_{\tilde b_k}\,.
 \end{equation}
 To prove \eqref{tramonto} we first show that 
 \begin{equation}\label{cuoricino1} \lim_{\e\da 0} \Big |
  \int d \nu ^\e _{\tilde \o}(x,z)  v^\e _M(x) \bigl[ \nabla_\e \varphi (x,z) -\nabla \varphi(x) \cdot z\bigr]\tilde b _k( \t_{x/\e} \tilde \o, z)
  \Big |=0\,.
 \end{equation}
 Since $\|v^\e _M\|_\infty \leq M $ and $\| \tilde b_k \| _\infty \leq k$, it is enough to show that 
 \be \label{cuoricino2}
 \lim _{\e\da 0}  \int d \nu ^\e _{\tilde \o}(x,z) \bigl| \nabla_\e \varphi (x,z) -\nabla \varphi(x) \cdot z\bigr| =0\,.
 \en
Since $\varphi \in C_c^2(\bbR^d)$,  by Taylor expansion we have $\nabla_\e \varphi (x,z) -\nabla \varphi(x) \cdot z=\frac{1}{2} \sum_{i,j} \partial^2 _{ij} \varphi(\, \z_\e(x,z)\,) z_i z_j \e$, where $\z_\e(x,z)$ is a point between $x$ and $x+\e z$. Moreover we note that 
$\nabla_\e \varphi (x,z) -\nabla \varphi(x) \cdot z=0$ if $|x| \geq \ell$ and $|x+\e z| \geq \ell$. All these observations imply that 
\be \label{mirra}
\bigl| \nabla _\e \varphi (x,z) - \nabla \varphi (x) \cdot z\bigr | \leq \e C(\varphi) |z|^2 \bigl( \phi(x) + \phi(x+\e z) \bigr)\,.
\en
Due to \eqref{micio2}  we can write
\[ \int d \nu ^\e _{\tilde \o}(x,z)  |z|^2  \phi(x+\e z)=
 \int d \nu ^\e _{\tilde \o}(x,z)  |z|^2  \phi(x)= \int d \mu^\e_{\tilde \o} (x) \phi(x) \l_2 ( \t_{x/\e} \tilde \o)\,.
 \]
 Due to (S7) and (S9)  we conclude that the  above r.h.s. has a finite limite as $\e \da 0$.  Due to \eqref{mirra}, we finally get \eqref{cuoricino2} and hence
 \eqref{cuoricino1}.
 
 Having \eqref{cuoricino1}, to get \eqref{tramonto} it is enough to show that 
  \begin{equation}\label{tramontoA}
 \lim _{\e \da 0} \int d \nu ^\e _{\tilde \o}(x,z)  v^\e _M(x) \nabla \varphi (x) \cdot z \tilde b _k( \t_{x/\e} \tilde \o, z)=
 \int   dx \,m v_M(x) \nabla \varphi(x) \cdot \eta_{\tilde b_k}\,.
 \end{equation}
 To this aim we observe that
 \be\label{favorita1}
 %\begin{split}
\int d \nu^\e_{\tilde \o} (x,z) v_M^\e (x) \partial _i \varphi (x) z_i \tilde{b}_k (\t_{x/\e} \tilde\o, z)=
\int d \mu^\e_{\tilde \o}(x)  v^\e_M(x) \partial _i \varphi (x) u_k( \t_{x/\e} \tilde \o)\,,
\en
where $u_k(\o):= 
\int d \hat{\o} (z) c_{0,z}(\o) z_i \tilde{b}_k (\t_{x/\e} \tilde \o, z)$.
% Since $|u_k(\o) | \leq k \l_1(\o)$ and $\l_1 \in L^2(\cP_0)$, we have that $u_k \in L^2(\cP_0)$. 
Since $\tilde \o \in \O_{\rm typ}$, $v^\e_M\stackrel{2}{\toup} v_M$ and $u_k \in \cG$ (cf. \eqref{nonfamale} and recall that $\tilde b \in \cW$  $\forall b\in \cW$), by \eqref{rabarbaro} we conclude that 
 \be\label{favorita2}
 \begin{split}
 \lim_{\e \da 0}
 \int d \mu^\e_{\tilde \o}(x)  v^\e_M(x) \partial _i \varphi (x) u_k( \t_{x/\e} \tilde \o)& = \int dx \, m v_M(x) \partial_i \varphi (x) 
 \int d\cP_0(\o) u_k (\o)\\
 & =  \int dx \, m v_M(x) \partial_i \varphi (x) 
 (\eta_{\tilde b_k} \cdot e_i)\,,
 \end{split}
 \en
 $e_1, \dots, e_d$ being the canonical basis of $\bbR^d$. Our target \eqref{tramontoA} then follows as a byproduct of \eqref{favorita1} and \eqref{favorita2}.
 \end{proof}
%Hence to get \eqref{provence} we only need to show that $\lim_{k \to \infty} B(\e,k)= \lim _{k \to \infty} C(\e, k)=0$.
%%%%%%%%%%%%%%%%%%%%%%%%%%%%%%%%%%%%%
%%%%%%%%%%%%%%%%%%%%%%%%%%%%%%%%%%%%%
%%%%%%%%%%%%%%%%%%%%%%%%%%%%%%%%%%%%%
%%%%%%%%%%%%%%%%%%%%%%%%%%%%%%%%%%%%%
%%%%%%%%%%%%%%%%%%%%%%%%%%%%%%%%%%%%%
%%%%%%%%%%%%%%%%%%%%%%%%%%%%%%%%%%%%%
%%%%%%%%%%%%%%%%%%%%%%%%%%%%%%%%%%%%%
%%%%%%%%%%%%%%%%%%%%%%%%%%%%%%%%%%%%%
\section{Proof of Theorem \ref{teo1}}\label{robot}

Without loss of generality, we prove Theorem \ref{teo1}
 with $\l=1$ to simplify the notation.
Due to Prop.~\ref{stemino} we only need to prove Items (i), (ii) and (iii). 
  Some arguments below are taken from \cite{ZP}, others are intrinsic to long jumps.
 We start with two results  (Lemmas \ref{ascoli} and \ref{blocco}) concerning the amorphous gradient $\nabla_\e\varphi$ for $\varphi \in C_c(\bbR^d)$.
 
\begin{Lemma}\label{ascoli} Let $\o \in \O_{\rm typ}$.
%\in \cA[\l_2]\cap \cA_1(|z|^2)$ (for example $\tilde \o \in \O_{\rm typ}$ by (S7), (S9)).
 Then 
 $\varlimsup_{\e\da 0}\|\nabla_\e \varphi \|_{L^2(\nu_\o ^\e) }<\infty$ for any $\varphi \in C^1_c (\bbR^d)$.
\end{Lemma}

\begin{proof} 
% kuka \blu{\spade Since $\nu^\e_{\o, S} \leq \nu^e_{\o, \bbR^d}$ it is enough to restrict to the setting of Section \ref{MR}.}
Let $\phi$ be as in \eqref{paradiso}. By \eqref{paradiso} and since $ \o \in \O_{\rm typ}$ 
 (apply \eqref{micio2} with $b(\o,z):=|z|^2$), we get
\begin{equation*}
\begin{split}
\| \nabla_\e \varphi \|_{L^2(\nu^\e_\o)}^2 &  \leq C(\varphi) \int d \nu^\e_\o (x,z)  |z|^2\bigl( \phi(x)+ \phi(x+\e z) \bigr) \\
&=
2C(\varphi) \int d \nu^\e_\o (x,z)  |z|^2 \phi(x) = 2 C(\varphi) \int d \mu^\e_\o(x)\phi(x) \l_2( \t_{x/\e} \o)\,.
\end{split}
\end{equation*}
The thesis then follows from Prop. \ref{prop_ergodico}  (recall  (S7) and  (S9)).
\end{proof}

\begin{Lemma}\label{blocco} Given $ \o \in \O_{\rm typ}$ and $\varphi \in C_c^2(\bbR^d)$  it holds 
\begin{equation}\label{football}
\lim  _{\e \da 0}  \int d \nu^\e_{\o}(x,z)  \bigl[\nabla_\e \varphi (x,z) - \nabla  \varphi (x) \cdot z \bigr]^2   =0\,.
\end{equation}
\end{Lemma}
\begin{proof}
 Let $\ell$ be as such that $\varphi (x)=0$ if $|x| \geq \ell$. Fix   $\phi \in C_c(\bbR^d)$ with values in $[0,1]$, such that $ \phi(x)=1$ for $|x| \leq \ell$ and $\phi(x)=0$ for $|x| \geq \ell+1$. 
 %Since $\nabla _\e \varphi(x,z)=0$ if $|x| \geq \ell$ and $|x+\e z|\geq \ell$ and due to  the mean value theorem, we conclude that 
%\be\label{paradiso}
%\bigl | \nabla _\e \varphi(x,z) \bigr | \leq \| \nabla \varphi \|_\infty |z| \bigl( \phi(x)+ \phi(x+\e z) \bigr) \,.
%\en
Recall  \eqref{paradiso}. 
The  upper bound given by   \eqref{paradiso}  with $\nabla_\e\varphi (x,z)$ replaced by  $\nabla \varphi (x) \cdot z $ is also true. We will apply the above bounds for $|z| \geq  \ell$.  On the other hand, we apply \eqref{mirra} for $|z|<\ell$.
%By Taylor expansion we have $\nabla_\e \varphi (x,z) -\nabla \varphi(x) \cdot z=\frac{1}{2} \sum_{i,j} \partial^2 _{ij} \varphi(\, \z_\e(x,z)\,) z_i z_j \e$, where $\z_\e(x,z)$ is a point between $x$ and $x+\e z$. Moreover we note that 
%$\nabla_\e \varphi (x,z) -\nabla \varphi(x) \cdot z=0$ if $|x| \geq \ell$ and $|x+\e z| \geq \ell$. All these observations imply that 
%\be \label{mirra_bis}
%\bigl| \nabla _\e \varphi (x,z) - \nabla \varphi (x) \cdot z\bigr | \leq \e C(\varphi) |z|^2 \bigl( \phi(x) + \phi(x+\e z) \bigr)\,.
%\en
%We will apply \eqref{mirra_bis} for $|z|  < \ell$.
As a result,  we  can bound
\be \label{piano}
 \int d \nu^\e_{ \o}(x,z)  \bigl[\nabla_\e \varphi (x,z) - \nabla  \varphi (x) \cdot z \bigr]^2   \leq 
C(\varphi) [A(\e, \ell )+B(\e, \ell)]\,,
\en
 where (cf. \eqref{micio2}) 
\begin{align*}
 A(\e, \ell ):& =   \int d \nu^\e_{ \o}(x,z) |z|^2  (\phi(x)  + \phi (x+ \e z)  )  \mathds{1}_{\{ |z| \geq  \ell\}} \\
& = 2 \int d \nu^\e_{ \o}(x,z) |z|^2  \phi(x) \mathds{1}_{\{ |z| \geq  \ell\}}  =2 \int d \mu ^\e_{ \o} (x)  \phi(x) h_\ell  (\t_{x/\e} \o) \,,  \\
h_\ell(\o):&= \int d \hat \o (z) c_{0,z}(\o) |z|^2  \mathds{1}_{\{ |z| \geq  \ell\}} \,, \\
 B(\e, \ell):& = \e ^2 \ell^4 \int d \nu^\e_{ \o}(x,z) (\phi(x)  + \phi (x+ \e z)  )  \\
& =2 \e^2  \ell^4 \int d \nu^\e_{ \o}(x,z)  \phi(x)  = 2  \e^2 \ell^4 \int d \mu ^\e_{ \o} (x)  \phi(x) \l_0 (\t_{x/\e} \o)
  \,.
\end{align*}
We now apply Prop. \ref{prop_ergodico}.
Due to (S7) and (S16),
    $\lim_{\e \da 0}  \int d \mu ^\e_{ \o} (x)  \phi(x) h_\ell  (\t_{x/\e} \o) = \int dx\,m \phi(x) \bbE_0[h_\ell]$.  We then get  that $\lim_{\ell\uparrow  \infty,\e \da 0} A(\e, \ell)=0$ by dominated convergence
    and (A5).
Due to (S12) and  (S17) the integral $ \int d \mu ^\e_{ \o} (x)  \phi(x) \l_0 (\t_{x/\e} \o)$ converges to $ \int dx\,m \phi(x) \bbE_0[\l_0]$ as $\e \da 0$. As a consequence,  $\lim_{\e \da 0} B(\e, \ell)=0$. Coming back to \eqref{piano} we finally get \eqref{football}.
\end{proof}

\smallskip

From now on we denote by $\tilde \o$ the environment in $\O_{\rm typ}$ for which we want to prove Items (i), (ii) and (iii) of Theorem \ref{teo1}. 

\smallskip

\noindent
$\bullet$ {\bf Convergence of solutions}. We start by proving Item (i).

\smallskip

We consider \eqref{limite1}.  We recall that the  weak solution $u_\e$ satisfies (cf. \eqref{salvoep})
\begin{equation}\label{salvoepino}
\frac{1}{2} \la \nabla_\e v, \nabla _\e u_\e \ra _{\nu_{\tilde \o} ^\e}+ \la v, u_\e \ra _{\mu_{\tilde \o}^\e}= \la v, f \ra_{\mu_{\tilde \o}^\e}\qquad \forall v \in H^1_{{\tilde \o}, \e}\,.
\end{equation} 
%Moreover, $u_\e$ exists and is unique (by Lax--Milgram theorem).
 Due to \eqref{salvoepino} with $v:=u_\e$  we get that $\|u_\e  \|_{L^2(\mu_{\tilde \o}^\e)}^2 \leq \la u_\e , f_\e\ra_{\mu_{\tilde \o}^\e}$ and therefore $ \|u_\e\| _{L^2(\mu^\e_{\tilde \o})}\leq \|f_\e\| _{L^2(\mu^\e_{\tilde \o})}$ by Schwarz inequality. Hence,  it holds (cf. \eqref{salvoepino})   
$
\frac{1}{2} \| \nabla_\e u_\e  \| _{L^2(\nu^\e_{\tilde \o})} ^2 \leq \| f_\e \| _{ L^2(\mu^\e_{\tilde \o})}^2$.
Since $f_\e \toup f$, the family $\{f_\e\}$ is bounded and therefore there exists $C>0$ such that, for $\e$ small enough as we assume below, 
\begin{equation}\label{salvezza}
\|u_\e  \|_{L^2(\mu_{\tilde \o}^\e)} \leq C \,, \qquad  \| \nabla_\e u_\e  \| _{L^2(\nu^\e_{\tilde \o})}  \leq C \,.
\end{equation}
Due to \eqref{salvezza} and  by \rosso{Proposition} \ref{oro}, 
%when taking above $\o =\tilde{\o}\in \O_{\rm typ}$ 
 along a subsequence we have:
\begin{itemize}
\item[(i)]
$u_\e \stackrel{2}{\toup} u$, where $u$  is of the form $u=u(x)$ and  \rosso{$u\in H^1 _*(m dx )$};
\item[(ii)] $\nabla _\e u_\e  (x,z) \stackrel{2}{\toup} w(x,\o,z):= \rosso{\nabla_* }u  (x)\cdot z + u_1(x,\o,z)$,  
% where $u_1$ belongs to $ L^2\bigl( \bbR^d, L^2_{\rm pot} (\nu)\bigr )$.
$u_1\in  L^2\bigl( \bbR^d, L^2_{\rm pot} (\nu)\bigr )$.
\end{itemize}

 \begin{Claim}\label{vanity} For $dx$--a.e. $x\in \bbR^d$ it holds 
\begin{equation}\label{mattacchione}
\int d \nu(\o, z)  w(x,\o, z)  z= 2 D \rosso{\nabla_*} u (x)\,.
\end{equation}
 \end{Claim}
 \begin{proof}[Proof of Claim \ref{vanity}]
We apply \eqref{salvoepino} to the test function $v(x):= \e \varphi (x) g(\t_{x/\e}\tilde \o)$, where 
$\varphi \in C^2_c(\bbR^d)$ and $g \in \cG_2$ (cf. Section \ref{topo}). Recall that $\cG_2$ is given by bounded functions.   We claim that  $v\in H^1_{\tilde \o, \e}$. Being bounded and with compact support, $v\in L^2( \mu ^\e_{\tilde \o})$. Let us bound $\|\nabla_\e v\|_{L^2( \nu^\e_{\tilde \o}) }$.  Due to \eqref{leibniz} we have
% On the other hand, since $\tilde \o \in \O_{\rm typ}$, by (S12) and (S17), we have that $\nu^\e_{\tilde \o}$ has finite total
 \begin{equation}\label{aquilotto}
\nabla_\e  v (x,z)
= \e \nabla _\e \varphi (x,z) g(\t_{z+x/\e}\tilde \o)+  \varphi(x) \nabla  g ( \t_{x/\e} \tilde \o, z)\,.
\end{equation}
In the above formula, 
the gradient $\nabla g$ is the one defined in \eqref{cantone}.
By Lemma \ref{ascoli}, since $\tilde \o \in \O_{\rm typ}$ and $g$ is bounded, the map   $(x,z)\mapsto  \e \nabla _\e \varphi (x,z) g(\t_{x/\e}\tilde \o)$ is in $ L^2( \nu ^\e_{\tilde \o})$. On the other hand, $\|\varphi(x) \nabla  g ( \t_{x/\e} \tilde \o, z)\|_{ L^2( \nu ^\e_{\tilde \o})}^2$ is bounded by $2\|g\|_\infty \int d\mu^\e_{\tilde \o} (x) \varphi(x)^2 \l_0 ( \t_{x/\e} \tilde \o)$, which converges to a finite number since $\tilde \o \in \O_{\rm typ}$ (cf.  (S12), (S17) and Prop.~\ref{prop_ergodico}). This complete the proof that  $v\in H^1_{\tilde \o, \e}$.

 Due to \eqref{aquilotto},
\eqref{salvoepino} can be rewritten as 
\begin{equation}\label{cioccolata}
\begin{split}
& \frac{\e}{2}\int d  \nu ^\e_{\tilde \o} (x,z)  \nabla _\e \varphi (x,z) g(\t_{z+x/\e}\tilde \o) \nabla _\e u_\e (x,z) +\\
&  \frac{1}{2}\int d  \nu ^\e_{\tilde \o}  (x,z)  \varphi(x) \nabla g ( \t_{x/\e} \tilde \o, z)\nabla_\e  u_\e (x,z) +\\
&\e \int d\mu^\e_{\tilde \o} (x)  \varphi (x) g(\t_{x/\e}\tilde \o)u_\e(x) =\e \int d\mu^\e_{\tilde \o} (x)  \varphi (x) g(\t_{x/\e}\tilde \o)f_\e(x)\,.
\end{split}
\end{equation}
Since the families of functions $\{u_\e(x) \}$, $\{f_\e(x)\}$,  $\{ \varphi (x) g(\t_{x/\e}\tilde \o)\}$ are bounded families in    $L^2(\mu^\e_{\tilde \o})$,  the expressions in the third line of \eqref{cioccolata} go to zero as $\e \da 0$.

%%%%%%%%%%%%%%%%%%%%%%%%%%%%%%%%%%%%%%%
We now claim that 
\be \label{ghiaccio}
\lim  _{\e \da 0}  \int d \nu^\e_{\tilde \o}(x,z)   \nabla _\e u_\e (x,z)  \bigl[\nabla_\e \varphi (x,z) - \nabla  \varphi (x) \cdot z \bigr] g (\t_{z+x/\e}\tilde \o) =0\,.
\en
This follows by  using that $\|g\|_\infty<+\infty$, applying Schwarz inequality and afterwards Lemma \ref{blocco} (recall that  $\|  \nabla _\e u_\e \| _{ L^2 ( \nu^\e_{\tilde \o})}\leq C$).
The above limit \eqref{ghiaccio},  the 2-scale convergence $\nabla_\e u_\e\stackrel{2}{ \toup }w$ and the fact that \eqref{yelena} holds for all functions in $\cH_3\subset \cH$ (cf. Section \ref{topo}), imply that 
\begin{multline}\label{bisonte}
\lim  _{\e \da 0}  \int d \nu^\e_{\tilde \o}(x,z)   \nabla_\e u_\e (x,z)  \nabla _\e \varphi (x,z)   g (\t_{z+x/\e}\tilde \o)=\\
\lim  _{\e \da 0}  \int d \nu^\e_{\tilde \o}(x,z)   \nabla_\e u_\e (x,z)  \nabla  \varphi (x)\cdot z   g (\t_{z+ x/\e}\tilde \o)=\\
\int dx \, m \int d \nu( \o,z) w(x,\o,z) \nabla \varphi(x) \cdot z g(\t_z \o) \,.
\end{multline}
Due to \eqref{bisonte} also the expression in the first line of \eqref{cioccolata} goes to zero as $\e\downarrow 0$. We conclude therefore that also the   expression in the second  line of \eqref{cioccolata} goes to zero as $\e\downarrow 0$. Hence 
\[\lim_{\e \da 0}\int d  \nu ^\e_{\tilde \o}  (x,z) \nabla_\e  u_\e (x,z)  \varphi(x) \nabla g ( \t_{x/\e} \tilde \o, z)=0\,.\]
Due to the 2-scale convergence $\nabla_\e  u_\e\stackrel{2}{ \toup} w$  and since \eqref{yelena} holds for all gradients $\nabla g$, $g\in \cG_2$ (since $\cH_2\subset \cH$), we conclude that 
\begin{equation*}
\int dx \, m \varphi(x) \int d\nu (\o, z) w(x,\o,z) \nabla g (\o,z) =0\,.
\end{equation*}
Since $\{\nabla g \,:\, g \in \cG_2\}$ is dense in $L^2_{\rm pot}(\nu)$,  the above identity implies that, for $dx$--a.e. $x$, the map  $(\o,z) \mapsto w(x,\o,z)$ belongs to $L^2_{\rm sol}(\nu)$.
 On the other hand, we know that 
$w(x,\o,z)= \rosso{\nabla_*  }u (x) \cdot z + u_1(x,\o,z)$, 
where $u_1\in L^2\bigl( \bbR^d, L^2_{\rm pot} (\nu)\bigr )$. Hence, by \eqref{jung}, for $dx$--a.e. $x$ we have that 
\[ u_1 (x,\cdot,\cdot ) = v^a \,, \qquad a:=\rosso{ \nabla _* }u (x)\,. \]
As a consequence (using also \eqref{solare789}),  for $dx$--a.e. $x$, we have
\begin{equation*}%\label{mattacchione}
\int d \nu(\o, z)  w(x,\o, z)  z= \int d \nu(\o,z) z [  \rosso{\nabla_*} u (x)\cdot z + v ^{\rosso{ \nabla _*} u (x)}(\o,z)]= 2 D \rosso{\nabla_*} u (x)\,.
\end{equation*}
This concludes the proof of  Claim \ref{vanity}.
\end{proof}
We now reapply \eqref{salvoepino} but with $v(x):= \varphi(x)$.
We get
\begin{equation}\label{cioccorana}
%\begin{split}
 \frac{1}{2}\int d  \nu ^\e_{\tilde \o} (x,z) \nabla _\e \varphi (x,z) \nabla _\e u_\e (x,z) +
  \int d\mu^\e_{\tilde \o} (x)  \varphi (x) u_\e(x) =\int d\mu^\e_{\tilde \o} (x)  \varphi (x) f_\e(x)\,.
\end{equation}
Let us analyze the first term in \eqref{cioccorana}.
By  \eqref{ghiaccio} which holds also with $g\equiv 1$,   the expression $ \int d  \nu ^\e_{\tilde \o} (x,z) \nabla_\e \varphi (x,z) \nabla_\e  u_\e (x,z) $ equals  $ \int d \nu^\e_{\tilde \o}(x,z)   \nabla _\e u_\e (x,z) \nabla  \varphi (x) \cdot z +o(1)$ as $\e\da 0$. Since the function  $(\o,z) \mapsto z_i$  is in $\cH$ and since $\tilde \o \in \O_{\rm typ}$, by  the 2-scale convergence  $\nabla _\e u_\e
 \stackrel{2}{\toup} w$ we obtain that 
\begin{equation*}
\lim_{\e \da 0} \int d  \nu ^\e_{\tilde \o} (x,z) \nabla_\e \varphi (x,z) \nabla _\e u_\e (x,z)  =
\int dx \, m \int d \nu(\o, z)  w(x,\o, z) \nabla  \varphi (x)\cdot z\,.
\end{equation*}
To treat the second and third terms in \eqref{cioccorana} we use  that   $u_\e \stackrel{2}{\toup} u$ with  $u=u(x)$, $1\in \cG$,   and that $f_\e {\toup} f$.  Due to the above observations, by taking the limit $\e \da 0$ in \eqref{cioccorana} we get 
\begin{multline}
\frac{1}{2}\int dx \, m \nabla \varphi (x)\cdot \int d \nu(\o, z)  w(x,\o, z)  z+\\
\int dx\, m\varphi(x) u(x) =\int dx\, m  \varphi(x) f(x)
\,.\end{multline}
Due to  \eqref{mattacchione} the above identity reads 
\begin{equation}\label{basta180}
 \int dx \nabla  \varphi (x) \cdot D\rosso{ \nabla_*} u (x)
+\\
\int dx \varphi(x) u(x) =\int dx  \varphi(x) f(x)\,,
\end{equation}
i.e. 
 $u$ is a weak solution of \eqref{eq2} \rosso{(recall  that $C_c^\infty (\bbR^d) $ is dense in $H^1_*(m dx)$ and note  that $\nabla \varphi$ in \eqref{basta180} can be replaced by $\nabla_* \varphi$ due to Warning \ref{divido})}. This concludes the proof of \eqref{limite1}. 

It remains to prove \eqref{limite2}. It is enough to apply the same arguments of \cite[Proof of Thm.~6.1]{ZP}.   Since $f_\e\to f$ we have $f_\e\toup f$ and therefore, by \eqref{limite1}, we have $u_\e \toup u$. This implies that $v_\e \toup v$ (again by \eqref{limite1}), where $v_\e$ and $v$ are respectively the weak solution of $-\bbL_\o^\e v_\e+  v_\e =u_\e$ and $-\rosso{\nabla_*\cdot  D  \nabla_*} v+  v  =u$. By taking the scalar product of the weak version of \eqref{eq1} with $v_\e$ (as in \eqref{salvoep}), the scalar product of the weak version of \eqref{eq2} with $v$ (as in \eqref{delizia}), the scalar product of  the weak version of  $-\bbL^\e_\o v_\e+  v_\e =u_\e$ with $u_\e$ and the scalar product of the weak version of  $-\rosso{\nabla_* \cdot D \nabla_*} v+  v  =u$ with $u$ and comparing the resulting expressions, we get
\begin{equation}\label{adria}
\la u_\e, u_\e \ra _{\mu^\e_\o} = \la v_\e, f_\e \ra _{\mu^\e_\o} \,, \qquad \int  u(x)  ^2   dx=  \int  f(x)v(x)  dx\,.
\end{equation}
Since $f_\e \to f$ and $v_\e \toup v$ we get that $\la v_\e, f_\e \ra _{\mu^\e_\o}\to \int v(x) f(x)  m dx$. Hence, by  \eqref{adria}, we conclude that
$\lim_{\e\da 0}  \la u_\e, u_\e \ra _{\mu^\e_\o}=  \int  u(x)^2 m  dx$. The last limit and the weak convergence $u_\e \toup u$ imply the strong convergence $u_\e \to u$ by Remark \ref{forte}. This concludes the proof of \eqref{limite2} and therefore of Theorem \ref{teo1}--(i).

\medskip

\noindent
$\bullet$ {\bf Convergence of flows}. We prove now \eqref{limite3} in   Item (ii), i.e. $\nabla_\e u_\e \toup \rosso{\nabla_* }u$. By \eqref{salvezza} the analogous of bound \eqref{recinto_grad} is satisfied. Suppose that $f_\e \toup f$. Take $\varphi \in C_c^1(\bbR^d)$, then $\la \varphi, f_\e \ra_{ \mu_{\tilde\o}^\e }\to \la \varphi, f \ra_{m dx }$. By Item (i) we know that $u_\e \toup u$ and therefore $\la \varphi, u_\e \ra_{\mu_{\tilde \o}^\e }\to \la \varphi, u \ra_{m dx }$. The above convergences and  \eqref{salvoepino} with $v$ given by  $\varphi $ restricted to $\e \widehat{ \tilde \o}$ (by  Lemma \ref{ascoli} $v\in H^1_{\tilde \o,\e}$), we conclude that \[
\lim _{\e\da 0} \frac{1}{2} \la \nabla_\e \varphi, \nabla _\e u_\e\ra_{ \nu^\e _{\tilde \o} } = \lim_{\e \da 0} \Big[
\la \varphi, f_\e \ra_{ \mu_{\tilde \o}^\e }-\la \varphi, u_\e \ra_{\mu_{\tilde \o}^\e }\Big]= \la  \varphi, f-u  \ra_{ m dx }\,.\]
Due to \eqref{eq2} and \eqref{delizia}, the r.h.s. equals $\int dx \, m D(x) \rosso{\nabla_* }\varphi(x) \rosso{\cdot\nabla_*} u(x) $.
This proves the analogous of \eqref{deboluccio_grad} and therefore  \eqref{limite3}.

Take now $f_\e \to f$. Then, by \eqref{limite2}, $u_\e \to u$.  Reasoning as above we get  that, given $g_\e \in H^1_{ \tilde \o,\e}$ and $g\in \rosso{H^1_*(m dx)}$ with  $L^2(\mu^\e_{\tilde \o}) \ni g_\e \toup g\in L^2(m dx)$, it holds 
\[
\lim _{\e\da 0} \frac{1}{2} \la \nabla_\e g_\e, \nabla _\e u_\e\ra_{\nu^\e _{\tilde \o} } = \lim_{\e \da 0} \Big[
\la g_\e , f_\e \ra_{\mu_{\tilde \o}^\e}-\la \varphi, u_\e \ra_{\mu_{\tilde\o}^\e}\Big]=
 \la  g, f-u   \ra_{m dx }\,.\]
 Since \rosso{$g\in H^1_*(m dx)$}, due to \eqref{eq2}, the r.h.s. equals $ \int dx \, m D(x) \rosso{\cdot \nabla_*} g(x) \rosso{\nabla_*} u(x) $.
This proves \eqref{limite4}.

\medskip

\noindent
$\bullet$ {\bf Convergence of energies}. We prove  Item (iii). Since $f_\e \to f$, we have $u_\e \to u$ by \eqref{limite2} and  $\nabla_\e u_\e \to \rosso{\nabla_*} u $ by \eqref{limite4}.  It is enough to apply \eqref{fortezza_grad} with $g_\e := u_\e$ and $g:= u$ and one gets \eqref{limite5}.

%%%%%%%%%%%%%%%%%

\section{Proof of Theorem \ref{teo2}}\label{dim_teo2}
The limit \eqref{marvel0} follows from Remark \ref{gelatino}, \eqref{limite2} in Theorem \ref{teo1} and \cite[Thm.~9.2]{ZP}.
%The limit \eqref{marvel2} follows from \eqref{marvel1} and the bounds $\|P^\e_{\o,t} f\|_\infty \leq \|f\|_\infty$, $ \|P_t f\| _\infty \leq \|f\|_\infty$.
To treat  \eqref{marvel1}  and \eqref{marvel2}
%The limits \eqref{marvel1} and \eqref{marvel2}  follow from Theorem \ref{teo1} by  the same arguments used in the proof of Corollary 2.5 and Lemma 6.1  in \cite{F1} if Condition (ii) is satisfied (indeed, in \cite{F1}, $\hat \o $ is a subset of $ \bbZ^d$ and therefore Condition (ii) is fulfilled there).
%
%We explain how to adapt the proof in \cite{F1}  under Condition (i).  To this aim 
we need the following fact:
\begin{Lemma}\label{pre_teo2} %There exists a set $\O_{\rm dem}
Suppose that Condition (i) or Condition (ii) in Theorem \ref{teo2} is satisfied. Fix 
a weakly decreasing  function $\psi: [0,+\infty) \to [0,+\infty)$,  such that $\bbR^d \ni x \mapsto \psi(|x|) \in [0,\infty)$ is  Riemann integrable.
Then
% there exists a Borel set $\O_\sharp\subset \O $ such that $\cP( \O_\sharp)=1$ and  such that, for each $\o \in \O_\sharp$,
$\cP$--a.s.  it holds 
%the following holds for each $\o \in \O_\sharp$:
%\begin{align}
%& \varlimsup_{\e \da 0} \int d \mu^\e _\o (x)\psi(|x|)  < \infty\,,\label{claudio1}\\
\be \label{claudio2}
 \lim _{\ell \uparrow \infty} \varlimsup_{\e \da 0} \int d \mu^\e _\o (x)\psi(|x|)  \mathds{1}_{\{ |x| \geq \ell\}}=0\,. \en
\end{Lemma}
Thanks to \eqref{marvel0} and  Lemma \ref{pre_teo2} (proved below)  applied to the function $\psi(r):=1/(1+ r^{d+1})$,  the limits \eqref{marvel1}  and \eqref{marvel2} follow   by  the same arguments used in the proofs of Corollary 2.5 and Lemma 6.1  in \cite[Sections~6,7]{F1}. We give the proof for \eqref{marvel1}, since  \eqref{marvel2} easily follows from \eqref{marvel1} as in the proof  of Corollary 2.5 in \cite{F1} due to some estimates presented below. Set $h(x):= P_t f(x)$ and $h_\o^\e(x):= P^\e_{\o,t} f(x)$.
To derive \eqref{marvel1} it is enough to show that $ \lim_{\e \da 0}  \| h \|^2 _{L^2(\mu^\e_\o)}$, $ \lim_{\e \da 0}   \la h_\o ^\e , h \ra _{\mu^\e_\o}$ and $\lim _{\e \da 0}  \| h _\o ^\e \|^2 _{L^2(\mu^\e_\o)}$ equal $\|h\|^2_{L^2(m dx )}$.
%\begin{align}
%& \lim_{\e \da 0}  \| h \|^2 _{L^2(\mu^\e_\o)}= \|h\|^2_{L^2(m dx )}\,,\label{rain1}\\
%&  \lim_{\e \da 0}   \la h_\o ^\e , h \ra _{\mu^\e_\o} = \|h\|^2_{L^2(m dx )}\,,\label{rain2}\\
%& \lim _{\e \da 0}  \| h _\o ^\e \|^2 _{L^2(\mu^\e_\o)}=  \|h\|^2_{L^2(m dx )} \label{rain3}\,.
%\end{align}
To this aim we fix $C>0$ such that $|h(x)| \leq C \psi(|x|)$ $\forall x \in \bbR^d$ (\rosso{note} that $h$ decays exponentially fast).
  For each $\ell >0$ fix a function $g_\ell \in C_c(\bbR^d)$ such that $|g_\ell(x)| \leq C \psi(|x|)$ $\forall x \in \bbR^d$ and  $g_\ell (x)=h(x)$ $ \forall x\in \bbR^d$ with $|x|\leq \ell$.
Then
\be\label{ranetta76}
\| g_\ell - h\|^2_{L^2(\mu_\o^\e)} 
\leq 2 C^2  \|\psi\|_\infty \int d \mu^\e _\o (x)\psi(|x|)  \mathds{1}_{\{ |x| \geq \ell\}}\\
%& \| g_\ell - h\|^2_{L^2(m dx )}  \leq (C+\|f\|_\infty) ^2\int d \mu^\e _\o (x)\psi(|x|)  \mathds{1}_{\{ |x| \geq \ell\}}
\en
Since 
\begin{multline}\label{yuyu}
\big| \| h \|^2 _{L^2(\mu^\e_\o)}- \|h\|^2_{L^2(m dx )} \big| \leq 
\big |  \| h \|^2 _{L^2(\mu^\e_\o)}-  \| g_\ell \|^2 _{L^2(\mu^\e_\o)}\big |\\
 + \big|  \| g_\ell \|^2 _{L^2(\mu^\e_\o)}- \| g_\ell \|^2 _{L^2(mdx)}\big|+  \big |  \| g_\ell \|^2 _{L^2(m dx )}-  \| h \|^2 _{L^2(m dx )}\big|\,,
\end{multline}
we get $ \lim_{\e \da 0}  \| h \|^2 _{L^2(\mu^\e_\o)}= \|h\|^2_{L^2(m dx )}$: the first term in the r.h.s. of \eqref{yuyu} can be controlled by \eqref{claudio2} and \eqref{ranetta76}, the second term by Prop.~\ref{prop_ergodico} as $1\in \cG$ (recall (S1) and that $g_\ell \in C_c(\bbR^d)$), while the third term trivially goes to zero.

To get  that $ \lim_{\e \da 0}   \la h_\o ^\e , h \ra _{\mu^\e_\o} = \|h\|^2_{L^2(m dx )} $,  we note that $\big|  \la h_\o ^\e , h \ra _{\mu^\e_\o}-  \la h_\o ^\e , g_\ell \ra _{\mu^\e_\o}\big|\leq \|h_\o^\e\|_{L^2(\mu_\o^\e)} \|h - g_\ell \|_{L^2(\mu^\e_\o)}$ goes to zero as $\e\da 0,\ell\uparrow \infty$ since $\{h_\o^\e\}$ is bounded in $L^2(\mu^\e_\o)$ (cf.  \eqref{marvel0}) and due to  \eqref{claudio2} and \eqref{ranetta76}, while $\lim_{\e\da0}  \la h_\o ^\e , g_\ell \ra _{\mu^\e_\o}= \la h, g_\ell \ra_{m dx}$ since $h_\o^\e \to h$ (cf. \eqref{marvel0})  and trivially $\lim_{\ell \uparrow \infty} \la h, g_\ell \ra_{m dx}= \la h, h \ra_{m dx}$.
Finally, the limit $ \lim _{\e \da 0}  \| h _\o ^\e \|^2 _{L^2(\mu^\e_\o)}=  \|h\|^2_{L^2(m dx )} $ follows from  \eqref{marvel0} and Remark \ref{forte}.
We have therefore proved \eqref{marvel1}.

%
% (for completeness we give the proof in
%% mammina
%\cite[App.~B]{FV2} %sottomesso
%%   Appendix \ref{palline}
%  ). Note that  in \cite{F1}  $\hat \o $ is a subset of $ \bbZ^d$ but indeed the proofs in \cite{F1} rely on \eqref{claudio1} and \eqref{claudio2}. 

\begin{proof}[Proof of Lemma \ref{pre_teo2}] 
% The limits \eqref{scivolo} and \eqref{claudio2} imply \eqref{claudio1}. Let us prove that \eqref{claudio2} holds $\cP$--a.s. 
To simplify the notation we prove a slightly different version of \eqref{claudio2}, the method can be easily adapted to \eqref{claudio2}. In particular, we now  prove that $\cP$--a.s. it holds 
\be\label{claudio3}
 \lim _{\ell \uparrow \infty} \varlimsup_{\e \da 0}X_{\e,\ell}=0\; \text{ where }\; X_{\e, \ell} (\o):=  \e^d \sum_{
\substack{ k\in \bbZ^d:|k| \geq  \ell/\e } }\psi( | \e k| ) N_k\,.
\en
Trivially Condition (i) implies \eqref{claudio3}. Let us suppose that Condition (ii) is satisfied.
%Consider the random variable 
%\[ X_{\e, \ell} (\o):=  \e^d \sum_{
%\substack{ k\in \bbZ^d:\\|k|_\infty \geq  \ell/\e } } \frac{1}{1+ | \e k|_\infty ^{d+1}} N_k\,.
%\]
Given $\e\in(0,1)$ let $r=r(\e) $ be  the positive integer of the form $2^a$, $a 
\in \bbN$, such that $r^{-1} \leq \e <2 r^{-1}$.  Then, since $\psi$ is weakly decreasing,  
\be \label{cenere}
X_{\e, \ell} (\o) \leq2^d    Y_{r, \ell} (\o) \;\text{ where } \;   Y_{r, \ell} (\o):=r^{-d} \sum_{
\substack{ k\in \bbZ^d: |k| \geq  r \ell/2 } } \psi(| k/r|)N_k\,.
\en
%Given $\e\in(0,1)$ let $r=r(\e) $ be  the positive integer such that $r^{-1} \leq \e <( r-1)^{-1}$.  Then 
%\be \label{cenere}
%X_{\e, \ell} (\o) \leq \frac{r^d}{ (r-1)^{d}}   Y_{r, \ell} (\o)\,, \;\;\;   Y_{r, \ell} (\o):=r^{-d} \sum_{
%\substack{ k\in \bbZ^d:\\|k| \geq  r \ell/2 } } \frac{1}{1+ | k/r| ^{d+1}} N_k\,.
%\en
In particular, to get \eqref{claudio3} it is enough to show that,  $\cP$--a.s,   $\lim _{\ell \uparrow \infty} \varlimsup_{r  \uparrow \infty }Y_{r,\ell}=0$, where $r$ varies in $\G:=\{2^0,2^1,2^2,\dots\}$.  From now on we understand that $r \in \G$. Since $\bbE[N_k]=m$ and since $\psi(|x|)$ is Riemann integrable,  we have
\be\label{fabietto1}
\lim_{r\uparrow \infty } \bbE[Y_{r, \ell}] = z_\ell:=m \int  \psi( |x|)   \mathds{1}_{\{|x| \geq \ell/2\}}dx <\infty\,.
\en
We now estimate the variance of $Y_{r, \ell}$. Due to the stationarity of $\cP$ and since $\bbE[N_0^2]<+\infty$, it holds $\sup_{k \in\bbZ^d} \text{Var} (N_k)<+\infty$.
 We let $\g:= \a$ if $d=1$ and $\g:=0$ if $d \geq 2$.  By Condition (ii) we have, for some fixed constant $C_1>0$,  
\begin{align*}
\text{Var} (Y_{r, \ell})& \leq C_1  r^{-2d} \sum_{
\substack{ k\in \bbZ^d:\\|k| \geq  r \ell/2 } }\;
\sum_{
\substack{ k'\in \bbZ^d:\\|k'| \geq  r \ell/2, } }
 \left[|k-k'|^{-1-\g}\mathds{1}_{k\not =k'}+ \mathds{1}_{k=k'}  \right] \psi (| k/r|) \psi( | k'/r|) \\
 &=: I_0(r,\ell)+ I_1(r,\ell)+I_2(r,\ell) \,,
\end{align*}
where $I_0(r, \ell)$, $I_1(r,\ell)$ and $I_2(r,\ell)$  denote   the contribution from addenda as above  respectively with  (a) $k=k'$,  (b)    $|k-k'|\geq r$ and (c) $1\leq |k-k'|<r$.
Then we have 
\begin{align}
&    \lim _{r \uparrow \infty}   r ^{d} I_0(r,\ell)=C_1  \int _{|x| \geq \ell/2} \psi(|x|)^2 dx <+\infty\,, \label{vino0} \\
&  \lim _{r \uparrow \infty}   r ^{1+\g} I_1(r,\ell)= C_1 \int _{|x| \geq \ell/2}dx \int _{|y| \geq \ell/2} dy \frac{ \mathds{1}_{\{ |x-y|\geq 1\} } }{  |x-y|^{1+\g}}\psi(| x| )\psi (|y|)  <+\infty \,. \label{vino1}
\end{align}
To control $I_2(r, \ell)$  we observe that %$\sum _{ k' \in \bbZ^d: 1\leq |k-k'|\leq r} |k-k'|^{-1} $
\[
 \sum _{\substack{v \in \bbZ^d\,:\\ 1\leq |v|_\infty \leq c r} }|v|_\infty ^{-1-\g } \leq C' \sum _{n=1}^{ c r } n^{d-2-\g}\leq \begin{cases}
 C'' r^{d-1} & \text{ if } d\geq 2\\
 C'' & \text{ if } d=1 
 \end{cases}
  \,.
\]
The above bound implies for $r$ large   that 
\be\label{vino20}
\begin{split}
I_2(r, \ell) & \leq   C_1 \|\psi\|_\infty   r^{-2d} \sum_{
\substack{ k\in \bbZ^d:\\|k| \geq  r \ell/2 } }\psi( | k/r| )
\sum_{
\substack{ k'\in \bbZ^d:\\ 1 \leq |k-k'| 	\leq r  }}
 |k-k'|^{-1-\g}\\
& \leq
 C_2 r^{-1 }\int_{|x| \geq \ell/2} \psi( | x| ) dx \,.
 \end{split}
\en
Due to \eqref{vino0}, \eqref{vino1} and \eqref{vino20},  $\text{Var} (Y_{r, \ell})\leq C_3(\ell) r^{-1}$ for $ r\geq C_4(\ell)$. 
Now we write explicitly $r= 2^j$. 
 By Markov's inequality, we have  for $j\geq C_5(\ell)$ that \[   \cP( |Y_{2^j, \ell}-\bbE[Y_{2^j, \ell}]| \geq 1/j ) \leq j^2 \text{Var} (Y_{2^j, \ell})\leq C_3(\ell) j^2  2^{-j}\,.
\]
Since the last term is summable among $j$, by Borel--Cantelli lemma we conclude that, $\cP$--a.s., 
$
|Y_{2^j, \ell}-\bbE[Y_{2^j, \ell}]| \leq  1/j $ for  all $ \ell \geq 1$ and $ j\geq C_6(\ell, \o)$.
This proves that,  $\cP$--a.s., $\lim _{r\uparrow \infty, r \in \G}Y_{r, \ell}= z_\ell$ (cf. \eqref{fabietto1}). Since $\lim_{\ell \uparrow \infty}z_\ell=0$, we get that $\lim _{\ell \uparrow \infty} \lim_{r  \uparrow \infty,r \in \G }Y_{r,\ell}=0$, $\cP$--a.s.
\end{proof}

%%%%%%%
\section{Proof of Theorem \ref{teo3}}\label{ida}
%Since we are dealing with Mott v.r.h. with bounded energy rates,  we have the bound $c_{x,y}(\o)\leq C_0 e^{-\g |x-y|}$ for some fixed $C_0>0$.
% Without loss we take $E_i \equiv 0$ and $\g=1$ so that $c_{x,y}(\o) \leq e^{-|x-y|}$.
%Let $\cP$ be as in Theorem \ref{teo3}. Then the diffusion matrix is strictly positive (cf. \cite{FSS}) and all assumptions (A1)--(A8) are all satisfied.
Note that Assumptions (A1), (A2), (A3) are automatically satisfied.
By extending the probability space, 
given $\o$   we  associate to each unordered pair of site $\{x,y\}$  in $\hat \o$
 a Poisson process $\bigl( N_{x,y}(t)\bigr)_{t\geq 0}$ with intensity $c_{x,y}(\o)$,  such that $N_{x,y}(\cdot)$ are independent processes when varying the pair $\{x,y\}$.
  Note that $N_{x,y}(t) = N_{y,x}(t)$.
   We write $\cK:=\bigl( N_{x,y}(\cdot ) \bigr)$ for the above family of Poisson processes and denote by $\bbP_\o$ the associated law.  We denote by $\bbP$  the annealed law of the pair $(\o, \cK)$, defined as $\bbP:= \int d\cP(\o) \bbP_\o$.
   % An alternative measure--theoretical definition of $\bbP$ can be obtained by adapting the standard construction for the random connection model presented in \cite{MR}.

\begin{Lemma}\label{igro1}
 There exists $t_0>0$ such that for $\bbP$--a.a.  $(\o,\cK)$ the undirected  graph $\cG _{t_0}(\o,\cK)$ with vertex set $\hat \o$ and edges $\{ \{x,y\}\,:\, x\not =y \text{ in }\hat \o\,,\; N_{x,y}(t_0)\geq 1\}$ has only connected components with finite 
cardinality.
\end{Lemma}
\begin{proof}
Note that   $\bbP_\o ( N_{x,y}(t)\rosso{\geq} 1)= 1- e^{-c_{x,y}(\o)t} \leq 1- \exp\{-g(|x-y|) t\}\leq  C_1 g(|x-y|)  t$ for some fixed $C_1>0$ if we take $t\leq 1$ (since $g$ is bounded).  We restrict to $t$ small enough  such that $C_1\|g\|_\infty  t <1$ and $t\leq 1$.
Consider  the random connection model \cite{MR} where first $\hat \o $ is sampled according to a Poisson point  process with intensity $m$, afterwards   an edge between    $x\not =y$ in $ \hat{\o}$ is created  with probability $C_1 g(|x-y|)  t $.  
Due to our initial estimate and the independence of the processes $N_{x,y}$,  one can couple the above random connection model with the previous process $(\o, \cK)$ with law $\bbP$ in a way that   the graph in the random connection model contains the graph $\cG _{t}(\o,\cK)$.
We choose $t=t_0$ small enough to have $ m C_1 t_0 \int _{\bbR^d} dx g(|x|) <1$. 
The above bound and the branching process argument in the proof of \cite[Theorem~6.1]{MR} (cf. (6.3) there) imply that   a.s. the  random connection model has only connected components with finite cardinality. 
\end{proof}

From now on $t_0$ will be as in Lemma \ref{igro1}. 
Due to the loss of memory of Poisson point processes,  from Lemma \ref{igro1}  we get the following:
\begin{Corollary}\label{lancia}
For $\bbP$--a.a.  $(\o,\cK) $ and for each $r \in \bbN$  the undirected  graph $\cG ^r_{t_0}(\o,\cK)$ with vertex set $\hat \o$ and edges $\{ \{x,y\}\,:\, x\not =y \text{ in }\hat \o\,,\; N_{x,y}((r+1) t_0)> N_{x,y} (r t_0) \}$ has only connected components with finite 
cardinality.
\end{Corollary}

By the graphical representation of the exclusion process and Harris' percolation argument \cite{Du}, we conclude that, for $\cP$--a.a. $\o$,  the exclusion process is well defined a.s. for all times $t\geq 0$. We explain in detail this issue.   
%Due to Lemma \ref{igro1} and the loss of memory of Poisson processes (consider the time intervals $[0,t_0]$, $[t_0, 2t_0]$,..),  for a.a. 
%$(\o, \cK)$ the 
%graph $\cG_t(\o,\cK) $ has only connected components of finite cardinality for all $t\geq 0$. 
Take such a good $(\o,\cK)$ fulfilling the property stated in Corollary \ref{lancia}.  Given a particle configuration  $\eta(0)\in \{0,1\}^{\hat \o}$ we define  the deterministic  trajectory $\bigl(\eta(t)[ \o,\cK]\bigr)_{t \geq 0 }$  starting at $\eta(0)$ by an iterative procedure.  Suppose  the trajectory has been defined up to time $r t_0$, $r\in \bbN$. Let $\cC$ be any connected component of $\cG^r_{t_0}(\o,\cK)$ and let 
 \begin{multline*}
  \{s_1<s_2< \cdots <s_k\} =\\
  \bigl \{s \,: N_{x,y}(s) = N_{x,y}(s-)+1\,, \;\{x,y\} \text{ bond in } \cC, \; r t_0  <s \leq (r+1) t_0\bigr\}\,.
 \end{multline*}
 Since $\cC$ is finite, the l.h.s. is indeed  a finite set. The local evolution $\eta(t)_z[ \o,\cK] $ with $z \in \cC$ and  $r t_0 < t  \leq (r+1) t_0$ is described as follows. 
 Start with $\eta(rt_0)[ \o,\cK]$ as configuration at time $r t_0$ in $\cC$. At time $s_1$ exchange the values between $\eta_x$ and $\eta_y$ if $N_{x,y}(s_1)= N_{x,y}(s_1-)+1$ and $\{x,y\}$ is an edge  in $\cC$ (there is exactly one such edge, a.s.).  Repeat the same operation orderly for times $s_2,s_3, \dots , s_k$. Then move to another connected component of  $\cG_{t_0}^r(\o,\cK)$ and so on.  This procedure defines $ \eta(t)[\o,\cK]_{ r t_0< t \leq (r+1) t_0}$.
It is standard to check that for $\cP$--a.a. $\o$  the random trajectory  $\bigl(  \eta(t)[\o, \cK]\bigr)_{t\geq 0}$ (where the randomness comes from $\cK$) is an exclusion process on $\hat \o$ with initial configuration $\eta(0)$ and formal generator  \eqref{ringo}.

\medskip

Due to  \eqref{marvel2}  in Theorem \ref{teo2} (Condition (ii) is satisfied in the present context), Lemma 
\ref{igro2} below and since $ \int  \varphi(x) \rho(x,t) dx= \int\rho_0(x) P_t \varphi(x) dx $,
 to prove \eqref{pasqualino} we only need to show that 
\be\label{natalino00}
\lim_{\e\da 0} \mathfrak{m}_\e \Big(\Big|  \e^d \sum_{x\in \hat \o} \eta_x(0) P_t  \varphi (\e x) - \int _{\bbR^d} \rho_0(x) P_t \varphi(x)  dx\Big| >\d 
\Big)=0\,.
\en
Since $P_t\varphi$ is a \rosso{continuous} function decaying fast to infinity, \eqref{natalino00} follows  by approximating $P_t\varphi$ with functions of compact supports and using both \eqref{marzolino} and \eqref{claudio2}. At this point, to complete the proof of Theorem \ref{teo3} it remains to prove the following:

%Due to  Lemma \ref{pre_teo2} (Condition (ii) is satisfied in the present context), Theorem \ref{teo3} follows from \eqref{marvel2} and Lemma \ref{igro2} below. This derivation follows exactly the same steps of \cite[Section 3]{F1} with a unique  exception: the $\lim_{\ell \uparrow \infty}\varlimsup_{\e\da 0}$ of  the  l.h.s. of (3.4) in \cite{F1}  goes to zero due to \eqref{claudio2}.

\begin{Lemma}\label{igro2}
For $\cP$--a.a. $\o$ the following holds. Fix $\d,t>0$ and $\varphi\in C_c(\bbR^d)$ and let $\mathfrak{n}_\e$ be an $\e$--parametrized family of probabililty measures on $\{0,1\}^{\hat \o}$. Then 
\be\label{vigorsol}
\varlimsup_{\e\da 0} \bbP_{\o, \mathfrak{n}_\e} \Big(\Big| \e^d \sum_{x\in \hat \o} \varphi (\e x) \eta_x ( \e^{-2} t)- 
\e^d \sum_{x\in \hat \o} \eta_x(0) P_{\o, t}^\e  \varphi (\e x) \Big| >\d \Big)=0\,.
\en
\end{Lemma}
\begin{proof} % Without loss of generality we restrict to positive $\varphi$.  
We think of  the exclusion process  as built according to the graphical construction described above, after sampling $\eta(0)$ with distribution $\mathfrak{n}_\e$. We fix   $\o \in \O_{\rm typ} \cap \O_\sharp$  (cf. Theorem \ref{teo2}) such that  $(\o, \cK)$ fulfills the property in Corollary \ref{lancia} for $\bbP_\o$--a.a. $\cK$ (this takes place for $\cP$--a.a. $\o$).
%The proof of the above Lemma can be obtained by  the same arguments used in the proof of Prop. 3.1 in \cite{F0}. We give some comments for the reader's convenience.
Given  $x\in \hat \o$, $r \in \bbN$, we denote by  $\cC_r(x)$ the connected component of $x$ in the graph $\cG_{t_0} ^r (\o, \cK)$. Fix $t \in (r t_0, (r+1) t_0]$. Due to  the above graphical construction, if we know $(\o, \cK)$, then to determine $ \eta_x(t)[\o, \cK]$  we only need to know $\eta _z (r t_0)[\o, \cK]$ with $z \in \cC_r(x)$. By iterating the above argument 
we conclude that, knowing  $(\o, \cK)$, the value of $ \eta_x(t)[\o, \cK]$ is determined by $\eta_z (0)$ as $z$ varies in the finite set 
\[ Q_r(x):=\cup_{z_r \in \cC_r (x) } \cup _{z_{r-1} \in \cC_{r-1}(z_r) } \cdots  \cup _{ z_1\in \cC_1(z_2)}   \cC_0(z_1)\,.\]
Suppose that $\varphi $ has support in the ball  $B(\ell)$ of radius $\ell$ centered at the origin.
Then, by the above considerations,    for  $\ell_\e$  depending on $\e$  large enough  we have 
\[ \bbP_\o ( A_\e) \leq \e\; \text{ where }\;  A_\e:=\Big\{\cK\,:\, \cup _{r \in \bbN \cap [0,\e ^{-2} t/t_0]}  \cup _{\substack{x \in  \hat \o: \\  \e |x| \leq \ell }} Q_r(x)  \subset B(\ell_\e) \Big\}\,.
\]
Note that, when the event $A_\e$ takes place, the value  $ \e^d \sum_{x\in \hat \o} \varphi (\e x) \eta_x ( \e^{-2} t)$ depends on $\eta(0)$ only through $\eta_z(0)$ with $ z \in \hat \o \cap B(\ell_\e)$. Hence, if we replace $\eta(0)$ by  removing all particles at sites $z \in \hat \o \setminus   B(\ell_\e)$, the value $ \e^d \sum_{x\in \hat \o} \varphi (\e x) \eta_x ( \e^{-2} t)$  remains unchanged.

On the other hand, due to \eqref{claudio2} and since $P_t \varphi$ decays fast to infinity, we have  that  $\varlimsup_{\ell \uparrow \infty, \e \da 0}  \e^d \sum_{x\in \hat \o}  P _t   \varphi (\e x)  \mathds{1}
 _{\{| \e x | > \ell\}}$=0. In particular, for $\e\leq\e_0(\d)$ and  by taking  $\ell _\e $ large enough,    it holds $ \e^d \sum_{x\in \hat \o: |x | \geq \ell_\e } \eta_x(0) P_{\o, t}^\e  \varphi (\e x) \leq \d/2$  for any initial configuration $\eta(0)$.

Call $\overline{\mathfrak{n}}_\e$ the probability measure on $\{0,1\}^{\hat \o}$ obtained as follows: sample $\eta(0)$ with law $\mathfrak{n}_\e$, afterwards set the particle number equal to zero at any site $x \in \hat \o$ with $|x| > \ell_\e$. By the above considerations, to get \eqref{vigorsol} it is enough to prove the same limit with $\mathfrak{n}_\e$ replaced by $\overline{\mathfrak{n}}_\e$ and with $\d$ replaced by $\d/2$. This implies  that, in order to prove Lemma \ref{igro2}, we can restrict  (as we do) to probability measures $\mathfrak{n}_\e$ such that  $\mathfrak{n}_\e \bigl(  \eta_x (0)=0 \; \forall x \in \hat \o \setminus B(\ell_\e)\bigr)=1$.

 The key observation now, going back to \cite{N} and proved below, is  that the symmetry of the jump rates implies the following pathwise representation
for each $x \in \hat \o$:
\be\label{razzo}
\eta_x(t)= \sum_{y\in \hat \o} p_\o(t,x,y )\eta_y(0) + \sum _{y \in \hat \o} \int _0^t p_\o(t-s,x,y) dM_{y}  (s)\,,
\en
where  $p_\o(t,x,y)$ is the probability to be at $y$ for a random walk on $\hat \o$  with jump probability rates $c_{a,b}(\o)$ and starting at $x$,  and   $M_{y}(\cdot)$'s are   martingales defined by
\be\label{deprimo}
dM_{y} (t) :=\sum_{z \in \hat\o}  (\eta_z- \eta_y)(t-) dA_{y,z}(t)  \,, \qquad A_{y,z}(t):= N_{y,z}(t) - c_{y,z}(\o) t\,.
\en
Recall that for $\o \in \O_{\rm typ}$ it holds $ \sum _{y\in \hat\o}c_{x,y}(\o)=\l_0(\t_x\o) <+\infty$ for all $x \in\hat \o$ (see the discussion before Theorem \ref{teo2}). Note moreover that $N_x(t):= \sum_{y\in \hat \o}N_{x,y}(t) $, $t\geq 0$,  is a Poisson process with intensity 
$ \sum _{y\in \hat\o}c_{x,y}(\o)$. Hence, we can restrict to environments $\o$ such $N_x(t)<+\infty$  for all $t\geq 0$ and 
 $ \sum _{y\in \hat\o}c_{x,y}(\o)<+\infty$ for all $x\in\hat\o$. As a consequence, $( M_x(t) )_{t\geq 0}$  in \eqref{deprimo} is well defined. 

 The above identity \eqref{razzo} is well posed since 
we start with a configuration $\eta(0)$ having a finite number of particles. Indeed, the first series  in the r.h.s. is trivially  finite. We now show that  the second series in the r.h.s. is absolutely convergent. Call $D= D(\o,\cK)$ the random set of points  $y\in \hat \o$ such that $\eta_y(s)=1$ for some $s\in [0,t]$. By the graphical construction $D$ is a finite set. 
We also note that  if $ | (\eta_z- \eta_y)(s-)| $  is nonzero then $y$ or $z$ must belong to $D$.
 Then, setting $u=t-s$, we have 
\[ \sum _{y,z \in \hat \o} p_\o(u ,x,y)  |(\eta_z- \eta_y)(s-)|c_{y,z}(\o) \leq 
\sum _{y\in D}\sum_{z\in \hat \o}    c_{y,z}(\o)+\sum _{y \in \hat \o}\sum_{z\in D}   c_{y,z}(\o)
\]
and the r.h.s.  is upper bounded by the finite constant $2 \sum_{y\in D}\l_0(\t_y \o)$. 
 Since $N_{y,z}(s)=N_{z,y}(s)$,  we can bound
% bounded by  $\mathds{1}_{y\in D} d N_{y,z} (s)+ \mathds{1}_{z\in D} d N_{z,y} (s)$. Hence
\be
\begin{split}
 \sum _{y \in \hat \o} \int _0^t p_\o(t-s,x,y) \sum_{z \in \hat\o} | (\eta_z- \eta_y)(s-)|  dN_{y,z}(s)\leq 2 \sum _{v\in D} N_v(t) <+\infty\,.
\end{split}
\en
This completes the proof that the r.h.s. of \eqref{razzo} is well defined.

We now verify \eqref{razzo} (the proof is different from the one in \cite{N}, which does not adapt well to our setting).   Let $0<\t_1<\t_2< \cdots$ be the jump times for the path $\eta(s)= \eta(s)[\o, \cK]$,  given $\eta(0)$.  Let $y_i,z_i$ be  such  that $N_{y_i,z_i}(\t_i)= N_{y_i,z_i}(\t_i-)+1$. 
 We fix   $t\in [\t_n , \t_{n+1})$. We set $t_0:=0$, $t_{n+1}:=t$ and $t_i:= \t_i$ for all $i=1,\dots, n$.
Then
\be\label{pollicino}
\sum _{y \in \hat \o} \int _0^t p_\o(t-s,x,y) dM_{y}  (s)=A_1-A_2\,,
\en
where
\begin{align*}
& A_1:=\sum_{ i=1}^n  c_i\,, \;\; c_i:=   \bigl[ p_\o(t-t_i,x,y_i )- p_\o(t-t_i,x,z_i )\bigr](\eta_{z_i}(t_{i-1}) -\eta_{y_i} (t_{i-1}) )
%+\sum_{ i=1}^n   p_\o(t-t_i,x,z_i )(\eta_{y_i}(t_i) -\eta_{z_i} (t_i) )   
\\
& A_2:= \sum_{i=0}^n \sum _{y \in \hat \o} \int _{t_i}^{t_{i+1}}  p_\o(t-s,x,y)  \sum_{z\in \hat\o} c_{y,z} \bigl( \eta_z(t_i)- \eta_y(t_i) \bigr) ds \,.
  \end{align*}
Recall the notation introduced before Theorem \ref{teo2}. We write $\bbE_x^\o$ for the expectation w.r.t. the law of  the random walk $X_t:= X^1_{\o,t}$ on $\hat \o$ with formal generator $\bbL_\o^1$, starting at $x$.
 Since $\bbL^1_\o f (y)=  \sum_{z\in \hat\o} c_{zy} \bigl( \eta_z(t_i)- \eta_y(t_i) \bigr)$ where $f: \hat \o\to \bbR$ is given by $f(a):= \eta_a(t_i)$, we have 
   \begin{equation*}
   \begin{split} A_2& =  \sum_{i=0}^n \int_{t- t_{i+1}}^{t-t_i} \frac{d}{ds} \bbE_{x}^\o [\eta_{X_s}(t_i)]=
    \sum_{i=0}^n  \left ( \bbE_{x}^\o [\eta_{X_{t-t_{i} }}(t_i)]-\bbE_{x}^\o [\eta_{X_{t-t_{i+1} }}(t_i)]  \right) \,.
   % &  \,.
   \end{split}
   \end{equation*}
Since, for $ 1\leq i \leq n$,  $\eta_a(t_{i})= \eta_a(t_{i-1})+  \d_{a,y_i} (\eta_{z_i} (t_{i-1}) - \eta_{y_i}(t_{i-1}) ) -\d_{a,z_i} (\eta_{z_i} (t_{i-1}) - \eta_{y_i}(t_{i-1}) )$, we have for  $ 1\leq i \leq n $ that 
\[
\bbE_{x}^\o [\eta_{X_{t-t_{i} }}(t_{i})]= \bbE_{x}^\o [\eta_{X_{t-t_{i} }}(t_{i-1})]+c_i\,.
\]
Hence we can write 
\begin{equation*}
\begin{split}A_2& = \bbE_{x}^\o [\eta_{X_{t}}(0)]+\sum_{i=1}^n (\bbE_{x}^\o [\eta_{X_{t-t_{i} }}(t_{i-1})]+c_i) -\sum_{i=0}^n  \bbE_{x}^\o [\eta_{X_{t-t_{i+1} }}(t_i)] \\
& = \bbE_{x}^\o [\eta_{X_{t}}(0)]-\bbE_{x}^\o [\eta_{X_{0}}(t_n)]+A_1= \sum_{y\in \hat \o} p_\o(t,x,y )\eta_y(0)-\eta_x(t_n)+A_1.
 \end{split}
 \end{equation*}The above identity and \eqref{pollicino} implies \eqref{razzo}. To justify    the  above manipulations of series,  we point out that $\eta(t_i) $ is uniformly bounded and it has only a finite number of nonzero entries and that $\l_0(\t_x \o)<+\infty$ for all $ x\in \hat \o$.

\medskip

We now denote by $ \bbE_{\o, \mathfrak{n}_\e} $ the expectation w.r.t. $ \bbP_{\o, \mathfrak{n}_\e} $.  Due to \eqref{razzo} and the symmetry $p_\o(t,x,y)= p_\o(t,y,x)$,  in order to conclude the proof of Lemma \ref{igro2} it is enough to show that 
\be%\label{vigorsolA}
\lim _{\e\da 0} \bbE_{\o, \mathfrak{n}_\e} \Big[ \Big(  \e^d \sum_{x\in \hat \o} \varphi (\e x)  \sum _{y \in \hat \o} \int _0^{ \e^{-2} t }  p_\o(\e^{-2} t-s,x,y) dM_{y}  (s)
\Big)^2\Big]=0\,.
\en
Due to \eqref{deprimo}, we can rewrite the expression inside  the  $(\cdot)$--brackets as 
\begin{multline*}
\cR_\e:=   \frac{\e^d}{2} \sum_{x\in \hat \o}  \sum _{y \in \hat \o}  \sum _{z \in \hat \o} 
 \varphi (\e x)  \cdot \\
 \int _0^{\e^{-2} t} ( \eta_z - \eta_y)(s-) \bigl( p_\o(\e^{-2} t-s,x,y) - p_\o(\e^{-2} t-s, x,z)\bigr) d A_{y,z} (s) \,. 
\end{multline*}
Hence, similarly to \cite{N}, we get  (using the symmetry of $p_\o (s, \cdot, \cdot)$)
\begin{equation*}
\begin{split}
\bbE_{\o, \mathfrak{n}_\e} \big[\cR_\e^2\big]&\leq \frac{\e^{2d}}{4}\sum _{y \in \hat \o  }  \sum _{z \in \hat \o}\int _0^{\e^{-2} t} c_{y,z}(\o)  \bigl( P^1_{\o, s} \varphi(\e y)- P^1_{\o, s} \varphi(\e z)\bigr)^2ds\\
&= \frac{\e^d}{2} \int _0^{ t} \la P^\e _{\o, s} \varphi, - \bbL_\o ^\e P^\e _{\o, s} \varphi \ra _{L^2 (\mu_\o ^\e)}= -\frac{\e^d}{4}  \int _0^{ t} \frac{d}{ds} \| P^\e _{\o, s} \varphi\|^2_{L^2(\mu_\o^\e)}ds\\
& =  \frac{\e^d}{4} \| P^\e _{\o, 0} \varphi\|^2_{L^2(\mu_\o^\e)}-\frac{ \e^d}{4} \| P^\e _{\o, t} \varphi\|^2_{L^2(\mu_\o^\e)}  \leq \frac{ \e^d}{4} \| \varphi \|^2_{L^2(\mu_\o^\e)}\stackrel{\e \to 0}{\longrightarrow}0\,.
\end{split}
\end{equation*}
\end{proof}

\appendix
%%%%%%%%%%%%%%%%%%%%%%%%%%%%%%%%%%%%%%%%%%%%
 
\section{Supplementary proofs}\label{dimo}

\subsection{Proof of Lemma \ref{re}}
For  Item (i)  see \cite[Exercise~12.4.2]{DV}. Item (ii) follows from stationarity and it is standard \cite{DV}.
We prove Item (iii).
Call  $A=\{\o \in \O_0\,:\, \t_x \o\not =\o \;\forall  x\in \hat\o\setminus\{0\}\}$.
Call $\tilde A:= \{\o\in \O\,:\, \t_z \o \in A\; \forall z \in \hat \o\}$ and observe that $\tilde A$ equals the  event appearing in (A3), i.e. $\tilde A= \{\o\,:\, \t_x \o \not = \t_y\o \text{ for all $x\not=y$ in $\hat \o$}\}$. Then  (A3) and \eqref{SS3} are equivalent due to Lemma \ref{matteo}.

We consider now Item (iv).  Let \eqref{spugna} be verified.
Calling $N_a:=\hat\o(a+[0,1]^d)$, it holds $F_*(\o)\leq \sum _{u\in \bbZ^d}\sum_{v\in \bbZ^d}  N_u N_v h(|u|) h(|u-v|) $.
If $x\in [0,1]^d$, then $F_*(\t_x \o)\leq \sum _{u\in \bbZ^d}\sum_{v\in \bbZ^d}  M_u M_v h(|u|) h(|u-v|) $, where $M_a:=\hat\o(a+[0,2]^d)$. By applying Campbell's identity   \eqref{campanello} with $f(x, \o) := \mathds{1}_{[0,1]^d}(x) F_*(\o)$, we get
\be\label{rollino}
\bbE_0[F_*]\leq m^{-1} \sum _{u\in \bbZ^d}\sum_{v\in \bbZ^d}  h(|u|) h(|u-v|) \bbE[ N_0  M_u M_v]\,.
\en
We use that $abc\leq C(a^3+b^3+c^3)$ for some $C>0$ and for any $a,b,c\geq 0$, we apply the first bound in \eqref{spugna} and stationarity to get that
 $\bbE[ N_0  ^3]$, $ \bbE[  M_u ^3]$, $ \bbE[  M_v^3]$ are uniformly bounded, and afterwards  we apply the second bound in \eqref{spugna} to conclude that the r.h.s. of \eqref{rollino} is bounded.
 
We prove Item (v)
for  Mott v.r.h. 
By \cite[Lemma 2]{FSS}, $\l_0 \in L^k(\cP_0)$ if and only if  $\bbE \bigl[ \hat \o ( [0,1]^d ) ^{k+1} \bigr]<+\infty$. The proof provided there remains true when substituting $\l_0$ by any function $f $  such that  $|f(\o) | \leq C \int d\hat \o (x) e^{-c|x|} $ with $C,c>0$. As $f$ we can take also $f=\l_1$ and $f=\l_2$. We therefore conclude that for Mott v.r.h. Assumption (A5) is equivalent to the bound $\bbE \bigl[ \hat \o  ([0,1]^d ) ^{3} \bigr]<+\infty$. The above bound implies  (A6)  due to Item (iv).
  The check of the other statements in Item (v) is trivial.

Finally, 
Item (vi) follows from \cite[Lemma 1--(ii)]{FSS}.

\subsection{Proof of Prop.~\ref{prop_ergodico}} It is enough to consider the case $g\geq 0$. As in   \cite[Cor.~7.2]{T},  given $g$ and $\varphi$ as in Prop.~\ref{prop_ergodico},  \eqref{limitone} 
 holds for $\cP$--a.a.    $\o\in \O$
 % mammina - cambia
(for a self-contained proof of a stronger result see  
%\cite[App.~B]{FV2}). %versione AIHP 
%\cite[App.~B]{FV2}). %versione AIHP 
 Lemma \ref{sorrisino}).   %versione arvix_2
In particular, we have  $\cP(\cA_{g, \varphi})=1$, where $\cA_{g, \varphi}:= 
\{ \o \in \O \,:\, \eqref{limitone} \text{ is fulfilled}\}$.
We define  $\cA[g ]:= \cap _{\varphi \in C_c(\bbR^d)}   \cA_{g, \varphi}$.
  We fix a countable subset $\cK\subset C_c(\bbR^d)$, dense in $C_c(\bbR^d)$ w.r.t. the uniform norm. 
%We recall that the space 
%$C_0(\bbR^d)$  of continuous functions vanishing at infinity endowed with the uniform norm has a dense countable subset $\cK$.
 For any $n\in \bbN$  we fix a continuous function $\psi_n$ with values in $[0,1]$ such that $\psi_n(x)=1$ for  $x\in [-n,n]^d$ and $\psi(x)=0$ for $x\not \in [-n-1,n+1]^d$. Given $\varphi\in C_c(\bbR^d)$  we let  $N$ be the smallest integer $n$ such that the support of $\varphi$ is contained in $[-n,n]^d$ and, fixed $\e>0$, we take 
  $h\in \cK$ with $\| \varphi - h \|_\infty \leq \d$.  Then 
\begin{equation}\label{ruskone} (h(x) -\d) \psi_N(x) \leq   \varphi (x) \leq (h(x) +\d) \psi_N(x) \,.
\end{equation}
Since
$\int  (h(x) \pm \d) \psi_N(x) dx = \int \varphi(x) dx + O(  \d N^d) $, \eqref{ruskone} and the fact that $g\geq 0$ imply that 
$\cA[g] = \cap _{ h \in \cK} \cap _{n=1}^\infty \left(  \cA_{g, h \psi_n} \cap \cA_{g, \psi_n}  \right)$.
Being a countable intersection of Borel sets with $\cP$--probability equal to  $1$, $\cA[g]$ is Borel and $\cP\bigl(\cA[g]\bigr)=1$.

It remains to show that $\t_y\o \in \cA[g]$ if $\o \in \cA[g]$ and $y\in \bbR^d$.  Fix $\varphi \in C_c (\bbR^d)$.
% By \eqref{the} and since $x\in \t_y\hat \o$ if and only if $x+y \in \hat \o$,
 We have
 \begin{equation}\label{carretto1}
 \int  d\mu_{\t_y \o}^\e  (x)  \varphi (x ) g(\t_{x/\e}  \t_y\o )
 %= \e^d \sum_{x\in \t_y\hat \o}  \varphi (\e x) g (\t_{x+y} \o) 
 =\e^d \sum_{a\in  \hat \o}  \varphi (\e a-\e y ) g (\t_{a} \o)   \,.
 \end{equation}
Given $\d>0$ we  take $\rho \in(0,1)$ such that the oscillation of $\varphi$ is bounded by $\d$ in any box with sides of length at most $\rho$.
We can suppose $\e$ small enough such that $|\e y| < \r$. Then we can bound
\begin{equation}\label{carretto2}
 \bigl(\varphi (\e a) - \d\bigr)\psi_{N+1}(\e a)  \leq \varphi (\e a-\e y )  \leq \bigl(\varphi (\e a) + \d\bigr)\psi_{N+1}(\e a) \,.
\end{equation}
As a byproduct of \eqref{carretto1} and \eqref{carretto2} we have 
\begin{multline}\label{zizzone}
    \int  d\mu_{  \o}^\e  (x)  \bigl( \varphi (x ) -\d \bigr) \psi_{N+1} (x) g(\t_{x/\e}  \o )
\leq  \int  d\mu_{\t_y \o}^\e  (x)  \varphi (x ) g(\t_{x/\e}  \t_y\o )
\\
\leq  \int  d\mu_{  \o}^\e  (x)  \bigl( \varphi (x ) +\d \bigr) \psi_{N+1} (x) g(\t_{x/\e}  \o )\,.
\end{multline}
%Since $\o \in \cA_g$ the first term converges as $\e\da 0$ to $m \int \varphi(x) dx \bbE_0[g]+ O(m  2^d (N+1)^d \d)$. The same holds for the third term. By the arbitrariness of $\d$ we conclud as e that the second term converges to $m \int \varphi(x) dx \bbE_0[g]$.
By taking  the limit $\e\da 0$ and  using that $\o \in \cA[g]$ to treat the first and third terms in \eqref{zizzone}, and afterwards taking the limit $\d\da0$, we get that the second term converges as $\e\da 0$ to $m \int \varphi(x) dx \bbE_0[g]$. This concludes the proof that $\t_y \o \in \cA[g]$.

\subsection{Proof of Lemma \ref{dis}}
Since   $H_{\o, \e}$ and  $H^1_{\o, \e} $ are isomorphic, it is enough to focus on  $H_{\o, \e}$. Take a sequence $(v_n, \nabla _\e v_n)$ in $H_{\o, \e}$ converging to $(v,g)$ in $L^2( \mu _\o ^\e) \times L^2(\nu_\o ^\e)$. Since  $\mu _\o ^\e$ is an atomic measure, we have that $v_n(\e x) \to v(\e x)$ for any $x\in \hat \o$. This implies that 
  $\nabla_\e v_n (\e x, z) \to \nabla_\e v (\e x , z)$ for any $x\in \hat \o$ and  $z\in \widehat{\t_x \o}$,  therefore $\nabla_\e v_n \to \nabla _\e v $ $\nu_\o^\e$--a.s.   On the other hand, since $\nabla_\e v_n \to g$ in  $L^2(\nu_\o ^\e)$, at cost to extract a subsequence we have that $\nabla_\e v_n\to g$ $\nu_\o^\e$--a.s. By the uniqueness of the a.s. limit it must be $g=\nabla_\e v$ $\nu_\o^\e$--a.s.

 \section{Supplementary proofs for the arxiv version}\label{aggiuntina} % mammina rimuovere
We collect in this section some proofs that can be obtained by  adapting proofs  present in other references, but given in a different context (and notation). This section appears only in this  arxiv version.

\subsection{Integration to  the proof of Prop. \ref{prop_ergodico}} We prove here that, given $g$ and  $\varphi$ as in Prop. \ref{prop_ergodico},  \eqref{limitone} holds $\cP$--a.s.  (cf. Lemma \ref{sorrisino}).

  Following \cite[Def.~10.2.I]{DV} a sequence $(A_n)$ of bounded Borel sets in $\bbR^d$ is called a \emph{convex averaging sequence} if 
 \begin{itemize}
 \item[(i)] each $A_n$ is convex;
 \item[(ii)] $A_n \subset A_{n+1}$ for $n=1,2,\dots$,
 \item[(iii)] $r(A_n)\to \infty$ as $n\to \infty$, where $r(A)$ is the supremum among $r\geq 0$ such that $A$ contains a ball of radius $r$.
 \end{itemize} 
  Since by Assumption (A1)  $\cP$ is ergodic  we 
 have the following ergodic result \cite[Prop.~12.2.VI]{DV}: given a nonnegative Borel function $g: \O_0\to [0,\infty)$ and given a convex averaging sequence $(A_n)$ it holds 
 \begin{equation}\label{lattone1}
 \lim_{n \to \infty} \frac{1}{ \ell (A_n) } \int_{A_n }  d\hat{\o}(x) \, g(\t_x \o) = m \,\bbE_0[ g ]\qquad \cP\text{--a.s.}\,,
 \end{equation}
where  $\ell(\cdot)$ denotes the Lebesgue measure. We point out that  \cite[Prop.~12.2.VI]{DV} refers to non--marked point processes, but it can be   generalized to  marked point processes.  % kuka

%%%%%%%%%%%%%%%%%%%%%%%%%%%%

\begin{Lemma}\label{asterix1}
 Let  $g: \O_0\to \bbR$ be a Borel   function  with $\bbE_0[g]<\infty$.
Then there exists a Borel set    $\cB_g\subset \O$ with $\cP(\cB_g)=1$ such that the following holds  for any $\o \in \cB_g$: if  $A=\prod_{i=1}^d (a_i, b_i]$ with $a_i < b_i $ and $a_i,b_i \in \bbQ$,  then 
\begin{equation}\label{lattone2}
\lim_{t\to +\infty} \frac{1}{ t^d } \int _{ t A} d \hat \o(x) g(\t_x \o) = m\, \ell(A) \bbE_0[g]\,.\end{equation}
\end{Lemma}
\begin{proof}
We give the proof for $d=2$, the general case is similar.  Without loss we can assume $g\geq 0$.
 We first take $A= \prod_{i=1}^2 (0, b_i]$ with $b_i \geq 0$. We have that $( n A)$ is a convex averaging sequence. In particular, by \eqref{lattone1}, the limit \eqref{lattone2} holds $\cP$--a.s. when restricting to integer numbers  $t$. Since $t\mapsto  \int _{ t A} d \hat \o(x) g(\t_x \o)$ is non--decreasing, it is immediate to recover \eqref{lattone2} for general $t$.

Trivially if \eqref{lattone2} holds $\cP$--a.s. for $A=A_1$ and for $A=A_2$ with $A_1\subset A_2$, then it holds $\cP$--a.s. for $A=A_2\setminus A_1$. As a consequence, if we have $0\leq a_i < b_i$,  then \eqref{lattone2} holds $\cP$ a.s. when $A$ is one of the following sets
\begin{align*}
&  (0,b_1]\times (a_2,b_2]= (0,b_1]\times (0,b_2]\,  \setminus \,
 (0,b_1]\times (0,a_2]\,, \\
 &  (0,a_1]\times (a_2,b_2]=  (0,a_1]\times (0,b_2]\, \setminus \, (0,a_1]\times (0,a_2]\,, \\
 & (a_1,b_1]\times (a_2,b_2]= (0,b_1]\times (a_2,b_2]\,\setminus\,   (0,a_1]\times (a_2,b_2]\,.
 \end{align*}
 Hence we have  proved that \eqref{lattone2} holds $\cP$--a.s. for $A= \prod_{i=1}^2 (a_i, b_i]$ if $0 \leq a_i \leq b_i$ for $i=1,2$. The same arguments allow to get that  \eqref{lattone2} holds $\cP$--a.s. when $A$ is any box $\prod_{i=1}^2 (a_i, b_i]$  contained in a quadrant (we say that $A$ is good).  
 Since any generic box  $A=\prod_{i=1}^2 (a_i, b_i]$  can be partitioned into four good boxes,  we get that \eqref{lattone2} holds $\cP$--a.s. for the box $A$. Since the boxes   $A=\prod_{i=1}^2(a_i, b_i]$ with rational $a_i,b_i$ are countable, we get the thesis.
\end{proof}

\begin{Lemma}\label{sorrisino}
The limit \eqref{limitone} holds  for any $\varphi\in C_c(\bbR^d)$ and $\o \in \cB_{\max\{g,0\}} \cap \cB_{\max\{-g,0\}}$ (cf. Lemma
\ref{asterix1}).
\end{Lemma}
\begin{proof} Without loss, we take $g \geq 0$.
Since $\varphi$ is uniformly continuous, given $\d>0$ we  take $\rho \in(0,1)$ such that the oscillation of $\varphi$ is bounded by $\d$ in any box with sides of length at most $\rho$. Let $N=N(\varphi)$ be the smallest integer $n$ such that the support of $\varphi$ is contained in $(-n,n]^d$. 
 We partition  $(-N,N]^d$ by  a finite family of disjoint boxes  $A_i$, $i=1,\dots, k$, with sides of  length at most  $\rho$, extremes in $\bbQ^d$ and of the  form  $\prod_{i=1}^d (a_i,b_i]$. Call $m_i := \inf_{A_i} \varphi$ and $M_i :=\sup_{A_i} \varphi$. Then, it holds
 %\[ \sum_{i=1} ^k m_i \mathds{1}_{A_i}(x)  \leq \varphi(x)  \leq \sum_{i=1}^k M_i \mathds{1}_{A_i}(x)\] and therefore
\[
\sum_{i =1}^k m_i \e ^d \int _{\frac{A_i }{\e}}  d \hat\o(x) g(\t_x\o) \leq \int_{\bbR^d} \mu_\o^\e  (x)  \varphi (x ) g(\t_{x/\e} \o )
\leq \sum_{i =1}^k M_i \e ^d \int _{\frac{A_i }{\e}}  d \hat\o(x) g(\t_x\o) \,.
\]
By applying Lemma \ref{asterix1}  we get that, for $\o \in \cB_g$,   the l.h.s. converges as $\e\da 0$ to 
$\bigl[\int _{\bbR^d} \varphi (x) m dx+ O(m \d (2N)^d) \bigr]  \cdot \bbE_0[g] $. The same holds for the r.h.s. By the arbitrariness of $\d$ we conclude that  \eqref{limitone} holds  for all $\o \in \cB_g$ and all $\varphi\in C_c(\bbR^d)$.
\end{proof}

\subsection{Proof of Lemma \ref{minerale}} 
%
%Given $f \in L^2(\cP_0)$ we consider the equation
%\begin{equation}\label{lorenzo}
%-{\rm div} \nabla u + u = f 
%\end{equation}
%in its weak form on the Hilbert space $H^1_{\rm env} $:
% an element $u \in H^1_{\rm env} $ is a weak solution of \eqref{lorenzo} if for any $v \in H^1_{\rm env} $  it holds
%\begin{equation}\label{lorenzo_esteso}
%\int d\nu  \, \nabla u \nabla v + \int  d\cP_0\, u v = \int d\cP_0\, f v \,.
%\end{equation}
%Since $H^1_{\rm env} $ is a Hilbert space, by the Lax--Milgram theorem, equation \eqref{lorenzo} has a unique solution $u \in H^1 _{\rm env} $.  
%
Let $u\in H^1_{\rm env}$ be the weak solution of  equation $-{\rm div} \nabla u + u=\z $ in  $H^1_{\rm env} $. This means that,
for any $v \in H^1_{\rm env} $,   it holds
\begin{equation}\label{lorenzo_esteso}
\int   \nabla u \nabla v  \, d\nu + \int   u v \,d\cP_0 = \int \z v  \,d\cP_0\,.
\end{equation}
Since $H^1_{\rm env} $ is a Hilbert space, by the Lax--Milgram theorem,  there exists a unique weak solution $u \in H^1 _{\rm env} $.  
%\begin{equation}\label{fungo}
%\int \nabla u \nabla v d\nu + \int u v d \cP_0 = \int \z v d \cP_0\,.
%\end{equation}
By taking $v:=u$  in \eqref{lorenzo_esteso} we get 
\begin{equation}\label{ialta} \int  |\nabla u| ^2 d\nu + \int  |u|^2  d\cP_0 = \int  \z u d \cP_0 \,.
\end{equation}
Hence it holds $\bbE_0 [u^2] \leq \bbE_0[u \z]$, which implies (by Schwarz  inequality) that 
$\bbE_0[u ^2] \leq \bbE_0[\z^2]$. We also observe that  \eqref{lorenzo_esteso}   can be written as 
\begin{equation}\label{fantasmi}  \int    (u-\z) v d\cP_0= -\int  \nabla u \nabla v d\nu \qquad \forall v \in H^1_{\rm env}\,.
\end{equation}
Take $v$ bounded and measurable. Then $v\in H^1_{\rm env} $. By Lemma \ref{ponte} we get 
\[
 -\int  \nabla u \nabla v d\nu=\int  {\rm div}\,(\nabla u) v  d\cP_0 \,.
 \]
 Hence, by \eqref{fantasmi}, for each $v$ bounded and measurable it must be $ \int (u-\z) v d\cP_0 = \int {\rm div}(\nabla u) \, v d\cP_0 $. This implies that  ${\rm div} (\nabla u)=u-\z \in L^2(\cP_0)$.
 Since 
 $u\in H^1_{\rm env}$,
% $\nabla u \in L^2(\nu)$, 
by the assumptions on $\z$ we get that $\bbE_0 [ (u-\z) \z] =0$. The last identity implies that  $\bbE_0 [ u \z] =\bbE_0[\z^2]$ and that 
$0\leq \bbE_0[ (\z-u)^2] = - \bbE_0[ u (\z-u) ]$. Hence we have  $\bbE_0 [ u \z] =\bbE_0[\z^2]$ and $\bbE_0[ u^2] \geq \bbE_0[ u \z]$. We conclude that  $\bbE_0[ u^2] \geq \bbE_0 [ u \z]= \bbE_0[\z^2]$. We  have already proved that $\bbE_0[ u^2] \leq \bbE_0[\z^2]$. Hence, we get that $\bbE_0[ u^2] = \bbE_0[\z^2]=\bbE_0 [ u \z] $.  Using this last identity in \eqref{ialta} we conclude that 
$\int |\nabla u| ^2 d\nu =0$.   By \eqref{fantasmi} and using that $\nabla u=0$ $\nu$--a.s.,  for any $v \in H^1_{\rm env}$  it holds
$ \int u v d \cP_0 = \int \z v d \cP_0$. By taking $v$ varying  among the bounded Borel functions, we conclude that $u=\z$ in $L^2(\cP_0)$. As $u=\z$ $\cP_0$--a.s. and $u\in H^1_{\rm env}$ we conclude that $\z=u$ in $ H^1_{\rm env}$ and $\| \nabla \z\|_{L^2(\nu)}= \| \nabla u \|_{L^2(\nu)}=0$.

\subsection{Proof of Lemma \ref{fantasia} (continuation)}
We need to prove \eqref{gingsen}  whenever  $u_\e \stackrel{2}{\rightharpoonup} u$. 
It is enough to show that, for any sequence $\e_n \da 0$, there exists a subsequence $\e_{k_n}$ such that \eqref{gingsen} holds for $\e$ varying in $\{\e_{k_n}\}$.
Since $\{u_\e\}$ and $\{v_\e\}$ are bounded families in $ L^2(\mu^\e _{\tilde \o})$, there exists $C>0$ such that $\int_{\bbR^d}d\mu^\e_{\tilde \o}(x)  v_\e(x) u_\e(x) $ and $\int_{\bbR^d} d\mu^\e_{\tilde \o}(x) u_\e(x)^2 $ are in $[-C,C]$  for $\e $ small enough. By compactness, there exists a subsequence $\e_{k_n}$ such  that 
\begin{equation}\label{kylo0}
\lim _{k \to\infty} \int   d\mu^{\e_k}
_{\tilde \o}(x) v_{\e_k}(x) u_{\e_k}(x)=\a\,, \qquad \lim _{k \to\infty} \int   d\mu^{\e_k}
_{\tilde \o}(x)  u_{\e_k}(x)^2 =\b\,,
\end{equation}
for suitable $\a, \b\in [-C,C]$.  Given $t\in \bbR$, since $v_\e+ t u_\e \stackrel{2}{\toup} v+t u $, by Lemma \ref{alba} we have 
\begin{equation}\label{kylo1}
\varlimsup _{k \to \infty} \int   d\mu^{\e_k}
_{\tilde \o}(x) \left( v_{\e_k}(x)+ t  u_{\e_k}(x)\right)^2 \geq \int  d\cP_0(\o) \int  dx \,m \, (v+t u) ^2(x,\o)\,.
\end{equation}
By expanding the square in the l.h.s.  and using  \eqref{orchino} and \eqref{kylo0}, we get
\[ 2 t \a + t^2 \b \geq 2 t  \int  d\cP_0(\o) \int  dx \,m v(x,\o) u(x, \o)+ t^2  \int   d\cP_0(\o) \int  dx \,m u(x, \o)^2\,.
\]
Dividing by $t$ and afterwards taking the limits $t\to 0^+$ and $t\to 0^-$, we get that $\a=  \int   d\cP_0(\o) \int dx \,m \,v(x,\o) u(x, \o)$, which corresponds to \eqref{gingsen}.

%%%%%%%%%%%%%%%%%%%%%%%%%%%%%%%%%%%%%%%%%%%%%%%%
%%%%%%%%%%%%%%%%%%%%%%%%%%%%%%%%%%%%%%%%%%%%%%%%%%%%%%%%%%%%%%%%%%%%%%%%%%%%%%%%%%%%%%%%%%%%%%%%%%%%%%%%%%%%%%%%%%%%%%%%%%%%%%%%%%%%%%%%%%%%%%%%%%%%%%%%%%%%%%%%%%%%%%%%%%%%%%%%%%%%%%%%%%%%%%%%

\bigskip

{\bf Acknowledgements}: I thank  Andrey Piatnitski for useful discussions.  I warmly thank Annibale Faggionato and Bruna Tecchio for their nice hospitality in Codroipo (Italy), where  part of this work has been completed.

\end{document}